\documentclass[12pt, letterpaper]{amsart}
\usepackage[utf8]{inputenc}
\usepackage[margin=90.0pt]{geometry}
\usepackage{amsmath,amsfonts,amssymb,amsthm}
\usepackage[normalem]{ulem}
\usepackage{url}
\usepackage{enumerate}
\usepackage{hyperref}
\usepackage{csquotes}
\usepackage{tikz}
\usepackage{verbatim}
\usepackage{amssymb, amsthm, enumerate}
\usepackage{amsmath}

\usepackage[skins]{tcolorbox}
\usepackage{lipsum}

\theoremstyle{plain}
\newtheorem{theorem}{Theorem}[section]
\newtheorem{lemma}[theorem]{Lemma}
\newtheorem{proposition}[theorem]{Proposition}
\newtheorem{corollary}[theorem]{Corollary}
\theoremstyle{remark}
\newtheorem*{remark}{Remark}
\theoremstyle{definition}
\newtheorem{definition}[theorem]{Definition}
\newtheorem{assumption}[theorem]{Assumption}

\tcolorboxenvironment{theorem}{blanker,before skip=10pt,after skip=10pt}
\tcolorboxenvironment{corollary}{blanker,before skip=10pt,after skip=10pt}
\tcolorboxenvironment{proposition}{blanker,before skip=10pt,after skip=10pt}
\tcolorboxenvironment{lemma}{blanker,before skip=10pt,after skip=10pt}
\tcolorboxenvironment{definition}{blanker,before skip=10pt,after skip=10pt}
\tcolorboxenvironment{assumption}{blanker,before skip=10pt,after skip=10pt}

\newcommand{\eps}{\varepsilon}
\title{Multiplicative functions in short intervals {II}}
\author{Kaisa Matom\"aki and Maksym Radziwi{\l}{\l}}

\begin{document}
\maketitle

\begin{abstract}  
  We determine the behavior of multiplicative functions vanishing at a positive
  proportion of prime numbers in almost all short intervals. Furthermore we
  quantify ``almost all'' with uniform power-saving upper bounds, that is,
  we save a power of the suitably normalized length of the
  interval regardless of how long or short the interval is.
  Such power-saving bounds are new even in the special case of the M\"obius
  function.

  These general results are motivated by several applications. First, we
  strengthen work of Hooley on sums of two squares by establishing an asymptotic
  for the number of integers that are sums of two squares in almost all short
  intervals. Previously only the order of magnitude was known. Secondly, we
  extend this result to general norm forms of an arbitrary number field $K$
  (sums of two squares are norm-forms of $\mathbb{Q}(i)$).
  Thirdly, Hooley determined the order of magnitude of the sum of $(s_{n + 1} -
  s_{n})^{\gamma}$ with $\gamma \in (1, 5/3)$ where $s_{1} < s_2 < \ldots$
  denote integers representable as sums of two squares. We establish a similar
  results with $\gamma \in (1, 3/2)$ and $s_n$ the sequence of integers
  representable as norm-forms of an arbitrary number field $K$. This is the first such result for a number field of degree greater than two. Assuming the Riemann Hypothesis for all Hecke $L$-functions
  we also show that $\gamma \in (1,2)$ is admissible. Fourthly, we improve on a recent result of Heath-Brown about gaps between $x^{\varepsilon}$-smooth numbers. More generally, we obtain results about gaps between multiplicative sequences. Finally our result is useful in other contexts aswell, for instance in our forthcoming work on Fourier uniformity (joint with Terence Tao, Joni Terav\"ainen and Tamar Ziegler). 
\end{abstract}
\setcounter{tocdepth}{1}
\tableofcontents

\section{Introduction}
\subsection{Special cases of results for multiplicative functions}
\label{ssec:intro}

Let $f : \mathbb{N} \rightarrow [-1, 1]$ be a multiplicative function\footnote{We focus on real-valued multiplicative functions for now, for the sake of exposition}. In our previous paper \cite{MainPaper} we have shown that ``short averages'' of $f$ are most of the time close to ``long averages'' of $f$, that is, for all $x \in [X, 2X]$ outside of a subset of cardinality $o(X)$ we have,
\begin{equation} \label{eq:mainxx}
\frac{1}{h} \sum_{x < n \leq x + h} f(n) - \frac{1}{X} \sum_{X < n \leq 2X} f(n) = o(1)
\end{equation}
provided that $h \rightarrow \infty$ with $X \rightarrow \infty$. Roughly speaking this result can be interpreted as saying that the way integers factorize in a typical interval $[x, x + h]$ is similar to the way integers factorize in the long interval $[X, 2X]$.

The importance of this result is that it allows $h$ to grow arbitrarily slowly with $X$. In particular even for $f$ equal to the M\"obius function \eqref{eq:mainxx} is not
implied by the Riemann Hypothesis. Nonetheless there are still a few drawbacks. First, if $f$ is lacunary, the mean-value of $|f|$ itself is $o(1)$ and \eqref{eq:mainxx} is trivial. This is the case for many interesting multiplicative functions, such as for example the indicator function of integers that are representable as sums of two squares. Secondly, while it is not possible to replace $o(1)$ in \eqref{eq:mainxx} by $h^{-c}$ for some $c > 0$ we can still hope to show that the exceptional set of $x \in [X, 2X]$ for which \eqref{eq:mainxx} does not hold is $\ll_{c} X h^{-c}$ for some $c > 0$\footnote{It is reasonable to conjecture that this exceptional set is $\ll_{A} X h^{-A}$ for any given $A > 0$, but this conjectures is far out of reach}. Thirdly, for many applications one requires \eqref{eq:mainxx} for complex valued multiplicative functions, and this requires a change in the main term. In this paper we address all these issues. 

We will say that a subset $\mathcal{N} \subset \mathbb{N}$ is multiplicative if for any $m,n \geq 1$ with $(m,n) = 1$ we have $m,n \in \mathcal{N}$ if and only if $m n \in \mathcal{N}$. For example one might think of $\mathcal{N} = \mathbb{N}$ or $\mathcal{N}$ equal to the set of integers that can be represented as sums of two squares. In general such a subset can be quite arbitrary and it is natural to require (from the point of view of sieve theory) that there exist a constant $\alpha > 0$ such that for all $2 \leq w \leq z$, 
\begin{equation} \label{eq:req}
\sum_{\substack{w < p \leq z \\ p \in \mathcal{N}}} \frac{1}{p} > \alpha \sum_{w < p \leq z} \frac{1}{p} - O \Big ( \frac{1}{\log w} \Big ). 
\end{equation}
If \eqref{eq:req} holds, then the density of $\mathcal{N}$ in $[1, X]$ is approximately given by 
$$
\delta(\mathcal{N}; X) := \prod_{\substack{p \leq X \\ p \not \in \mathcal{N}}} \Big ( 1 - \frac{1}{p} \Big ) .
$$
so that the average spacing between consecutive elements of $\mathcal{N} \cap [1, X]$ is $\delta(\mathcal{N}; X)^{-1}$. We are now ready to state our first main result.

\begin{corollary}\label{cor:main1}
  Let $\mathcal{N}$ be a multiplicative subset of $\mathbb{N}$.
  Let $f : \mathbb{N} \rightarrow [-1, 1]$ be a multiplicative function. 
  Suppose that \eqref{eq:req} holds for some $\alpha > 0$. 
  Then there exists a constant $\kappa := \kappa(\alpha) > 0$, such that, for all $\delta \in (0, 1/1000)$ and $2 \leq h_0 \leq X$, $$
  \Big | \frac{1}{h_0} \sum_{\substack{x < n \leq x + h_0 \delta(\mathcal{N}; X)^{-1} \\ n \in \mathcal{N}}} f(n) - \frac{1}{X \delta(\mathcal{N}; X)} \sum_{\substack{X < n \leq 2X \\ n \in \mathcal{N}}} f(n) \Big | < \delta 
  $$
  outside of a set of $x \in [X, 2X]$ of cardinality
  $
  \ll X h_0^{-\delta^{\kappa}}.
  $
  Moreover if $\mathcal{N} = \mathbb{N}$ and $\delta \geq (\log h_0)^{-1/300}$, then the exceptional set is bounded by $\ll X (h_0^{-\delta / 15} + X^{- \delta^4 / 10^{16}})$. 
\end{corollary}

In \cite{MainPaper} we obtained results only for $\mathcal{N} = \mathbb{N}$ and in that case the cardinality of our exceptional set was $\ll X h^{-c}$ for some $c > 0$ only for $h \leq \log^{\nu} X$ for some small $\nu > 0$. If $|f| \equiv 1$ then intervals of length $\delta(\mathcal{N}; X)^{-1}$ are the shortest intervals for which Corollary \ref{cor:main1} can hold. However if $|f|$ is not close to $1$ then it is possible to obtain meaningful results for shorter intervals. This is accomplished in the more technical Theorem \ref{th:MT} below. Furthermore Theorem \ref{th:MT} describes explicitely the exponent $\kappa$ appearing in Corollary \ref{cor:main1} and establishes results for complex valued multiplicative functions. On the other hand the special case $\mathcal{N} = \mathbb{N}$ follows immediately from Theorem \ref{th:MTDense} below. 

It would be possible to prove a variant of Theorem \ref{th:MT} for multiplicative functions $f$ such that $f(n) = O_\varepsilon(n^\varepsilon)$ for every $n$ and $\varepsilon > 0$ and $f(p^{k}) = O_{k}(1)$ for every prime $p$ and integer $k \geq 1$. We refrain from doing this in this paper. 


We can obtain much stronger results if we seek only the
order of magnitude and not asymptotics. 

\begin{corollary} \label{cor:HooleyGen}
  Let $\mathcal{N}$ be a multiplicative subset of $\mathbb{N}$.
  Suppose that \eqref{eq:req} holds for some $\alpha > 0$ and all $w \leq z \leq X^{\alpha}$. 
  \begin{enumerate}[(i)] \item
  Let $\varepsilon > 0$ be given. There exists a constant $\delta = \delta(\alpha, \varepsilon) > 0$ such that, for all $2 \leq h_0 \leq X$, the number of $x
\in [X, 2X]$ for which 
$$ \sum_{\substack{x < n \leq x + h_0
  \delta(\mathcal{N}; X)^{-1} \\ n \in \mathcal{N}}} 1 \leq
 \delta h_0
$$
is $\ll_{\alpha, \varepsilon} X h_0^{-1/2 + \varepsilon}$. 
\item Let $\gamma \in [1, 3/2)$ be given. If $1 \leq
n_1 < n_2 < \ldots$ is an enumeration of elements of $\mathcal{N}$, then
$$
\sum_{n_i \leq X} (n_{i + 1} - n_{i})^{\gamma} \asymp_{\alpha, \gamma} X \delta(\mathcal{N};X)^{1 -
  \gamma}.
$$
\end{enumerate}
\end{corollary}

We refer the reader to Theorem \ref{th:LowerBound} below for a stronger but more technical variant. 
The second part of Corollary \ref{cor:HooleyGen} is a simple consequence of the
first part. An important feature of Corollary \ref{cor:HooleyGen} is that the exponent in the
exceptional set does not shrink with $\alpha$ in \eqref{eq:req}. 


\subsection{Applications to smooth numbers}

Corollary \ref{cor:HooleyGen} has immediate consequences for smooth numbers. 

\begin{corollary} \label{cor:main4}
Let $\theta > 0$ be given.
\begin{enumerate}[(i)]
  \item Let $\varepsilon > 0$ be given. For all $2 \leq h \leq X$, the number of intervals $(x, x + h]$ with $x \in [X, 2X]$ 
that do not contain an $x^{\theta}$-smooth
number is $\ll_{\varepsilon, \theta} X h^{-1/2 + \varepsilon}$. 
\item Let $\gamma \in [1, 3/2)$ be given. Let $1 \leq n_1 < n_2 < \ldots$ denote the
sequence of integers $n$ such that all prime factors of $n$ are $\leq
n^{\theta}$. Then 
$$
\sum_{n_i \leq x} (n_{i + 1} - n_i)^{\gamma} \asymp_{\gamma, \theta} x.
$$
\end{enumerate}
\end{corollary}

 Part (i) improves on the result in \cite{MainPaper} where weaker bounds
 on the exceptional set are obtained. Moreover a minor modification of Corollary
 \ref{cor:main4} also improves the result on sign changes of multiplicative functions in \cite{MainPaper}. 
 Part (ii) improves on a result of Heath-Brown~\cite[Theorem
 2]{Heath-Brown18} who established the weaker upper bound $\ll_{\varepsilon} x^{1+\varepsilon}$ for
 any $\varepsilon > 0$. 

\subsection{Applications to norm-forms}

Specializing Corollary \ref{cor:main1} to the set
of integers representable as sums of two squares improves on a result of Hooley \cite{Hooley4} who showed that if $1
\leq s_1 < s_2 < \ldots$ is the sequence of integers representable as sums of
two
squares, then, for any $h_0 \to \infty$ with $X \to \infty$, one has
\begin{equation} \label{eq:hooley01}
\sum_{x < s_i \leq x + h_0 \sqrt{\log x}} 1 \asymp h_0
\end{equation}
for almost all $x \in [X, 2X]$. Moreover in an earlier paper Hooley \cite{Hooley71} established that
\begin{equation} \label{eq:hooley02}
\sum_{s_n \leq x} (s_{n + 1} - s_{n})^{\gamma} \asymp x (\log x)^{\tfrac 12
  (\gamma - 1)}
\end{equation}
for $\gamma \in [1, 5/3)$. The much more general Corollary~\ref{cor:HooleyGen} gives this only in the range $\gamma \in [1, 3/2)$ but in a forthcoming work we will establish~\eqref{eq:hooley02} for every $\gamma \in [1, 2)$.

Sums of two squares are norm-forms of $\mathbb{Q}(i)$ and it is natural to wonder
to what extent \eqref{eq:hooley01} and \eqref{eq:hooley02}
generalize to norm-forms of other number fields.
Recall that an integer $n$ is a norm-form of a number field $K$ if $n$ is equal to the norm of an algebraic integer in $K$. Alternatively the set of norm forms of $K$ corresponds to the image of the homogeneous polynomial $Q(x_1, \ldots, x_k) = N_{K / \mathbb{Q}}(x_1 \omega_1 + \ldots + x_k \omega_k)$ where $x_1, \ldots, x_k$ ranges over integers and $\omega_1, \ldots, \omega_k$ is a $\mathbb{Z}$-basis of the ring of algebraic integers of $K$.

Following Odoni \cite{Odoni75} the density in $[1,X]$ of norm-forms of an algebraic number field $K$ is  
\begin{equation} \label{eq:density}
\delta_K(X) := \prod_{\substack{p \leq X \\ p \neq N \mathfrak{a} \\ \mathfrak{a} \text{ integral ideal}}} \Big (1 -
\frac{1}{p} \Big )
\end{equation}
If $K$ is a normal extension of $\mathbb{Q}$ of degree $k$, then $\delta_K(X) \asymp (\log
X)^{-1 + 1/k}$. 

The main arithmetic input in Hooley's work on \eqref{eq:hooley01} and \eqref{eq:hooley02}
is a solution to the shifted convolution problem,
\begin{equation} \label{eq:scp}
\sum_{n \leq x} r_{K}(n) r_{K}(n + h)
\end{equation}
with $r_{K}(n)$ the coefficients of the Dedekind zeta function of $K = \mathbb{Q}(i)$. 
Estimating \eqref{eq:scp} is completely open as soon as the degree of $K$ exceeds two.
For this reason Hooley's approach does not generalize beyond quadratic fields.
Furthermore, when the class number of $K$ differs from one, being a norm-form is no
longer a multiplicative condition\footnote{If $K$ is not a principal ideal domain, then a positive proportion of the prime factors of a typical norm-form are not themselves norm-forms}. This presents additional difficulties. 
Nonetheless we obtain the following generalization of Hooley's result \eqref{eq:hooley01} to arbitrary number
fields. 

\begin{theorem} \label{thm:main2}
  Let $K$ be a number field over $\mathbb{Q}$. Let $\delta_K(X)$
  be defined as in \eqref{eq:density}.
  Let $1 \leq n_1 < n_2 < \ldots$ be the sequence
  of non-negative norm-forms of $K$. Then, as $X \rightarrow \infty$, uniformly in $2 \leq h_0 \leq X$ and $\delta \in (0, 1/1000)$,
  \begin{equation} \label{eq:new2}
  \Big | \sum_{x < n_k \leq x + h_0 \delta_K(X)^{-1}} 1 - C_K h_0 \Big | \leq \delta h_0
  \end{equation}
  for all $x \in [X, 2X]$ with at most $O(X h_0^{-c \delta^{\kappa}})$ exceptions where $C_K > 0$ and $c, \kappa > 0$ are
  three constants that depend solely on $K$. 
\end{theorem}

Furthermore we obtain the following generalization of Hooley's result
\eqref{eq:hooley02}
to arbitrary number fields.

\begin{theorem} \label{thm:NormFormLowBound}
Let $K$ be a number field over $\mathbb{Q}$ and let $\delta_K(X)$ be as in
\eqref{eq:density}. 
Let $1 \leq n_1 < n_2 < \ldots$ denote an enumeration of
positive norm-forms of $K$.
\begin{enumerate}[(i)]
\item  Let $\varepsilon > 0$ be given. There exists a constant $\delta = \delta(K, \varepsilon) > 0$ such that, for all $2 \leq h_0 \leq \delta_K(X) X$, the number of $x \in [X, 2X]$
for which
$$
\sum_{x < n_i \leq x + h_0 \delta_K(X)^{-1}} 1 \leq
\delta h_0
$$
is $\ll_{\varepsilon, K} X h_0^{-1/2 + \varepsilon}$.
Moreover if the Riemann Hypothesis holds for all Hecke $L$-functions then the exceptional set has size $\ll_{\varepsilon, K} X h_0^{-1 + \varepsilon}$.
\item Let $\gamma \in [1, 3/2)$ be given. Then
  \begin{equation} \label{eq:new}
  \sum_{n_i \leq x} (n_{i + 1} - n_{i})^{\gamma} \asymp_{\gamma, K} x \delta_K(X)^{1 - \gamma}.
  \end{equation}
  Moreover if the Riemann Hypothesis holds for all Hecke $L$-functions then the above holds for every $\gamma \in [1, 2)$. 
\end{enumerate} 
\end{theorem}

As we pointed out already, previously there was not a single tuple $(K, \gamma)$ with $K$ a number field of
degree $> 2$ and $\gamma > 1$ for which \eqref{eq:new} or \eqref{eq:new2} was known. We believe that
a remarkable feature of \eqref{eq:new} is that the exponent $\gamma$ does not
shrink when the degree $k = [K:\mathbb{Q}]$ increases. We note also that given
the current technology $\gamma \leq 2$ is the best exponent for which one can
hope. Incidentally note that \eqref{eq:new2} implies \eqref{eq:new} for
$\gamma < 1 + c$ with  some $c > 0$.

\subsection{Applications to Fourier Uniformity}

We also note that the power-saving for the exceptional set that we obtain for example in Corollary \ref{cor:main1} is an ingredient in our forthcoming work on Fourier Uniformity \cite{FourierUniformity}, in which we establish that, for any given $k \in \mathbb{N}$, and any multiplicative function $f$ that is not $\chi(n) n^{it}$ pretentious for some $|t| \leq X^{k + 1}$ and Dirichlet character $\chi$ with bounded conductor, 
$$
\int_{X}^{2X} \sup_{\substack{P(Y) \in \mathbb{R}[Y] \\ \text{deg } P = k}} \Big | \sum_{x < n \leq x + H} f(n) e(P(n)) \Big | dx = o (H X)
$$
as $X \rightarrow \infty$, uniformly in $\exp(\log^{5/8 + \varepsilon} X) \leq H \leq X^{1/2 - \varepsilon}$. The results of the present work come into play when we prove the theorem for small $H$, in particular when $H$ is below the threshold $\exp(\log^{2/3} X)$. The latter is a natural threshold because of the limitations of the Vinogradov-Korobov zero-free region.

\subsection{Precise results for multiplicative functions}
\label{ssec:MultFunc}

We are now ready to discuss the main theorems from which all of the previous
corollaries eventually follow. In order to obtain results for complex-valued multiplicative functions $f$ we introduce a parameter $t_{f,X}$ that roughly measure the ``complex part'' of $f$ in the sense that $f(n) n^{-i t_{f,X}}$ essentially behaves as a  real-valued function for $n \leq X$.

\begin{definition} \label{def:complex}
  Let $f : \mathbb{N} \to \mathbb{U} := \{ z\in \mathbb{C} : |z| \leq 1\}$ be a multiplicative function.
  We define
  \begin{equation} \label{eq:tfx}
  \widehat{M}(f; X) := \min_{|t| \leq X} \sum_{p \leq X} \frac{|f(p)| - \Re f(p) p^{-it}}{p}.
  \end{equation}
  and let $\widehat{t}_{f,X}$ be (one of) $t \in [-X, X]$ that attains the minimum. Similarly we define   
  \begin{equation} \label{eq:MfXdef}
  M(f; X) := \min_{|t| \leq X} \sum_{p \leq X} \frac{1 - \Re f(p) p^{-it}}{p}.
  \end{equation}
  and let $t_{f,X}$ be (one of) $t \in [-X, X]$ that attains the minimum.
  
  Moreover we will say that $f$ is almost real-valued if
  $$
  \sum_{\substack{ p \in \mathbb{P} \\ f(p) \not \in \mathbb{R}}} \frac{|f(p)|}{p} < \infty. 
  $$
\end{definition}

We note that for the theorems that we are about to state
the condition $|t| \leq X$ in Definition~\ref{def:complex} can be relaxed but not significantly: specifically the theorems remain true if we
require that $|t| \leq X / h^{1 - \varepsilon}$ for some $\varepsilon > 0$,
but become false if we require $|t| \leq X / h^{1 + \varepsilon}$. 

We first record the following direct improvement of the main theorem from \cite{MainPaper} which however does not yet
address the case of sparse multiplicative functions $f$. 

\begin{theorem} \label{th:MTDense}
Let $f: \mathbb{N} \rightarrow \mathbb{U}$ be a multiplicative function. Fix $\rho < \rho_1 := 1/3 -2/(3\pi)$. There exists a constant $C' > 1$ such that, for any $2 \leq h \leq X^{1/2}$ and $\delta \in (0, 1/1000)$, 
\begin{equation} \label{eq:mainresulttt}
\begin{split}
&\Big | \frac{1}{h} \sum_{x < n \leq x + h} f(n) - \frac{1}{h} \int_{x}^{x + h} u^{i t_{f,X}} du \cdot \frac{1}{X} \sum_{X < n \leq 2X} f(n) n^{-it_{f, X}} \Big | \\
&\qquad \qquad \leq \delta + C'\frac{\log \log h}{\log h} + \frac{1}{(\log X)^{\rho/36}}
\end{split}
\end{equation}
for all but at most
$$
\ll_\rho  X \Bigl( \frac{1}{h^{\delta/15}} +  \frac{1}{X^{\delta^4/10^{16}}}\Bigr)
$$
integers $x \in [X, 2X]$.
Moreover, if $f$ is almost real-valued, then the claim also holds with $t_{f,X}$ replaced by $0$.
\end{theorem}

We notice that the theorem would not be true if $t_{f, X}$ was the smallest real-number in $[-X / h^{1 + \varepsilon}, X / h^{1 + \varepsilon}]$ minimizing the expression inside the minimum in~\eqref{eq:MfXdef}. 

In order to extend Theorem \ref{th:MTDense} to multiplicative functions that vanish on many primes, we restrict our attention to a wide sub-class of multiplicative functions that we call $(\alpha,\Delta)$-non-vanishing. This is a weighted analogue of the condition \eqref{eq:req}. 

\begin{definition} \label{def:vanishing}
  Given $\alpha \in (0,1]$ and $\Delta \geq 1$, a multiplicative function $f : \mathbb{N} \rightarrow \mathbb{U}$ is said to be $(\alpha, \Delta)$-non-vanishing if, for all $2 \leq w \leq z \leq \Delta$, we have
  \begin{equation} \label{eq:defnonvani}
  \sum_{w < p \leq z} \frac{|f(p)|}{p} \geq \alpha \sum_{w < p \leq z} \frac{1}{p} - O\Big(\frac{1}{\log w}\Big),
  \end{equation}
  where the implied constant is understood to be fixed, and other constants are allowed to depend on it.
\end{definition}

Note that if $f : \mathbb{N} \rightarrow [0,1]$ is $(\alpha, X^{\theta})$-non-vanishing with $\alpha \in (0,1]$ and $\theta > 0$, then,
\begin{equation} \label{eq:fundm}
\sum_{n \leq X} f(n) \asymp X \prod_{p \leq X} \Big ( 1 + \frac{f(p) - 1}{p} \Big )
\end{equation}
with the implicit constant in $\asymp$ depending only on $\alpha, \theta$ and the implied constant in \eqref{eq:defnonvani}\footnote{see Lemma~\ref{le:GrKoMa}(iii) below for a stronger version of this claim}. One can think of \eqref{eq:fundm} as a weighted version of the fundamental lemma of sieve theory.
Finally, given a multiplicative function $g : \mathbb{N} \rightarrow \mathbb{U}$, define
$$
H(g; X) := \prod_{p \leq X} \Big (1 + \frac{(|g(p)| - 1)^2}{p} \Big )
$$
and given $\alpha \in (0, 1]$ set,
  \begin{equation} \label{eq:etaadef}
  \rho_{\alpha} := \frac{\alpha}{3} - \frac{2}{3\pi} \sin \Big ( \frac{\pi \alpha}{2} \Big ) > 0.
  \end{equation}
We are now ready to state our main theorem.

\begin{theorem} \label{th:MT}
  Let $\alpha \in (0, 1], \theta \in (0, 1/16]$ and $0 < \rho < \rho_\alpha$.
  Let $f : \mathbb{N} \rightarrow \mathbb{U}$ be an $(\alpha, X^{\theta})$-non-vanishing multiplicative function.
  Given $h_0 \in [2, X^{\theta}]$ set $h := h_0 H(f; X)$. There exists a constant $C' > 1$ depending only on $\theta$ such that, for any $\delta \in (0, 1/1000)$,
      \begin{equation} 
      \label{eq:ress}
      \begin{split}
\Big |\frac{1}{h} & \sum_{x < n \leq x + h} f(n) - \frac{1}{h} \int_{x}^{x+ h} u^{i \widehat{t}_{f, X}} du \cdot \frac{1}{X} \sum_{X < n \leq 2X} f(n) n^{-i\widehat{t}_{f, X}} \Big | \\
& \leq \Big (\delta + C'\Big(\frac{\log \log h_0}{\log h_0}\Big)^{\alpha} + \frac{1}{(\log X)^{\alpha \rho / 36}} \Big) \prod_{p \leq X}\Big(1+\frac{|f(p)|-1}{p}\Big )
      \end{split}
      \end{equation}
for all but at most
$$
  \ll_{\rho, \theta} X\Bigl(\frac{1}{h_0^{(\delta / 2000)^{1 / \alpha}}} + \frac{1}{X^{\theta^3 (\delta / 2000)^{6 / \alpha}}}\Bigr)
$$
integers $x \in [X, 2X]$. Moreover, if $f$ is almost real-valued, then the claim also holds with $\widehat{t}_{f, X}$ replaced by $0$.
\end{theorem}

A few features of this theorem deserve further comment.

First, one might believe that the shortest intervals for which Theorem \ref{th:MT} should hold are of length
\begin{equation} \label{eq:naive} 
\prod_{p \leq X} \Big ( 1 - \frac{|f(p)| - 1}{p} \Big )
\end{equation}
since this is the inverse of the mean-value of $|f|$. However this is larger than $H(f; X)$ unless $|f|$ is concentrated in $\{0,1\}$ ! Therefore for functions $f$ such that $|f|$ is not concentrated in $\{0,1\}$ we obtain a result in intervals shorter than one would naively expect. 

In fact, it should be possible to obtain non-trivial results in intervals of length shorter than $H(f;X)$. The example $f(n) = (1 - \varepsilon)^{\Omega(n)}$ is in this respect instructive. It can be shown that the main contribution to the mean value of $f$ comes from integers in $[1,X]$ having $(1 - \varepsilon + o(1)) \log\log X$ prime factors. Since the mean-spacing of integers $n \leq X$ with $\Omega(n) = (1 - \varepsilon + o(1)) \log\log X$ is $(\log X)^{\varepsilon^2 / 2 + O(\varepsilon^3)}$ we expect that Theorem \ref{th:MT} should hold on intervals of length $(\log X)^{\varepsilon^2 / 2 + O(\varepsilon^3)}$. For comparison Theorem \ref{th:MT} gives results in intervals of length $\gg H(f; X) = (\log X)^{\varepsilon^2}$ while the ``naive length'' \eqref{eq:naive} is $(\log X)^{\varepsilon}$. In this example it may actually be possible to reach shorter intervals than we do here by restricting first to integers that have $(1 - \varepsilon + o(1)) \log\log X$ prime factors (see \cite{Correlations2} for such a manoeuver at work and \cite{Goudout, Teravainen} for results on the distribution of integers $n$ with $\omega(n) = k$ in almost all short intervals). 

Secondly, in many applications we work with multiplicative functions for which $\sum_{X < n \leq 2X} f(n)n^{-i\widehat{t}_{f, X}}$ is small compared to $\sum_{X < n \leq 2X} |f(n)|$. In that case the main term in Theorem \ref{th:MT} can be removed. A simple sufficient condition for this to happen is that $\widehat{M}(f; X)$ is large. We provide for convenience the corollary below. 

\begin{corollary} \label{cor:convenient}
\begin{enumerate}[(i)]
\item Suppose that the assumptions of Theorem \ref{th:MT} hold. Then
\[
\begin{split}
  &\Big |\frac{1}{h}  \sum_{x < n \leq x + h} f(n) \Big | \\
  &\leq \Big ( \delta + C' \Big (\frac{\log\log h_0}{\log h_0} \Big )^{\alpha} + C' \frac{\widehat{M}(f; X)}{\alpha \exp(\widehat{M}(f; X))} + \frac{1}{\alpha (\log X)^{\alpha \rho/36}} \Big ) \prod_{p \leq X} \Big ( 1 + \frac{|f(p)| - 1}{p} \Big )  
\end{split}
\]
  for all but at most
  $$
  \ll_{\rho, \theta} X\Bigl(\frac{1}{h_0^{(\delta / 2000)^{1 / \alpha}}} + \frac{1}{X^{\theta^3 (\delta / 2000)^{6 / \alpha}}}\Bigr)
  $$
  integers $x \in [X, 2X]$. 
\item Suppose that the assumptions of Theorem~\ref{th:MTDense} hold. Then
\[
\Big | \frac{1}{h} \sum_{x < n \leq x + h} f(n)\Big | \leq \delta + C'\frac{\log \log h}{\log h} + C' \frac{M(f; X)}{\exp(M(f; X))} + \frac{1}{(\log X)^{\rho/36}}
\]
for all but at most
$$
\ll_\rho  X \Bigl( \frac{1}{h^{\delta/15}} +  \frac{1}{X^{\delta^4/10^{16}}}\Bigr)
$$
integers $x \in [X, 2X]$.
\end{enumerate}
  \end{corollary}
Investigating the proofs rather than directly applying Theorems~\ref{th:MTDense} and~\ref{th:MT} one could obtain better $\log$-powers for this corollary.

We also have the following ``weighted'' analogue of Corollary \ref{cor:HooleyGen}.

\begin{theorem} \label{th:LowerBound}
  Let $\alpha \in (0, 1]$, $\theta \in (0, 1/9)$ and $\varepsilon > 0$. Let $f : \mathbb{N} \rightarrow [0,1]$ be an $(\alpha, X^{\theta})$-non-vanishing multiplicative function. Let $2 \leq h_0 \leq X/H(f; X)$. Then, there exists a positive constant $\delta = \delta(\alpha, \theta, \varepsilon)$ such that the the number of $x \in [X, 2X]$ for which 
    $$
    \frac{1}{h_0 H(f; X)} \sum_{x < n \leq x + h_0 H(f; X)} f(n) \leq \frac{\delta}{X} \sum_{X < n \leq 2X} f(n)
    $$
    is
    $$
    \ll_{\alpha, \varepsilon, \theta} X h_0^{-1/2 + \varepsilon}.
    $$
\end{theorem}

Finally, in principle it would be possible to extend both Theorem \ref{th:MT} and Theorem \ref{th:LowerBound} to unbounded multiplicative functions $f$ satisfying $f(n) = O_\varepsilon(n^\varepsilon)$ for every $n$ and $\varepsilon > 0$ and $f(p^{k}) = O_{k}(1)$ for every prime $p$ and integer $k \geq 1$ though one might need to change the definition of $H(f; X)$ somewhat. To obtain this extension requires one to follow through the proofs of Theorem \ref{th:MT} and Theorem \ref{th:LowerBound} paying close attention to when the bound $|f| \leq 1$ is being used. Moreover this extension requires one to use slightly more general sieve weights than the ones used in section \ref{se:sieve} and to generalize the results of Section \ref{se:halasz}. 
We do not do this extension in this paper, since considering these new cases would increase the length and complexity of this paper.

We note that establishing Theorem \ref{th:MT} and Theorem \ref{th:LowerBound} with optimally short intervals for general $f$ is likely to be a difficult question requiring new ideas. At the moment we do not have even a good conjectural understanding of what this shortest length should be, and the situation is well understood essentially only when $f$ is such that $|f(p)| \in \{0,1\}$ for all primes $p$. 

\section{Outline of the argument}

The notation employed throughout this outline might differ from the notation
employed in the proofs. We therefore caution the reader to take this outline
merely as a quick indication of the new and interesting points of our proof.

\subsection{The proof of Theorem \ref{th:MT}}

To make the presentation simpler we will ignore some complications
and focus only on the proof of Corollary \ref{cor:main1} with an exceptional set $O_\delta(X h_0^{-\kappa})$ for some $\kappa = \kappa(\delta)$. This is a special case of Theorem \ref{th:MT}
but already highlights all the main ideas. Furthermore we will assume that the multiplicative
function $f$ has mean-value zero and in fact does not pretend to be $n^{it}$ for any $|t| \leq 4 X$ so that in particular
\[
\sup_{|t| \leq 4X} \left|\sum_{X < n \leq 2X} f(n)n^{-it}\right| = o (\delta(\mathcal{N}; X)).
\] 
This  eliminates minor difficulties related to handling the main terms. 

Let $\varepsilon > 0$ be given. Let $\mathcal{S}_{\varepsilon}$ denote
the set of integers $n \in [X, 2X]$ that have a prime factor in intervals $(P_1, Q_1] \subset (1, h_0]$
and $(P_j, Q_j]$ for $2 \leq j \leq J$, where 
$$
P_j = \exp(j^{8j/\alpha^2} (\log Q_1)^{j - 1} \log P_1) \ , \ Q_j = \exp(j^{(8j+6)/\alpha^2}(\log Q_1)^j)
$$
and $J$ is the largest index for which $Q_J \leq \exp(\sqrt{\log X})$, say.
Moreover (as a novelty compared to \cite{MainPaper}) we require that integers in $\mathcal{S}_{\varepsilon}$ have two large
prime factors, say in
$$
(X^{\varepsilon^3}, X^{\varepsilon^2}] \text{ and } (X^{\varepsilon^2}, X^{\varepsilon}]
$$
respectively. The majority of integers $n \in [X, 2X]$ belongs to $\mathcal{S}_\varepsilon$
provided that $\varepsilon$ and $\log P_1/\log Q_1$ are sufficiently small.

Write $H =  h_0 \delta(\mathcal{N};X)^{-1}$.We start by noticing that if
$$
\Big | \sum_{\substack{x < n \leq x + H \\ n \in \mathcal{N}}} f(n) \Big | > \delta h_0
$$
then either
\begin{equation} \label{eq:alternative}
\Big | \sum_{\substack{x < n \leq x + H \\ n \in \mathcal{S}_{\varepsilon}\cap \mathcal{N}}} f(n) \Big | > \frac{\delta h_0}{2} \quad \text{ or } \quad \Big | \sum_{\substack{x < n \leq x +H \\ n \in \mathcal{S}_{\varepsilon}^{c} \cap \mathcal{N}}} f(n) \Big | > \frac{\delta h_0}{2},
\end{equation}
where $\mathcal{S}_{\varepsilon}^{c}$ denotes the complement of $\mathcal{S}_{\varepsilon}$ in $[X, 2X]$. 
To handle the contribution of the second term we use a sharp sieve upper bound, bounding, 
$$
\mathbf{1}_{n \in \mathcal{S}_{\varepsilon}^{c} \cap \mathcal{N}} \leq \sum_{\substack{d \leq X^{\theta} \\ d | n}} \lambda_d
$$
for some small $\theta > 0$. We construct these sieve majorants using the Brun-Hooley sieve (which is more clearly useful when proving the full Theorem \ref{th:MT}, in fact our sieve majorants give an alternative to the construction used by Matthiesen \cite{Matthiesen}). Subsequently we use the work of Friedlander (see \cite[Chapter 6.10]{Opera}) to show that, after the application of these sieve majorants, the exceptional set of $x \in [X, 2X]$ for which the second inequality in \eqref{eq:alternative} holds is bounded by $\ll_{\eta, \varepsilon'} X h^{-1 + \varepsilon'}$. This is the optimal bound given the $L^2$ techniques that are currently available (in reality one would expect that the exceptional set is $\ll_{A} X h^{-A}$ for any given $A > 0$ but we have no idea how to prove this).

Therefore it remains to address the case in which the first inequality in \eqref{eq:alternative} holds. By replacing $f(n)$ by $f(n) 1_{n \in \mathcal{N}}$, we can assume that $f$ is supported on $\mathcal{N}$.
By Perron's formula, we can more or less write
\begin{equation} \label{eq:integral1}
\begin{split}
\sum_{\substack{x < n \leq x + H \\ n \in \mathcal{S}_{\varepsilon} \cap \mathcal{N}}} f(n)& \approx \frac{1}{2\pi i} \int_{- X / H}^{X / H} \sum_{\substack{n \sim X \\ n \in \mathcal{S}_{\varepsilon}}} \frac{f(n)}{n^{1 + it}} \frac{(x+H)^{1+it} - x^{1+it}}{1+it} dt \\
&\approx \frac{H}{2\pi i} \int_{- X / H}^{X / H} \sum_{\substack{n \sim X \\ n \in \mathcal{S}_{\varepsilon}}} \frac{f(n)}{n^{1 + it}} x^{it} dt.
\end{split},
\end{equation}
If one estimated the right hand side by adding absolute values, then, even with square-root cancellation, one would only obtain a bound like $O(X^{1/2})$ which is much worse than the trivial bound $O(H)$ for the left hand side (unless $h_0$ is very large). 

Hence one must take advantage of the averaging over $x$ in the problem. Typically (for example in~\cite{MainPaper}) one studies the mean square of~\eqref{eq:integral1} over $x$, obtaining something like
\begin{equation}
\label{eq:MeanSquare}
\begin{split}
&\frac{1}{X} \int_X^{2X} \Big|\frac{1}{H} \sum_{\substack{x < n \leq x + H \\ n \in \mathcal{S}_{\varepsilon}}} f(n) \Big|^2 dx \\
& \approx \frac{1}{X} \int_{X}^{2X} \Big |\int_{-X / H}^{X/H} \Big( \sum_{\substack{n \sim X \\ n \in \mathcal{S}_{\varepsilon}}} \frac{f(n)}{n^{1 + it}} \Big ) \cdot x^{it} dt \Big |^2 dx \approx \int_{-X / H}^{X/H} \Big | \sum_{\substack{n \sim X \\ n \in \mathcal{S}_{\varepsilon}}} \frac{f(n)}{n^{1 + it}} \Big |^2 dt.
\end{split}
\end{equation}
To obtain the claim one would need the bound $O_\delta(\delta(\mathcal{N}; X)^2 h_0^{-\kappa})$ for this. However, this is not true in general --- there might be points $t$, where the integrand has size like $\delta(\mathcal{N}; X)^2 (\log X)^{-\nu}$. 

Our key new idea is to handle these ``bad'' points $t$ before taking the mean square. This is where the additional requirement (which is new compared to \cite{MainPaper}) that $\mathcal{S}_{\varepsilon}$ consists of integers having two large prime factors comes into play. Using the fact that integers in $\mathcal{S}_{\varepsilon}$ have two large prime factors we can more or less write, 
\begin{equation} \label{eq:integral2}
\begin{split}
\sum_{\substack{x < n \leq x + H \\ n \in \mathcal{S}_{\varepsilon} \cap \mathcal{N}}} f(n) & \approx \frac{H}{2\pi i} \int_{- X / H}^{X / H} \sum_{\substack{n \sim X \\ n \in \mathcal{S}_{\varepsilon}}} \frac{f(n)}{n^{1 + it}} x^{it} dt \\ & \approx \sum_{\substack{P_1 \in (X^{\varepsilon^3}, X^{\varepsilon^2}] \\ P_2 \in (X^{\varepsilon^2}, X^{\varepsilon}]}} \frac{H}{2\pi i} \int_{- X / H}^{X / H} \Big ( \sum_{p \sim P_1} \frac{f(p)}{p^{1 + it}} \Big ) \Big ( \sum_{p \sim P_2} \frac{f(p)}{p^{1 + it}} \Big ) \Big ( \sum_{\substack{n \sim X / (P_1 P_2) \\ n \in \mathcal{S}'_{\varepsilon}}} \frac{f(n)}{n^{1 + it}} \Big ) x^{it} dt,
\end{split}
\end{equation}
where $\mathcal{S}_{\varepsilon}'$ is the set of integers having a prime factor in each interval $(P_j, Q_j]$ with $1 \leq j \leq J$.  

The advantage that the introduction of these two large prime factors confers is that it essentially allows us to remove from the integration range, at the price of a negligible error term, any subset $\mathcal{T} \subset [-X/H, X/H]$ of measure $\ll X^{1/2 - \varepsilon}$ (more precisely: any set $\mathcal{T}$ that can be covered by a union of $\ll X^{1/2 - \varepsilon}$ unit intervals).

The reason for this is the following : Given an arbitrary subset $\mathcal{T} \subset [-X / H, X / H]$ that can be covered by $\ll X^{1/2 - \varepsilon}$ unit intervals we partition $\mathcal{T} = \mathcal{T}_1 \cup \mathcal{T}_2$, where
\begin{equation} \label{eq:smallll}
\mathcal{T}_1 := \left\{t \in \mathcal{T} \colon \Big | \sum_{p \sim P_1} \frac{f(p)}{p^{1 + it}} \Big | \leq P_1^{-\varepsilon^6}\right\}
\end{equation}
and $\mathcal{T}_2 = \mathcal{T} \setminus \mathcal{T}_1.$ The contribution of $t \in \mathcal{T}_1$ to the right hand side of~\eqref{eq:MeanSquare} is negligible using \eqref{eq:integral2}, the definition of $\mathcal{T}_1$, the Hal\'asz-Montgomery inequality (Lemma~\ref{le:Hallargevalint} below), and the assumption on the cardinality of $\mathcal{T}$. On the other hand $\mathcal{T}_2$ is a very small set of cardinality $\ll X^{3\varepsilon^6}$.  
We can bound the total contribution of $t \in \mathcal{T}_2$ to the integral \eqref{eq:integral2} by
$$
H \sum_{\substack{P_1 \in [X^{\varepsilon^3}, X^{\varepsilon^2}] \\ P_2 \in [X^{\varepsilon^2}, X^{\varepsilon}]}} \Big ( \sup_{|t| \leq 4 X} \Big | \sum_{\substack{n \sim X / (P_1 P_2) \\ n \in \mathcal{S}_{\varepsilon}' \cap \mathcal{N}}} \frac{f(n)}{n^{1 + it}} \Big | \Big ) \cdot \Big ( \sum_{t \in \mathcal{T}_2'} \Big | \sum_{\substack{p \sim P_1}} \frac{f(p)}{p^{1 + it}} \Big | \cdot \Big | \sum_{p \sim P_2} \frac{f(p)}{p^{1 + it}} \Big | \Big ) 
$$
where $\mathcal{T}_2'$ is a set of $\ll X^{3\varepsilon^6}$ one-spaced points. 
Since we assumed in this sketch that $f$ does not pretend to be $n^{it}$ for any $|t| \leq 4X$ we obtain a little bit of cancellation in the sum over $n \sim X / (P_1 P_2)$ (in case $f$ is pretentious we would need to handle a main term separately). As a result it remains to show that
\begin{equation} \label{eq:prer}
\sum_{t \in \mathcal{T}_2'} \Big | \sum_{p \sim P_1} \frac{f(p)}{p^{1 + it}} \Big | \cdot \Big | \sum_{p \sim P_2} \frac{f(p)}{p^{1 + it}} \Big | \ll \frac{1}{\log P_1 \log P_2}. 
\end{equation}
In other words we need to show that there is essentially at most one term $t \in \mathcal{T}'_2$ at which there is no cancellation, and this term dominates the whole sum. We establish this by applying Cauchy-Schwarz and using a sieved variant of Hal{\'a}sz-Montgomery inequality that relies on Vinogradov's bounds for exponential sums (see Lemma~\ref{lem:primeshalasz} below). We need to appeal to Vinogradov's bounds because the Dirichlet polynomials over primes in \eqref{eq:prer} are short, of length as short as $X^{\varepsilon^3}$. 

As a result of this operation we can take away an arbitrary set $\mathcal{T}$ from \eqref{eq:integral1} as long as this set consists of no more than $X^{1/2 - \varepsilon}$ neighborhoods of well-spaced points. In particular, in view of \eqref{eq:integral1}, to prove the corollary it suffices to show that,
\begin{equation} \label{eq:parse}
\frac{1}{X} \int_{X}^{2X} \Big |\int_{\substack{|t| \leq X / H \\ t \not \in \mathcal{T}}} \Big ( \sum_{\substack{n \sim X \\ n \in \mathcal{S}_{\varepsilon} \cap \mathcal{N}}} \frac{f(n)}{n^{1 + it}} \Big ) \cdot x^{it} dt \Big |^2 \asymp \int_{\substack{|t| \leq X / H \\ t \not \in \mathcal{T}}} \Big | \sum_{\substack{n \sim X \\ n \in \mathcal{S}_{\varepsilon} \cap \mathcal{N}}} \frac{f(n)}{n^{1 + it}} \Big |^2 dt \ll \frac{\delta(\mathcal{N}; X)^2}{h_0^{\kappa}}
\end{equation}
  for some $\kappa > 0$ and an essentially arbitrary $\mathcal{T}$ of our choosing, as long as $\mathcal{T}$ is not too large in measure.

  We choose now $\mathcal{T}$ to be the set of points $t$ at which at least one of the Dirichlet polynomials
  \begin{equation} \label{eq:reqw}
  \sum_{p \sim P} \frac{f(p)}{p^{1 + it}} \ , \ X^{\varepsilon^3} \leq P \leq X^{\varepsilon}
  \end{equation}
  with $P$ varying over powers of two in $[X^{\varepsilon^3}, X^{\varepsilon}]$ is 
  $\gg P^{-1/4+\varepsilon}$. Taking moments we see that $\mathcal{T}$ can indeed be covered by neighborhoods of fewer than $X^{1/2 - \varepsilon}$ points. With this choice of $\mathcal{T}$ we now repeat (with some minor technical innovations) the argument from our earlier paper \cite{MainPaper} to bound \eqref{eq:parse}. This produces a bound for \eqref{eq:parse} that saves $P_1^{-1/2 + \varepsilon} + X^{-\varepsilon^3 \cdot (1/2 - \varepsilon)}$.


  Two things are important to note. First, the bound $P_1^{-1/2 + \varepsilon} + X^{- \varepsilon^3 \cdot (1/2 - \varepsilon)}$ that the argument of \cite{MainPaper} produces is tied to our choice of the set $\mathcal{T}$ as the set at which the Dirichlet polynomial in \eqref{eq:reqw} is greater than $P^{-1/4 + \varepsilon}$. More precisely if we had choosen $\mathcal{T}$ as the set of $t$ on which the Dirichlet polynomial in \eqref{eq:reqw} is greater than $P^{-\beta}$ for some $\beta > 0$ then the argument from \cite{MainPaper} can only produce a bound that is $P_1^{- 2 \beta} + X^{- 2 \varepsilon^3 \beta}$ at best. The saving of $P_1^{-1/2 + \varepsilon}$ with an exponent $\tfrac 12$ is important for the proof of our second result, Theorem \ref{th:LowerBound}. Secondly, unconditionally the point-wise bounds for \eqref{eq:reqw} are weak (or unavailable for \eqref{eq:reqw} if $f$ is arbitrary), but since we took the set $\mathcal{T}$ away from the integral in \eqref{eq:parse}, it is as if we had pointwise power-savings in \eqref{eq:reqw}. This accounts for the term $X^{-\varepsilon^3 \cdot (1 /2 - \varepsilon)}$ in the bound $P_1^{-1/2 + \varepsilon} + X^{-\varepsilon^3 \cdot (1/2 - \varepsilon)}$. As a result the arguments from \cite{MainPaper} produce a saving in \eqref{eq:parse} that is of the form $P_1^{-1/2 + \varepsilon} + X^{-\varepsilon^3(1/2-\varepsilon)}$. Choosing $P_1$ to be $h^{\varepsilon^3}$ and $Q_{1} = h$ then gives the desired bound.
  
While running the argument from \cite{MainPaper} we need to also introduce an additional modification. Specifically, we need to keep track of the fact that $n$ is supported on a sparse set of integers. We accomplish this by using throughout a mean-value theorem that incorporates sieve estimates for $\mathcal{N}$. 

\subsection{The proof of Theorem \ref{th:LowerBound}}

Again for simplicity we will only discuss the proof of Corollary \ref{cor:HooleyGen} as this already illustrates all the main ideas that also enter in the proof of Theorem \ref{th:LowerBound}. First, (ii) of Corollary \ref{cor:HooleyGen} is an elementary consequence of (i), so we will only discuss the proof of (i). 
Second, unlike in the proof of Corollary \ref{cor:main1} we no longer need to select $\mathcal{S}_{\varepsilon}$ so that it consists of almost all the integers. Instead it suffices to choose a set $\mathcal{S}_{\varepsilon}$ such that a positive proportion of integers belong to $\mathcal{S}_{\varepsilon}$. 

For $h_0 \leq X^{\varepsilon^3/20000}$ we take $\mathcal{S}_{\varepsilon}$ similarly to the proof of Corollary \ref{cor:main1} but with a narrow first interval $(P_1, Q_1] = (h_0^{1 - \varepsilon}, h_0]$. Then,
\begin{equation} \label{eq:loww}
\sum_{\substack{x < n \leq x + h_0 \delta(\mathcal{N}; X)^{-1}}} 1 \geq \sum_{\substack{x < n \leq x + h_0 \delta(\mathcal{N}; X)^{-1} \\ n \in \mathcal{S}_{\varepsilon} \cap \mathcal{N}}} 1 
\end{equation}
  and repeating the argument of the proof of Corollary \ref{cor:main1} we can show that outside of an exceptional set of cardinality $\ll X P_1^{-1/2 + \varepsilon} + X^{1 - \varepsilon^3/5} \ll X h_0^{-1/2 + 2\varepsilon}$ the left-hand side of \eqref{eq:loww} is $\gg \varepsilon h_0$. Therefore we are done for $h_0 \leq X^{\varepsilon^3/20000}$.

  Therefore in the remainder assume that $h_0 > X^{\varepsilon^3/20000}$. In this case we pick $\mathcal{S}_{\varepsilon}$ to consist of the set of integers in $[X, 2X]$ that can be written as $p_1 \ldots p_k m$ with $p_1, \ldots p_k$ distinct primes in the interval $(X^{\varepsilon^{10}(1-\varepsilon^{20})}, X^{\varepsilon^{10}(1+\varepsilon^{20})}]$. We pick $k$ so large that $m$ is also essentially of size $X^{\varepsilon^{10}}$.

Write $H = h_0 \delta(\mathcal{N}; X)^{-1}$. We want to essentially bound, for some $\delta > 0$ depending only on $\alpha$ and $\varepsilon$, the frequency of those $x$ for which,
\begin{equation} \label{eq:mainqq}
\delta h_0 \geq \sum_{\substack{x < n \leq x + H \\ n \in \mathcal{S}_{\varepsilon} \cap \mathcal{N}}} 1 
\approx \sum_{\substack{P_1, \dotsc, P_k \\ P_1 \ldots P_k \asymp X^{1-\varepsilon^{10}}}}  \frac{H}{2\pi i} \int_{\substack{|t| \leq X / H}} 
\prod_{i = 1}^{k} \Big ( \sum_{\substack{p \sim P_i \\ p \in \mathcal{N}}} \frac{1}{p^{1 + it}} \Big ) \cdot \Big ( \sum_{\substack{m \asymp X/(P_1 \dotsm P_k) \\ m \in \mathcal{N}}} \frac{1}{m^{1 + it}} \Big ) x^{it} dt,
\end{equation}
where $P_i$ run over powers of two in $(X^{\varepsilon^{10}(1-\varepsilon^{20})}, X^{\varepsilon^{10}(1+\varepsilon^{20})}]$. The part of the integral with small $t$, say $|t| \leq T_0 := (\log X)^{\varepsilon'
}$ for some small $\varepsilon' > 0$ contributes to the main term, which is larger than $\delta h$ provided that $\delta > 0$ is chosen sufficiently small in terms of $\alpha$ and $\varepsilon$. Therefore our main task is to show that the contribution of the integral with $|t| > T_0$ is bounded by $o(h_0)$ outside of an exceptional set of cardinality $\ll X h_0^{-1/2 + \varepsilon}$.

In particular given $P_1, \ldots, P_k$ it suffices to show that outside of a small exceptional set of cardinality $\ll X h_0^{-1/2 + \varepsilon}$,
\begin{equation} \label{eq:inte2}
\int_{T_0}^{X / H} F(1 + it) x^{it} dt = o \Big ( \frac{\delta(\mathcal{N}; X)}{\log^k X} \Big ) 
\end{equation}
where $F(1 + it) = (P_1 \ldots P_k M)(1 + it)$,
$$
P_i(s) := \sum_{\substack{p \sim P_i \\ p \in \mathcal{N}}} \frac{1}{p^s}, \quad \text{ and } \quad M(s) = \sum_{\substack{m \asymp X/(P_1 \dotsm P_k) \\ m \in \mathcal{N}}} \frac{1}{m^{s}}.
$$
We first show that for all $x$ we can exclude from the integral \eqref{eq:inte2} the set of $t$ that belongs to either $\mathcal{U}_1$ or $\mathcal{U}_2$, where
\begin{align*}
  \mathcal{U}_1 & := \{T_0 \leq |t| \leq X / H : |F(1 + it)| > X^{-\varepsilon^{100}} \} \\
  \mathcal{U}_2 & := \{ T_0 \leq |t| \leq X / H : X^{-\varepsilon^{100}} \geq |F(1 + it)| > h_0^{-1/2+\varepsilon/2} \}. 
\end{align*}
For this it is enough to show that 
\begin{equation} \label{eq:ine1}
\int_{\substack{t \in \mathcal{U}_j}} |F(1 + it)| dt = o \Big ( \frac{\delta(\mathcal{N}; X)}{\log^k X} \Big ). 
\end{equation}
for $j \in \{1,2\}$.

The proof of \eqref{eq:ine1} for $j = 1$ is similar to an argument that we described in the proof sketch of Corollary \ref{cor:main1}. Specifically, if $|F(1 + it)|$ is large then at least one of $P_i(1 + it)$ is large, consequently the set $\mathcal{U}_1$ is small (because $P_1$ is short and therefore we can estimate the frequency with which $P_1(1 + it)$ is large by taking high moments). We then apply a point-wise bound on $M(1 + it)$ obtaining a small saving, and a variant of the Hal\'asz-Montgomery inequality on the remaining Dirichlet polynomials $P_i$, showing thus that the remaining integral $\int_{\mathcal{U}_1} |P_1 \ldots P_k| dt$ is dominated by at most one term that exhibits no cancellations (the application of Hal\'asz-Montgomery inequality relies on the fact that the set $\mathcal{U}_1$ is small). Since we obtained some cancellations in $|M(1 + it)|$ we win.

The proof of \eqref{eq:ine1} for $j = 2$ uses the fact that if $|F(1 + it)|$ is larger than $X^{-\beta}$ for some $\beta$ then at least one of the Dirichlet polynomials $|P_i(1 + it)|$ or $|M(1 + it)|$ is larger than $P_i^{-\beta}$ or $M^{-\beta}$. For simplicity let us assume that $|P_1(1 + it)|$ is larger than $P_1^{-\beta}$. Since $P_1$ is short we can then apply Huxley's large value estimate to $P_1^{k}$ with $k$ a conveniently choosen large power to estimate the frequency with which $|P_1(1 + it)| > P_1^{-\beta}$. This allows to exploit the full strength of Huxley's large value estimates. The process will involve an error of the size of the longest Dirichlet polynomial $P_i$ or $M$, and an additional fixed logarithmic loss, and here having $h_0 > X^{\varepsilon^3/20000}$ is useful in neutralizing these losses. 

After these reductions, in view of the claim \eqref{eq:inte2}, we see that it is enough to show that
$$
\frac{1}{X} \int_{X}^{2X} \Big |\int_{\substack{T_0 \leq |t| \leq X / H \\ t \not \in \mathcal{U}_1 \cup \mathcal{U}_2}} F(1 + it) x^{it} dt \Big |^2 dx \asymp \int_{\substack{T_0 \leq |t| \leq X / H \\ t \not \in \mathcal{U}_1 \cup \mathcal{U}_2}} |F(1 + it)|^2 dt \ll_{\varepsilon} \frac{1}{h_0^{1/2 - \varepsilon/2}}.
    $$
We separate the remaining values of $t \not \in \mathcal{U}_1 \cup \mathcal{U}_2$ into two sets,
  \begin{align*}
    \mathcal{T}_1 & :=  \{ |t| \leq X / H : X^{-1/4+\varepsilon/8} \leq |F(1 + it)| \leq h_0^{-1/2+\varepsilon/2} \} \\
    \mathcal{T}_2 & :=  \{ |t| \leq X / H : |F(1 + it)| \leq X^{-1/4+\varepsilon/8} \}. 
  \end{align*}

  Whenever $|F(1 + it)| \leq X^{-1/4+\varepsilon/8}$ we can write $F(1 + it)$ as $R(1 + it) N(1 + it)$ with $R(s)$ of length $\leq h_0$ and such that $|R(1 + it)| \leq h_0^{-1/4+\varepsilon/4}$. There is only a bounded number of possible choices for $R(s)$. As a result we can bound the integral over $\mathcal{T}_1$ simply by applying the point-wise bound to $R$ and the standard mean-value theorem to $N$.

  It remains to deal with the integral over the range $\mathcal{T}_2$ and this is addressed once again by appealing to Huxley's large value estimate. Since we now work with $|F|^2$ instead of $|F|$ this leads to a slightly different choice of parameters, so the repeated application of Huxley's estimate yields different result than when we applied it to deal with $\mathcal{U}_2$. 

\subsection{The proofs of the corollaries}

Regarding the result for norm-forms we notice that if $K$ is a number field with class number one,
then weaker versions of Theorems~\ref{thm:main2} and~\ref{thm:NormFormLowBound} in which we look at integers representable as $|N_{K / \mathbb{Q}}(x)|$ with $x \in \mathcal{O}_K$ are immediate consequences of Corollaries \ref{cor:main1} and~\ref{cor:HooleyGen}. 

Therefore the main difficulty that we are facing concerns number fields of class number exceeding one. We resolve this difficulty by showing that the indicator function $g_K(n)$ of the event ``$n$ is a norm-form of $K$'' can be expressed as a linear combination of multiplicative function. This is essentially implicit in the work of Odoni \cite{Odoni75} and we follow his argument to a large extent. Thanks to this we can prove Theorems~\ref{thm:main2} and~\ref{thm:NormFormLowBound} using similar arguments as in proofs of Theorems~\ref{th:MT} and~\ref{th:LowerBound}.

The proof of the conditional part of Theorem \ref{thm:NormFormLowBound} is different depending on whether $h$ is small or large. If $h_0 \leq X^{\varepsilon^3/20000}$, then the result follows from a minor modification of Corollary \ref{cor:main1} which uses the fact that on the Riemann Hypothesis for Hecke L-functions we have square-root cancellation in
$$
\sum_{p \leq x} g_{K}(p) p^{it}.
$$
This allows us to show that, outside of an exceptional set of cardinality $\ll_{\eta} X P_1^{-1 + \eta} + X^{1 - \varepsilon^3/20000}$, we have
$$
\sum_{\substack{x < n \leq x + h_0 \delta_K(X)^{-1} \\ n \in \mathcal{S}_{\varepsilon}}} g_K(n) \gg_\varepsilon h_0
$$
with $\mathcal{S_{\varepsilon}}$ the set of integers that have a prime factor in every interval $[P_i, Q_i]$ with $1 \leq i \leq J$ and $[P_1, Q_1] \subset [1, h_0]$. Choosing $P_1 = h_0^{1 - \varepsilon}$ and $Q_1 = h_0$ and the rest of the intervals $[P_i, Q_i]$ as in the proof of Corollary \ref{cor:main1} then leads to the result for  $h_0 \leq X^{\varepsilon^3/20000}$. 

For $h_0 > X^{\varepsilon^3/20000}$ we use the fact that the Riemann Hypothesis for Hecke $L$-functions implies square-root cancellation in
$$
\sum_{n \leq x} g_K(n) n^{it}
$$
and the argument in this case is rather simple, relaying essentially only on this point-wise bound. 
\addtocontents{toc}{\setcounter{tocdepth}{-10}}
\subsection*{Acknowledgments}
\addtocontents{toc}{\setcounter{tocdepth}{1}} 
The authors would like to thank Lilian Matthiesen and Jesse Thorner for pointing out some helpful references and Andrew Granville for discussions concerning the correct main term in the complex case. The second author would like to thank {\'E}tienne Fouvry for comments about normforms and for the hospitality at Universit{\'e} Orsay-Paris Sud in Winter 2018 where part of this work was completed. The first author was supported by Academy of Finland grant no. 285894. The second author was supported by a Sloan Fellowship and NSF grant DMS-1902063. Part of this work was completed while the authors were in residence at MSRI in Spring 2017, which was supported by NSF grant DMS-1440140.

\section{Mean value theorem for sparse Dirichlet polynomials}
For a Dirichlet polynomial $A(s) = \sum_{n \leq N} a_n n^{-s}$, the mean value theorem for Dirichlet polynomials (see~\cite[Theorem 9.1]{IwKo04}) gives, for any $T \geq 1$,
\begin{equation}
\label{eq:contMVT}
\int_{-T}^T |A(it)|^2 dt = (2T+O(N)) \sum_{n \leq N} |a_n|^2.
\end{equation}
On the right hand side the first term is supposed to reflect the contribution coming from $t$ with $|A(it)|^2$ of typical size. On the other hand the second term is supposed to reflect the contribution coming from a small set of $t$ for which $|A(it)|^2$ is close to its maximal size $(\sum_{n \leq N} |a_n|)^2$, but when $a_n$ is supported on a thin set, $N \sum_{n \leq N} |a_n|^2$ is not a good approximation to this. In order to get optimal results, we need to be very careful about such losses, and for this reason we use the following variant of the mean value theorem. The lemma below has been previously used in \cite{Teravainen, Goudout} to understand the distribution of integers with a fixed number of prime factors. 

\begin{lemma}
\label{le:contMVT2}
Let $A(s) = \sum_{n \leq N} a_n n^{-s}$ and $T \geq 1$. Then
\[
\begin{split}
\int_{-T}^T |A(it)|^2 dt &\ll T \sum_{n \leq N} |a_n|^2 + T \sum_{n \leq N} \sum_{0 < |k| \leq n/T} |a_n| |a_{n+k}| \\
&\ll T \sum_{n \leq N} |a_n|^2 + T \sum_{0 < |k| \leq N/T} \sum_{n \leq N} |a_n| |a_{n+k}|.
\end{split}
\]
\end{lemma}
\begin{proof}
This follows from \cite[Lemma 7.1]{IwKo04} taking $Y = 10T$ and $x_m = \frac{1}{2\pi}\log m$ there.
\end{proof}

To effectively use the previous lemma in our setting, we need to, for a multiplicative function $f$, understand the sums $\sum_{x< n \leq 2x} |f(n)|^2$ and $\sum_{0 < |k| \leq K} \sum_{x < n \leq 2x} |f(n)f(n+k)|$. 
The average of $|f(n)|^2$ as well as several other averages we will encounter can be estimated by the following result of Shiu~\cite{Shiu80}.
\begin{lemma}
\label{le:Shiu}
Let $\theta \in (0,1)$ and let $f \colon \mathbb{N} \to \mathbb{U}$ be multiplicative. Then, for $x \geq y \geq x^{\theta}$, one has
\[
\sum_{x < n \leq x+y} |f(n)| \ll_\theta y \prod_{p \leq x} \Big(1+\frac{|f(p)|-1}{p}\Big).
\]
\end{lemma}

The previous lemma as well as many results below actually work for a certain class of unbounded multiplicative functions but here we restrict our attention to the bounded functions. To estimate the shifted convolution sum we use the following lemma which is a consequence of the work of Henriot~\cite{Henriot}.

\begin{lemma}
\label{le:Henriot}
Let $\theta \in (0, 1]$ and let $f \colon \mathbb{N} \to \mathbb{U}$ be multiplicative. Let $1 \leq r_1, r_2 \leq x^{3\theta/7}$ be integers. Assume that $f(pm) = f(p) f(m)$ whenever $p \mid r_1 r_2$.

Then, for all $x \geq y \geq x^\theta$ and $K \in [1, x]$, one has
\[
\begin{split}
&\sum_{\substack{0 \neq |k| \leq K \\ (r_1, r_2) \mid k}}
\sum_{\substack{x < n \leq x + y \\ r_1 \mid n, r_2 \mid n+k}} |f(n) f(n + k)| \\
&\ll_\theta K \frac{|f(r_1) f(r_2)|}{r_1 r_2} y \prod_{p \leq x} \Big ( 1 + \frac{2|f(p)|-2}{p}
\Big ) \prod_{\substack{p \mid r_1 r_2 \\ p > K}} \Big(1+\frac{1-|f(p)|}{p}\Big).
\end{split}
\]
\end{lemma}
\begin{proof}
Writing $k = k_0 (r_1, r_2)$ and $n = n_0 (r_1, r_2)$, we can re-write the left hand side as
\[
|f((r_1, r_2))|^2 \sum_{\substack{0 \neq |k_0| \leq K/(r_1, r_2)}}
\sum_{\substack{\frac{x}{(r_1, r_2)} < n_0 \leq \frac{x + y}{(r_1, r_2)} \\ \frac{r_1}{(r_1, r_2)} \mid n_0, \frac{r_2}{(r_1, r_2)} \mid n_0+k_0}} |f(n_0) f(n_0 + k_0)|.
\]
Hence it suffices to prove the claim with the additional assumption that $(r_1, r_2) = 1$.

Write $n = r_1 m_1$ and $n+k = r_2m_2$, so that $r_2 m_2 - r_1 m_1 = k$. The solutions of this system can be parametrised as
\[
\begin{cases}
m_1 = x_1(k) + l r_2 &\\
m_2 = x_2(k) + l r_1, &
\end{cases}
\]
where $l$ runs through $\mathbb{Z}$ and $(x_1(k), x_2(k))$ is any solution. Let us choose $x_1'$ and $x_2'$ such that $r_2 x_2' - r_1 x_1' = 1$. By Bezout's theorem we can choose these such that $|x_1'| \leq r_2$ and $|x_2'| \leq r_1$ and furthermore necessarily $(x_1', x_2') = (x_1', r_2) = (x_2', r_1) = 1$. Then we take $x_1(k) = k x_1'$ and $x_2(k) = k x_2'$ getting that
\[
\begin{cases}
n = r_1m_1 = r_1(kx_1'+lr_2) = r_1 \cdot (k, r_2) \cdot \Big( \frac{k}{(k, r_2)} x_1' + l \frac{r_2}{(k, r_2)}\Big) &\\
n + k = r_2 m_2 = r_2(kx_2'+lr_1) = r_2 \cdot (k, r_1) \cdot \Big( \frac{k}{(k, r_1)} x_2' + l \frac{r_1}{(k, r_1)}\Big). &
\end{cases}
\]

Writing $S$ for the left hand side of the claim, we get
\begin{equation}
\label{eq:Shenupp}
\begin{split}
S &\leq |f(r_1) f(r_2)| \sum_{\substack{0 \neq |k| \leq K}} |f((k, r_2)) f((k, r_1))|  \\
& \qquad \cdot \sum_{\substack{\frac{x- r_1 k x_1'}{r_1 r_2} < l \leq \frac{x - r_1 k x_1' + y}{r_1 r_2}}} \Big|f\Big( \frac{k}{(k, r_2)} x_1' + l \frac{r_2}{(k, r_2)}\Big) f\Big( \frac{k}{(k, r_1)} x_2' + l \frac{r_1}{(k, r_1)}\Big)\Big|.
\end{split}
\end{equation}
We shall use Henriot's result~\cite[Theorem 3]{Henriot} (see also~\cite{Henriot2}). In his notation we have $Q_1(n) = \frac{r_2}{(k,r_2)} n + \frac{kx_1'}{(k, r_2)}, Q_2(n) = \frac{r_1}{(k,r_1)} n + \frac{kx_2'}{(k, r_1)}, F(n_1, n_2) = |f(n_1) f(n_2)|$ and
\[
D = \Big(\frac{r_2}{(k,r_2)} \cdot \frac{kx_2'}{(k, r_1)} - \frac{r_1}{(k,r_1)} \cdot \frac{kx_1'}{(k, r_2)}\Big)^2 = \Big(\frac{k(r_2 x_2'-r_1x_1')}{(k, r_1)(k, r_2)}\Big)^2 = \Big(\frac{k}{(k, r_1) (k,r_2)}\Big)^2,
\]
so that in particular $p \mid D \implies p \mid k$. Furthermore
\[
\rho_{Q_1}(p) =
\begin{cases}
1 & \text{if $p \nmid \frac{r_2}{(k, r_2)}$;} \\
0 & \text{otherwise,}
\end{cases}
\quad
\rho_{Q_2}(p) =
\begin{cases}
1 & \text{if $p \nmid \frac{r_1}{(k, r_1)}$;} \\
0 & \text{otherwise,}
\end{cases}
\]
and
\[
\rho(p) =
\begin{cases}
2 & \text{if $p \nmid \frac{r_1}{(k, r_1)} \frac{r_2}{(k, r_2)} D$;} \\
1 & \text{otherwise.}
\end{cases}
\]
Also
\[
\Delta_D = \prod_{p \mid D} \Big(1+\frac{O(1)}{p}\Big).
\]
We get from~\cite[Theorem 3]{Henriot} that, uniformly for $0 \neq |k| \leq K$, the sum over $l$ in~\eqref{eq:Shenupp} is bounded by
\[
\begin{split}
&\ll_\theta \frac{y}{r_1 r_2} \Delta_D \prod_{\substack{p \leq x/(r_1 r_2)}} \Big(1-\frac{\rho(p)}{p}\Big) \sum_{\substack{n_1 n_2 \leq x/(r_1 r_2) \\ (n_1 n_2, D) = 1}} |f(n_1)f(n_2)| \cdot \frac{\rho_{Q_1}(n_1) \rho_{Q_2}(n_2)}{n_1 n_2} \\
&\ll \frac{y}{r_1 r_2} \prod_{p \mid D}\Big(1+\frac{O(1)}{p}\Big) \prod_{\substack{p \leq x}} \Big(1-\frac{2}{p}\Big) \prod_{\substack{p \leq x}} \Big(1+\frac{|f(p)|}{p}\Big)^2 \prod_{p \mid \frac{r_1}{(k, r_1)} \cdot \frac{r_2}{(k, r_2)}} \Big(1 + \frac{1-|f(p)|}{p}\Big) \\
&\ll \frac{y}{r_1 r_2} \prod_{p \mid k}\Big(1+\frac{O(1)}{p}\Big) \prod_{\substack{p \leq x}} \Big(1+\frac{2|f(p)|-2}{p}\Big) \prod_{p \mid r_1 r_2} \Big(1 + \frac{1-|f(p)|}{p}\Big).
\end{split}
\]
Hence, by~\eqref{eq:Shenupp},
\[
\begin{split}
S &\leq \frac{y}{r_1 r_2} |f(r_1) f(r_2)| \sum_{\substack{0 \neq |k| \leq K}} |f((k, r_2)) f((k, r_1))|  \prod_{p \mid k}\Big(1+\frac{O(1)}{p}\Big) \\
&\qquad \cdot \prod_{\substack{p \leq x}} \Big(1+\frac{2|f(p)|-2}{p}\Big) \prod_{p \mid r_1 r_2} \Big(1 + \frac{1-|f(p)|}{p}\Big).
\end{split}
\]
By Lemma~\ref{le:Shiu},
\[
\begin{split}
&\sum_{0 < |k| \leq K} |f((k, r_1))f((k, r_2))| \prod_{p \mid k}\Big(1+\frac{O(1)}{p}\Big) \\
&\ll K \prod_{\substack{p \leq K \\ p \nmid r_1 r_2}} \Big(1+\frac{1+\frac{O(1)}{p}-1}{p}\Big) \prod_{\substack{p \leq K \\ p\mid r_1 r_2}} \Big(1+\frac{|f(p)|(1+\frac{O(1)}{p})-1}{p}\Big) \\
&\ll K \prod_{\substack{p \leq K \\ p \mid r_1 r_2}} \Big(1+\frac{|f(p)|-1}{p}\Big),
\end{split}
\]
so we obtain the claim.
\end{proof}


Combining Lemmas~\ref{le:contMVT2}--\ref{le:Henriot} we get the following mean-value theorem, which we will use repeatedly.
\begin{lemma}
\label{le:MVTwithHenriot}
Let $\theta \in (0,1)$ and $x \geq y \geq x^\theta$. Let $f \colon \mathbb{N} \to \mathbb{U}$ be multiplicative and let
$$
A(s) = \sum_{x < n \leq x+y} \frac{a_n}{n^{s}},
$$
where $|a_n| \leq |f(n)|$ for every positive integer $n$.
Then, for any $T \geq 1$,
\begin{align*}
\int_{-T}^T & |A(1+it)|^2 dt \ll \frac{T y}{x^2} \prod_{p \leq x} \Big(1+\frac{|f(p)|^2-1}{p}\Big) + \frac{y}{x} \prod_{p \leq x} \Big ( 1 + \frac{2|f(p)|-2}{p}\Big).
\end{align*}
\end{lemma}

\section{Hal\'asz-Montgomery type mean-value theorems}
In addition to the mean value theorem (Lemma~\ref{le:contMVT2}), we shall need some large value results for Dirichlet polynomials. We shall say that a set $\mathcal{T} \subset \mathbb{R}$ is one-spaced if $|t-u| \geq 1$ for all distinct $t, u \in \mathcal{T}$.
\begin{lemma}[Hal\'asz-Montgomery inequality for integers]
\label{le:Hallargevalint}
Let $A(s) = \sum_{n \leq N} a_n n^{-it}$, $T \geq 1$, and let $\mathcal{T} \subseteq [-T, T]$ be one-spaced. Then
$$
\sum_{t \in \mathcal{T}} |A(it)|^2 \ll (N + |\mathcal{T}| \sqrt{T}) \log 2T \sum_{n \leq N} |a_n|^2
$$
\end{lemma}
\begin{proof}
See \cite[Theorem 9.6]{IwKo04}.
\end{proof}

In the proofs of Theorems~\ref{thm:NormFormLowBound} and~\ref{th:LowerBound} we use Huxley's large value theorem which we state now.
\begin{lemma} \label{lem:Huxley}
Let $N, T \geq 3$. Let $A(s) = \sum_{N < n \leq 2N} a_n n^{-s}$ be a Dirichlet polynomial of length $N$, and write $G = \sum_{N < n \leq 2N} \frac{|a_n|^2}{n^2}$. Let $\mathcal{T} \subset [-T,T]$ be a one-spaced set such that $|A(1+it)| \geq V^{-1}$ for every $t \in \mathcal{T}$. Then
\[
|\mathcal{T}| \ll \Big(G N V^2 + G^3 N T V ^6 \Big) (\log T)^6.
\]
In particular if $|a_n| \leq 1$ for all $n$, then
\[
|\mathcal{T}| \ll \Big(V^2 + \frac{TV^6}{N^2}\Big) (\log T)^6
\]
\end{lemma}
\begin{proof}
See~\cite[Corollary 9.9]{IwKo04}.
\end{proof}

We will also need a Hal{\'a}sz-Montgomery type result on the primes. Before stating it we state a standard linear sieve upper bound which we shall also need in a few other occasions.
\begin{lemma} \label{le:linearsieve}
Let $D \geq z \geq 1$ and $\mathcal{P} \subset \mathbb{P} \cap [1,z]$. Write $P(z) = \prod_{p \in \mathcal{P}} p$. There exists a set $S^+$ satisfying the following three conditions.
\begin{enumerate}[(i)]
\item One has $1 \in S^+$ and if $d \in S^+$, then $d \leq D$ and $d \mid P(z)$.
\item One has
\[
\mathbf{1}_{(n, P(z)) = 1} \leq \sum_{\substack{d \mid n \\ d \in S^+}} \mu(d),
\]
where $\mathbf{1}_{A}$ is the indicator of the claim/set $A$.
\item For any multiplicative $g$ such that $g(p) \in [0, 1)$ for all $p \in \mathbb{P}$ and
\[
\prod_{\substack{w \leq p < z \\ p \in \mathcal{P}}} (1-g(p))^{-1} \leq \frac{\log z}{\log w}\Big(1+\frac{L}{\log w}\Big),
\]
for all $2 \leq w \leq z$ and some $L \geq 1$, one has
\[
\sum_{d \in S^+} \mu(d) g(d) \leq (F(s)+O_L((\log D)^{-1/6})) \prod_{p \mid P(z)} (1-g(p)),
\]
where $s = \frac{\log D}{\log z}$ and $F \colon [1, \infty) \to [1, \infty)$ is a decreasing function such that $F(s) = 2e^\gamma/s$ for $s \in [1, 3]$.
\end{enumerate}
\end{lemma}
\begin{proof}
See for example~\cite[Section 12.1]{Opera}
\end{proof}

Like~\cite[Lemma 11]{MainPaper}, the following lemma is optimized for very short polynomials. For longer polynomials, better results could be obtained by using different bounds for the zeta function at the end of the proof.
\begin{lemma}[Hal\'asz-Montgomery inequality for primes] \label{lem:primeshalasz}
Let $T \geq 3$ and let $\mathcal{T} \subset [-T, T]$ be one-spaced. Let $P(s) = \sum_{N < p \leq 2N} a(p) p^{it}$ be a Dirichlet polynomial of length $N \leq T^2$ whose coefficients are supported on primes. Then, for any $\varepsilon', \eta \in (0, 1/2)$,
\[
\sum_{t \in \mathcal{T}} |P(it)|^2 \ll_{\varepsilon'} \Big ( \frac{N}{\log N} + |\mathcal{T}| \cdot T^{\frac{9}{2} \eta^{3/2}} (\log T)^{2} \cdot N^{1 - \eta(1-\varepsilon')} \Big ) \sum_{N < p \leq 2N} |a(p)|^2.
\]
\end{lemma}
\begin{remark}
  A result with $T^{\frac{9}{2} \eta^{3/2}}$ replaced by $T^{\eta^{3/2} / 2 + \varepsilon}$ can be obtained by using a more recent result of Heath-Brown \cite{HBExp} instead of Ford's result \cite{FordZeta} that we will use in the proof. However, in that variant the implied constant would depend on $\varepsilon$ in the exponent of $T$ which is not acceptable in our application.
\end{remark}

\begin{proof}
Note that if $\eta < 1/\log N$, then the claim is trivial. Hence by duality (see e.g. \cite[Theorem 6 in Chapter 7]{MontgomeryTenLecture}) it is enough to show that,
$$
\sum_{N < p \leq 2N} \Big | \sum_{t \in \mathcal{T}} \frac{a(t)}{p^{i t}} \Big |^2 \ll_{\varepsilon'} \Big ( \frac{N}{\log N} + |\mathcal{T}| \cdot \eta^{-1} \cdot T^{\frac{9}{2} \eta^{3/2}} (\log T)^{2/3} \cdot N^{1 - \eta(1-\varepsilon')} \Big ) \sum_{t \in \mathcal{T}} |a(t)|^2
$$
for arbitrary coefficients $a(t) \in \mathbb{C}$. Let $\Phi \geq 0$ be a smooth function with $\Phi(x) = 1$ for $1 \leq x \leq 2$ and $\widetilde{\Phi}(1+it) \ll_A (1+ |t|)^{-A}$ for every $A > 0$ and $t \in \mathbb{R}$, where $\widetilde{\Phi}(s) := \int_{0}^{\infty} \Phi(x) x^{s - 1} dx$ denotes the Mellin transform of $\Phi(x)$. In addition, let $S^+$ be as in Lemma~\ref{le:linearsieve} with $D = z = N^{\varepsilon'}$ and $\mathcal{P} = [2, z] \cap \mathbb{P}$. We see that
\begin{equation}\label{eq:square}
\sum_{N < p \leq 2N} \Big | \sum_{t \in \mathcal{T}} \frac{a(t)}{p^{it}} \Big |^2 \leq \sum_{\substack{d \in S^+}} \mu(d)  \sum_{n}  \Big | \sum_{t \in \mathcal{T}} \frac{a(t)}{(d n)^{it}} \Big |^2 \Phi \Big ( \frac{d n}{N} \Big ).
\end{equation}
Now, for $u, v \in \mathbb{R}$,
$$
\sum_{n} n^{i u - i v} \Phi \Big ( \frac{d n}{N} \Big ) = \frac{1}{2\pi i} \int_{1 + \varepsilon - i \infty}^{1 + \varepsilon + i \infty} \zeta(s - i u + iv) \frac{N^s}{d^s} \widetilde{\Phi}(s) ds.
$$
Shifting the contour to $\Re s = 1 - \eta$ we collect a pole at $s = 1 + i u - iv$. We bound the remaining integral using a result of Ford~\cite[Theorem 1]{FordZeta} which gives
\[
|\zeta(\sigma + it)| \ll 1 + |t|^{\tfrac{9}{2} (1 - \sigma)^{3/2}} (\log (|t|+2))^{2/3} \quad \text{for $1/2 \leq \sigma \leq 1$}.
\]
This shows that
$$
\sum_{n} n^{i u - iv} \Phi \Big ( \frac{d n}{N} \Big ) = \frac{N^{1 + i u - iv}}{d^{1 + i u - i v}} \cdot \widetilde{\Phi}(1 + i u - iv) + O \Big ( (1+|u| + |v|)^{\tfrac{9}{2} \eta^{3/2}} \log((|u|+|v|+2)^{2/3} \Big ( \frac{N}{d} \Big )^{1 - \eta} \Big ).
$$
Therefore, using $|a(u) a(v)| \leq |a(u)|^2 + |a(v)|^2$, the right hand side of~\eqref{eq:square} is equal to
\[
\sum_{d \in S^+} \frac{\mu(d)}{d} \sum_{u, v \in \mathcal{T}} \overline{a(u)} a(v) N^{1 + iu - iv} \widetilde{\Phi}(1 + iu - iv) + O \Big ( T^{\tfrac{9}{2} \eta^{3/2}} (\log T)^{2/3} \cdot \eta^{-1} N^{\eta \varepsilon'} \cdot N^{1 - \eta} \cdot |\mathcal{T}| \cdot \sum_{t \in \mathcal{T}} |a(t)|^2 \Big )
\]
Using Lemma~\ref{le:linearsieve}(iii) with $g(p) = 1/p$ and again the inequality $|a(u) a(v)| \leq |a(u)|^2 + |a(v)|^2$ we bound this further by
$$
\ll_{\varepsilon'} \Big ( \frac{N}{\log N} + |\mathcal{T}| \cdot T^{\tfrac{9}{2} \eta^{3/2}} (\log T)^{2/3} \eta^{-1} N^{1 - \eta(1-\varepsilon')} \Big ) \sum_{t \in \mathcal{T}} |a(t)|^2
$$
as we claimed.
\end{proof}

\section{Hal\'asz type results}\label{se:halasz}
In this section we deduce Hal\'asz and Lipschitz type results for multiplicative functions taking values in $\mathbb{U}$. Since the average of the absolute value of our function over $(X, 2X]$ might be of order $(\log X)^{-\alpha}$ with $\alpha \in (0, 1)$, we cannot directly use the standard results in the literature which typically win a small power of logarithm, but we need to slightly modify the proofs of the existing results to take into account the average value of $f$. Variants in a similar spirit can also be found from papers of Matthiesen~\cite{Matthiesen} (see e.g. Lemma 4.6 there) and Tenenbaum~\cite{Tenenbaum17} (see in particular Corollaire 2.1 there).

As usual, we will relate averages of a multiplicative function $f$ to a Dirichlet series of the type
\[
F(s;X) := \prod_{p \leq X} \Big(1+\frac{f(p)}{p^s}+\frac{f(p^2)}{p^{2s}} + \dotsb\Big).
\]
To estimate $|F(1+it; X)|$, we notice that
\begin{equation}
\label{eq:FexpBound}
|F(1+it; X)| \ll \exp\Big(\Re \sum_{p \leq X} \frac{f(p)p^{-it}}{p}\Big),
\end{equation}
and use Lemma 
\ref{le:sparseDistEst} below to estimate the sum on the right hand side. Recall the definitions of $t_{f, X}$ and $\widehat{t}_{f, X}$ from Definition~\ref{def:complex} and the definition of $\rho_\alpha$ from~\eqref{eq:etaadef} --- these quantities will occur several times in this section. 

We also point the reader to the Appendix which contains a ``trivial'' inequality that will be used in the proof of the lemma below. 

\begin{lemma}
\label{le:sparseDistEst}
let $f: \mathbb{N} \rightarrow \mathbb{U}$ be a multiplicative function.
\begin{enumerate}[(i)]
\item Assume that $f$ is $(\alpha, X^\theta)$-non-vanishing for some $\alpha, \theta \in (0, 1]$. One has, for any $|t| \leq X$, and any $0 < \rho < \rho_{\alpha}$
\[
\sum_{p \leq X} \frac{|f(p)| - \Re f(p)p^{-it}}{p} \geq  \rho \min\{\log \log X, 3\log(|t-\widehat{t}_{f, X}|\log X+1)\} + O_{\rho, \theta}(1).
\]
\item One has, for any $|t| \leq X$, and any $0 < \rho < \rho_{1}$, 
\[
\sum_{p \leq X} \frac{1 - \Re f(p)p^{-it}}{p} \geq  \rho \min\{\log \log X, 3\log(|t-t_{f, X}|\log X+1)\} + O_{\rho}(1).
\]
\item If $f$ is almost real-valued, then (i) and (ii) hold also with $\widehat{t}_{f, X}$ and $t_{f, X}$ replaced by $0$.
\end{enumerate}
\end{lemma}
\begin{proof}
We prove (i) and point out the differences in proofs of (ii) and (iii) at the end of the proof. The claim is trivial for $|t- \widehat{t}_{f, X}| < 2/(\theta \log X)$. Therefore we can assume that $|t - \widehat{t}_{f, X}| \geq 2 / (\theta \log X)$. We start by using a similar argument as in works of Granville and Soundararajan (see e.g.~\cite[Proof of Lemma 2.3]{GS03}). By definition of $\widehat{t}_{f, X}$, we know that
\begin{equation}
\label{eq:Refplowbound}
\begin{split}
&\sum_{p \leq X} \frac{|f(p)| - \Re f(p)p^{-it}}{p} \geq \frac{1}{2} \sum_{p \leq X} \frac{|f(p)| - \Re f(p)p^{-it}}{p} + \frac{1}{2} \sum_{p \leq X} \frac{|f(p)| - \Re f(p)p^{-i\widehat{t}_{f, X}}}{p} \\
& = \sum_{p \leq X} \frac{|f(p)|}{p} - \Re \sum_{p \leq X} f(p) p^{-i \frac{t+\widehat{t}_{f, X}}{2}} \Big(\frac{p^{i \frac{|t-\widehat{t}_{f, X}|}{2}} + p^{-i \frac{|t-\widehat{t}_{f, X}|}{2}}}{2}\Big) \\
&\geq \sum_{p \leq X} \frac{|f(p)|}{p}\Big(1-\Big|\cos\Big(\frac{(t-\widehat{t}_{f, X}) \log p}{2} \Big)\Big|\Big) \\
&\geq \sum_{Y < p \leq X^\theta} \frac{|f(p)|}{p}\Big(1-\Big|\cos\Big(\pi \Big\Vert \frac{(t-\widehat{t}_{f, X}) \log p}{2\pi} \Big\Vert\Big)\Big|\Big),
\end{split}
\end{equation}
where $\Vert x \Vert$ denotes the distance to the nearest integer, and
\[
Y := \max\{\exp((\log X)^{2/3+\varepsilon}, \exp(1/|t-\widehat{t}_{f, X}|)\}
\]
for some small $\varepsilon > 0$.

One would expect that the right hand side of ~\eqref{eq:Refplowbound} is smallest among $(\alpha, X^\theta)$-non-vanishing $f$ when
\[
|f(p)| = 
\begin{cases}
1 & \text{if $\left\Vert \frac{(t-\widehat{t}_{f, X}) \log p}{2\pi} \right\Vert \leq \alpha/2$;} \\
0 & \text{otherwise.}
\end{cases}
\]
(depending on the parameters, this might not be $(\alpha, X^\theta)$-non-vanishing but let us ignore this) and that 
\begin{equation}
\label{eq:f(p)coslb}
\begin{split}
&\sum_{Y < p \leq X^\theta} \frac{|f(p)|}{p}\Big(1-\Big|\cos\Big(\pi \Big\Vert \frac{(t-\widehat{t}_{f, X}) \log p}{2\pi} \Big\Vert\Big)\Big|\Big) \\
&\geq  2 \int_{0}^{\alpha/2} (1-\cos(\pi x)) dx \cdot \log \frac{\log X^\theta}{\log Y} - O_{\rho, \theta}(1). \\
\end{split}
\end{equation}
While \eqref{eq:f(p)coslb} is a natural-looking inequality, the proof is a bit tedious, so we postpone the rigorous proof to an appendix (see Lemma~\ref{le:appendix}).

Once we have~\eqref{eq:f(p)coslb}, we immediately obtain the claim (i) since
\begin{equation}
\label{eq:rhoint}
 2 \int_{0}^{\alpha/2} (1-\cos(\pi x)) dx = \alpha - \tfrac{2}{\pi} \sin(\tfrac{\pi}{2} \alpha).
\end{equation}

In case (ii) one can run the same argument with $1$ in place of $|f(p)|$ and $t_{f, X}$ in place of $\widehat{t}_{f, X}$, getting that
\[
\begin{split}
\sum_{p \leq X} \frac{1 - \Re f(p)p^{-it}}{p} &\geq \sum_{Y < p \leq X} \frac{1}{p}\Big(1-\Big|\cos\Big(\pi \Big\Vert \frac{(t-t_{f, X}) \log p}{2\pi} \Big\Vert\Big)\Big|\Big).
\end{split}
\]
Now the claim follows from~\eqref{eq:f(p)coslb} and \eqref{eq:rhoint} with $\alpha = 1$ and $f$ identically $1$.

In case (iii), we argue similarly, except we notice that when $f$ is almost real-valued, one has e.g.
\[
\sum_{p \leq X} \frac{|f(p)| - \Re f(p)p^{-it}}{p} \geq \sum_{p \leq X} \frac{|f(p)|(1- |\cos(t \log p)|)}{p} + O(1).
\]

\end{proof}

In the following, we often need to apply results on slight variants of the original multiplicative function $f$ where we have changed $f$ to be zero on some primes. This might affect the values of $t_{f, X}$ and $\widehat{t}_{f, X}$ as well as the non-vanishing-condition. For those situations we have the following variant of Lemma~\ref{le:sparseDistEst}

\begin{lemma}
\label{le:sparseDistEst2}
Let $f: \mathbb{N} \rightarrow \mathbb{U}$ be a multiplicative function, and let $\mathcal{P} \subset \mathbb{P} \cap (1, X]$.
\begin{enumerate}[(i)]
\item Assume that $f$ is $(\alpha, X^\theta)$-non-vanishing for some $\alpha, \theta \in (0, 1]$. One has, for any $|t| \leq X$, and any $0 < \rho < \rho_{\alpha}$, 
\[
\sum_{\substack{p \leq X \\ p \not \in \mathcal{P}}} \frac{\Re f(p)p^{-it}}{p} \leq \sum_{p\leq X} \frac{|f(p)|}{p} - \frac{\rho}{2} \min\{\log \log X, 3\log(|t-\widehat{t}_{f, X}|\log X+1)\} + O_{\theta, \rho}(1).
\]
\item One has, for any $|t| \leq X$, and any $0 < \rho < \rho_{1}$, 
\[
\sum_{\substack{p \leq X \\ p \not \in \mathcal{P}}} \frac{\Re f(p)p^{-it}}{p} \leq \sum_{p\leq X} \frac{1}{p} - \frac{\rho}{2} \min\{\log \log X, 3\log(|t-t_{f, X}|\log X+1)\} + O_{\rho}(1). 
\]
\item If $f$ is almost real-valued, then (i) and (ii) hold also with $\widehat{t}_{f, X}$ and $t_{f, X}$ replaced by $0$
\item One has, for all $|t| \leq X$,
\[
\sum_{\substack{p \leq X \\ p \not \in \mathcal{P}}} \frac{\Re f(p)p^{-it}}{p} \leq \sum_{p\leq X} \frac{|f(p)|}{p} - \frac{1}{2} \widehat{M}(f; X)
\]
and
\[
\sum_{\substack{p \leq X \\ p \not \in \mathcal{P}}} \frac{\Re f(p)p^{-it}}{p} \leq \sum_{p\leq X} \frac{1}{p} - \frac{1}{2} M(f; X)
\]
Moreover, if $\sum_{p \in \mathcal{P}} \frac{1}{p} = O(1)$, then the factors $\frac{1}{2}$ can be removed if one adds an additional term $O(1)$ to the right hand side. 
\end{enumerate}
\end{lemma}

\begin{proof}
To prove (i), we obtain two upper bounds:
First by the trivial estimate
\[
\sum_{\substack{p \leq X \\ p \not \in \mathcal{P}}} \frac{\Re f(p)p^{-it}}{p} \leq \sum_{p \leq X} \frac{|f(p)|}{p} - \sum_{\substack{p \leq X \\ p \in \mathcal{P}}} \frac{|f(p)|}{p} = : B_1.
\]
Second, by Lemma~\ref{le:sparseDistEst}(i),
\[
\begin{split}
\sum_{\substack{p \leq X \\ p \not \in \mathcal{P}}} \frac{\Re f(p)p^{-it}}{p} &\leq \sum_{p \leq X} \frac{\Re f(p)p^{-it} - |f(p)|}{p} + \sum_{p \leq X} \frac{|f(p)|}{p} + \sum_{\substack{p \leq X \\ p \in \mathcal{P}}} \frac{|f(p)|}{p} \\
&\leq - \rho_\alpha \min\{\log \log X, 3\log(|t-\widehat{t}_{f, X}|\log X+1)\} + O_{\theta, \rho}(1) \\
& \qquad + \sum_{p \leq X} \frac{|f(p)|}{p} + \sum_{\substack{p \leq X \\ p \in \mathcal{P}}} \frac{|f(p)|}{p} \\ &=: B_2.
\end{split}
\]
The joint upper bound $(B_1 + B_2)/2$ gives the claim.

Cases (ii) and (iii) follow similarly from Lemma~\ref{le:sparseDistEst}(ii)--(iii), and the first two claims of case (iv) follow similarly using Definition~\ref{def:complex}. The last claim in case (iv) follows directly from Definition~\ref{def:complex}.
\end{proof}

Let us now state our variant of Hal\'asz's theorem.
\begin{lemma} \label{le:SparseHalaszComplex}
Let $f: \mathbb{N} \rightarrow \mathbb{U}$ be a multiplicative function, and let $\mathcal{P} \subset \mathbb{P} \cap (1, X]$.
\begin{enumerate}[(i)]
\item Assume that $f$ is $(\alpha, X^\theta)$-non-vanishing for some $\alpha, \theta \in (0, 1]$. One has, for all $|t| \leq X/2$ and $x \leq X$, and any $0 < \rho < \rho_{\alpha}$, 
  $$
  \Big | \sum_{\substack{x < n \leq 2x \\ p \mid n \implies p \not \in \mathcal{P}}} \frac{f(n)}{n^{1 + it}} \Big | \ll_{\theta, \rho} \Big ( \frac{\log\log X}{|t-\widehat{t}_{f, X}|^{1/2} + 1} + \frac{1}{(\log X)^{\rho/2}} \Big ) \frac{1}{\log x} \prod_{p \leq X} \Big ( 1 + \frac{|f(p)|}{p} \Big )
  $$
  and
\[
  \Big | \sum_{\substack{x < n \leq 2x \\ p \mid n \implies p \not \in \mathcal{P}}} \frac{f(n)}{n^{1 + it}} \Big | \ll_{\theta} \left(\frac{\widehat{M}(f; X)}{\exp(\frac{1}{2} \widehat{M}(f; X))} + \frac{1}{(\log X)^{\alpha}}\right) \frac{1}{\alpha \log x}  \prod_{p \leq X} \Big ( 1 + \frac{|f(p)|}{p} \Big ).
\]
\item One has, for all $|t| \leq X/2$ and $x \leq X$, and any $0 < \rho < \rho_1$, 
  \[
  \Big | \sum_{\substack{x < n \leq 2x \\ p \mid n \implies p \not \in \mathcal{P}}} \frac{f(n)}{n^{1 + it}} \Big | \ll_\rho  \frac{\log\log X}{|t-t_{f, X}|^{1/2} + 1} + \frac{(\log X)^{1-\rho/2}}{\log x}.
  \]
  and
  \[
  \Big | \sum_{\substack{x < n \leq 2x \\ p \mid n \implies p \not \in \mathcal{P}}} \frac{f(n)}{n^{1 + it}} \Big | \ll \left( \frac{M(f; X)}{\exp(\frac{1}{2} M(f; X))} + \frac{1}{\log X}\right) \cdot \frac{\log X}{\log x}.
\]
\item If $f$ is almost real-valued, then (i) and (ii) hold also with $\widehat{t}_{f, X}$ and $t_{f, X}$ replaced by $0$. 
\item If $\sum_{p \in \mathcal{P}} \frac{1}{p} = O(1)$, the claims hold with $\frac{1}{2}\widehat{M}(f; X)$ and $\frac{1}{2}M(f; X)$ replaced by $\widehat{M}(f; X)$ and $M(f; X)$.
\end{enumerate}
\end{lemma}

\begin{proof}
We prove the case (i) and point out the differences to cases (ii)--(iv) in the end of the proof.

By partial summation it is essentially enough to show the same bounds for 
\[
\frac{1}{x}\Big|\sum_{\substack{x < n \leq 2x \\ p \mid n \implies p \not \in \mathcal{P}}} f(n)n^{-it}\Big|.
\]

Let us first concentrate on the first claim of (i). Notice that we can assume that $|t-\widehat{t}_{f, X}| > \log \log X$ since otherwise the claim follows from Lemma~\ref{le:Shiu}.

We use intermediate results in Montgomery's refinement of the proof of Hal\'asz's theorem. Write, for $\Re s > 1$,
\[
F(s) = \sum_{\substack{n \in \mathbb{N} \\ p \mid n \implies p \not \in \mathcal{P}}} \frac{f(n)}{n^s}
\]
and
\[
H(\beta)^2 = \sum_{k \in \mathbb{Z}} \frac{1}{k^2+1} \max_{|\tau-k| \leq 1/2} |F(1+\beta+it+i\tau)|^2.
\]
Now Montgomery's work (see~\cite[Theorem 4.7 in Section III.4.3]{Tenenbaum}) gives
\begin{equation}
\label{eq:Montg}
\frac{1}{x}\Big|\sum_{\substack{x < n \leq 2x \\ p \mid n \implies p \not \in \mathcal{P}}} f(n)n^{-it}\Big| \ll \frac{1}{\log x} \int_{1/\log x}^1 \frac{H(\beta)}{\beta} d\beta.
\end{equation}
Similarly to~\cite[Formula (4.61) in Section III.4.3]{Tenenbaum} we have
\[
|F(1+\beta+it + i\tau)| \ll \exp\Big(\Re \sum_{\substack{p \leq \exp(1/\beta) \\ p \not \in \mathcal{P}}} \frac{f(p)}{p^{1+it+i\tau}}\Big).
\]
Hence
\begin{equation}
\label{eq:Hbetaup}
H(\beta)^2 \ll \sum_{k \in \mathbb{Z}} \frac{1}{k^2+1} \max_{|\tau-k| \leq 1/2} \exp\Big(2 \Re \sum_{\substack{p \leq \exp(1/\beta) \\ p \not \in \mathcal{P}}} \frac{f(p)}{p^{1+it+i\tau}}\Big).
\end{equation}
Let us first note that, for $\frac{1}{\log x} \leq \beta \leq 1$,
\[
\begin{split}
&\sum_{|k| > (|t-\widehat{t}_{f, X}|+1)/2} \frac{1}{k^2+1} \max_{|\tau-k| \leq 1/2} \exp\Big(2 \Re \sum_{\substack{p \leq \exp(1/\beta) \\ p \not \in \mathcal{P}}} \frac{f(p)}{p^{1+it+i\tau}}\Big) \\
&\leq \sum_{|k| > (|t-\widehat{t}_{f, X}|+1)/2} \frac{1}{k^2+1} \exp\Big(2\sum_{p \leq x} \frac{|f(p)|}{p}\Big) \ll \frac{1}{|t-\widehat{t}_{f, X}|+1} \exp\Big(2\sum_{p \leq X} \frac{|f(p)|}{p}\Big),
\end{split}
\]
so this part leads to an acceptable contribution to~\eqref{eq:Montg}. Similarly $|k| \geq (\log X)^4$ lead to an acceptable contribution.

Let us now consider the contribution of $k$ with $|k| \leq (|t-\widehat{t}_{f, X}|+1)/2$ into $H(\beta)^2$. By Lemma~\ref{le:sparseDistEst2}(i) (taking $\mathcal{P}$ there to be $\mathcal{P} \cup (\exp(1/\beta), X] \cap \mathbb{P}$), recalling that $|t-\widehat{t}_{f, X}| > \log \log X$,
\[
\begin{split}
2\Re \sum_{\substack{p \leq \exp(1/\beta) \\ p \not \in \mathcal{P}}} \frac{f(p)}{p^{1+it+i\tau}} &\leq  2\sum_{p \leq X} \frac{|f(p)|}{p} - \rho' \log \log X +O_{\rho, \theta}(1),
\end{split}
\]
where $\rho' = (\rho + \rho_\alpha)/2$.

Combining the previous estimates we see that
\[
H(\beta)^2  \ll_{\rho, \theta} \Big(\frac{1}{|t-\widehat{t}_{f, X}|+1} + \frac{1}{(\log X)^{\rho'}}\Big) \exp\Big(2\sum_{p \leq X} \frac{|f(p)|}{p}\Big)
\]
from which the claim follows by~\eqref{eq:Montg} since the integration over $\beta$ contributes $\log \log x$.

To prove the second claim of case (i) we again estimate $|k| \geq (\log X)^4$ in~\eqref{eq:Hbetaup} trivially. Then we use Lemma \ref{le:sparseDistEst2}(iv) for $\beta \leq \exp(\widehat{M}(f; X)/(2\alpha))/\log X$ whereas for larger $\beta$ we use the trivial bound
\[
\begin{split}
2\Re \sum_{\substack{p \leq \exp(1/\beta) \\ p \not \in \mathcal{P}}} \frac{f(p)}{p^{1+it+i\tau}} &\leq  2\sum_{p \leq X} \frac{|f(p)|}{p} - 2\sum_{\exp(1/\beta) < p \leq X^\theta} \frac{|f(p)|}{p} \\
&\leq 2\sum_{p \leq X} \frac{|f(p)|}{p} - 2 \alpha \log (\beta \log X) + O_\theta(1).
\end{split}
\]
Combining these with~\eqref{eq:Montg}, we obtain 
\[
\begin{split}
\frac{1}{x}\Big|\sum_{\substack{x < n \leq 2x \\ p \mid n \implies p \not \in \mathcal{P}}} f(n)n^{-it}\Big| &\ll_\theta \frac{1}{\log x}  \exp\left(\sum_{p\leq X} \frac{|f(p)|}{p}\right) \\
&\cdot \left(\int_{1/\log X}^{\frac{\exp(\widehat{M}(f; X)/(2\alpha))}{\log X}} \frac{\exp(- \frac{1}{2} \widehat{M}(f; X))}{\beta} d\beta + \int_{\frac{\exp(\widehat{M}(f; X)/(2\alpha))}{\log X}}^1 \frac{(\beta \log X)^{-\alpha}}{\beta} d\beta \right),
\end{split}
\]
and the claim follows by executing the integrals.

Cases (ii)--(iv) follow similarly using Lemma~\ref{le:sparseDistEst2}(ii)--(iv). 
\end{proof}

We will also need to evaluate the average of $f(n)$ on intervals slightly shorter than dyadic. For this we use the following Lipschitz type result which we deduce from work of Granville and Soundararajan~\cite{GS03} (compare in particular with~\cite[Theorem 4]{GS03}). See also~\cite[Lemma 4.6]{Matthiesen} for another sparse version.

\begin{lemma}
\label{le:Lipschitz}
Let $f: \mathbb{N} \rightarrow \mathbb{U}$ be a multiplicative function, and let $\mathcal{P} \subset \mathbb{P} \cap (1, X]$.
\begin{enumerate}[(i)]
\item Assume that $f$ is $(\alpha, X^\theta)$-non-vanishing for some $\alpha, \theta \in (0, 1]$ and that $|\widehat{t}_{f, X}| \leq X/2$. Let $0 < \rho < \rho_{\alpha}$. One has, for all $y \in [x/(\log X)^{\rho/2}, x]$ and $x \in [X/2, X]$,
\[
\Big | \frac{1}{y}\sum_{\substack{x < n \leq x + y \\ p \mid n \implies p \not \in \mathcal{P}}} f(n) n^{-i\widehat{t}_{f, X}} - \frac{1}{X} \sum_{\substack{X < n \leq 2X \\ p \mid n \implies p \not \in \mathcal{P}}} f(n) n^{-i\widehat{t}_{f, X}} \Big | \ll_{\rho, \theta} \frac{X/y}{(\log X)^{\rho/2}} \prod_{p \leq X} \Big(1+\frac{|f(p)|-1}{p}\Big) .
\]
\item Let $0 < \rho < \rho_1$. One has, for all $y \in [x/(\log X)^{\rho/2}, x]$ and $x \in [X/2, X]$,
\[
\Big | \frac{1}{y}\sum_{\substack{x < n \leq x + y \\ p \mid n \implies p \not \in \mathcal{P}}} f(n) n^{-it_{f, X}} - \frac{1}{X} \sum_{\substack{X < n \leq 2X \\ p \mid n \implies p \not \in \mathcal{P}}} f(n) n^{-it_{f, X}} \Big | \ll_{\rho} \frac{X/y}{(\log X)^{\rho/2}}.
\]
\item If $f$ is almost real-valued, then (i) and (ii) hold also with $\widehat{t}_{f, X}$ and $t_{f, X}$ replaced by $0$.
\end{enumerate}
\end{lemma}

\begin{proof}
We prove the case (i) and point out the differences to cases (ii)--(iii) in the end of the proof.

It suffices to show that, for any $1 \leq w \leq 4$ and any $X' \in [X/4, 4X]$, one has
\begin{equation}
\label{eq:Lipsch}
S := \Big|\frac{1}{X} \sum_{\substack{n \leq X \\ p \mid n \implies p \not \in \mathcal{P}}} f(n) n^{-i\widehat{t}_{f, X}} - \frac{1}{X/w} \sum_{\substack{n \leq X/w \\ p \mid n \implies p \not \in \mathcal{P}}} f(n) n^{-i\widehat{t}_{f, X}}\Big| \ll_{\rho, \theta} \frac{1}{(\log X)^{\rho/2}} \prod_{p \leq X} \Big(1+\frac{|f(p)|-1}{p}\Big).
\end{equation}
from which the claim follows easily.

Let
\[
F(s) = \prod_{\substack{p \leq X \\ p \not \in \mathcal{P}}} \Big(1+\frac{f(p)}{p^s} + \frac{f(p^2)}{p^{2s}} + \dotsb \Big).
\]
By~\cite[Proposition 3.3]{GS03} with $T = \log X$, we get
\begin{equation}
\label{eq:GSProp3.3appl}
\begin{split}
S &\ll \frac{1}{\log X} \int_0^1 \min\{\log X, 1/\beta\} \Big(\max_{|t| \leq T} |(1-w^{-\beta-it})F(1+\beta+i\widehat{t}_{f, X}+it)|\Big) d\beta \\
& \qquad +O\Big(\frac{\log \log X}{\log X}\Big).
\end{split}
\end{equation}
Next, as in~\cite[Proof of Theorem 4]{GS03}, we use~\cite[Lemma 2.2]{GS03} but with $a_n = f(n)1_{p \mid n \implies p \leq X, p \not \in \mathcal{P}}$, which gives
\[
\max_{|t| \leq T} |(1-w^{-\beta-it})F(1+\beta+i\widehat{t}_{f, X}+it)| \leq \max_{|t| \leq 2T} |(1-w^{-it})F(1+i\widehat{t}_{f, X}+it)| + O(\beta)
\]
The error term contributes to the right hand side of~\eqref{eq:GSProp3.3appl} only $O(\frac{1}{\log X})$. Hence we obtain that
\begin{equation}
\label{eq:LipIntEst}
S \ll \frac{\log \log X}{\log X}\Big(\max_{|t| \leq 2T} |(1-w^{-it})F(1+i\widehat{t}_{f, X}+it)| + 1\Big).
\end{equation}
Now $|1-w^{-it}| \ll \min\{1, |t| \log 2w\}$. Thus, writing $\rho' = (\rho + \rho_\alpha)/2$ and applying~\eqref{eq:FexpBound} and Lemma~\ref{le:sparseDistEst2}(i), we obtain
\[
\begin{split}
&\max_{|t| \leq 2T} |(1-w^{-it})F(1+i\widehat{t}_{f, X} + it)|\\
&\ll_{\rho, \theta} \max_{|t| \leq 2T} \Big\{\min\{1, |t|\} \exp\Big(\sum_{\substack{p \leq X}} \frac{|f(p)|}{p} - \frac{\rho'}{2} \min\{\log \log X, 3\log (|t|\log X+1)\}\Big)\Big\} \\
&\ll \max_{|t| \leq 2T}\Big\{ \min\{1, |t|\} \max\{(\log X)^{-\rho'/2}, (|t| \log X+1)^{-3\rho'/2}\} \prod_{p \leq X}\Big(1+\frac{|f(p)|}{p}\Big)\Big\}\\
&\leq (\log X)^{-\rho'/2} \prod_{p \leq X} \Big(1+\frac{|f(p)|}{p}\Big).
\end{split}
\]
The claim follows now from~\eqref{eq:LipIntEst}.

Parts (ii) and (iii) follow similarly using Lemma~\ref{le:sparseDistEst2}(ii)--(iii).
\end{proof}

We will actually need to apply Lemma~\ref{le:SparseHalaszComplex} in case $f$ has a non-multiplicative dependence on primes in a certain range $(P, Q]$. To this end, we have the following variant of~\cite[Lemma 3]{MainPaper} which is tailored for large $P$ whereas~\cite[Lemma 3]{MainPaper} was intended for somewhat smaller $P$ and $Q$.

\begin{lemma}
\label{le:Halappl}
Let $A> 0$, $X \geq P \geq 2$. Let $f: \mathbb{N} \rightarrow \mathbb{U}$ be a multiplicative function. Let $r \colon \mathbb{N} \to \mathbb{U}$ depend only on prime factors of $n$ that are $> P$ (i.e. for $m \mid n$ one has $r(n) = r(m)$ whenever $p \mid n/m \implies p \leq P$). Let $\mathcal{P} \subset \mathbb{P} \cap (1, X]$ and write
\[
F(s) = \sum_{\substack{x < n \leq 2x \\ p \mid n \implies p \not \in \mathcal{P}}} \frac{f(n) r(n)}{n^s}.
\]
\begin{enumerate}[(i)]
\item Assume that $f$ is $(\alpha, X^\theta)$-non-vanishing for some $\alpha, \theta \in (0, 1]$. For any $|t| \leq X/2$, $x \in (P, X]$, and $0 < \rho < \rho_{\alpha}$, 
\begin{align} \label{eq:Halappl}
|F(1+it)| &\ll_{\theta, \rho} \frac{1}{\log P} \Big(\frac{1}{(\log X)^{\rho/2}} + \frac{(\log \log X)^2}{(|t-\widehat{t}_{f, X}|+1)^{1/2}}\Big) \prod_{p \leq X} \Big(1+\frac{|f(p)|}{p}\Big).
\end{align}
\item For any $|t| \leq X/2$, $x \in (P, X]$, and $0 < \rho < \rho_1$, 
\[
|F(1+it)| \ll_{\rho} \frac{\log X}{\log P} \Big(\frac{1}{(\log X)^{\rho/2}} + \frac{(\log \log X)^2}{(|t-t_{f, X}|+1)^{1/2}}\Big).
\]
\item If $f$ is almost real-valued, then (i) and (ii) hold   also with $\widehat{t}_{f, X}$ and $t_{f, X}$ replaced by $0$.
\end{enumerate}
\end{lemma}

\begin{proof}
We write
\[
F(1+it) = \sum_{\substack{m \leq 2x \\ p \mid m \implies p > P, p \not \in \mathcal{P}}} \frac{f(m) r(m)}{m^{1+it}} \sum_{\substack{x/m < n \leq 2x/m \\ p \mid n \Rightarrow p \leq P, p \not \in \mathcal{P}}}  \frac{f(n)}{n^{1+it}}.
\]
In case (i) we get by Lemma~\ref{le:SparseHalaszComplex} that, with $\rho' = (\rho + \rho_\alpha)/2$,
\[
|F(1+it)| \ll_{\rho, \theta} \sum_{\substack{m \leq 2x \\ p \mid m \implies p > P}} \frac{1}{m \log x/m} \Big(\frac{\log \log X}{(|t-\widehat{t}_{f, X}|+1)^{1/2}} + \frac{1}{(\log X)^{\rho'/2}}\Big) \prod_{p \leq X} \Big(1+\frac{|f(p)|}{p}\Big).
\]
Now~\eqref{eq:Halappl} follows by noting that
\[
\sum_{\substack{m \leq 2x \\ p \mid m \implies p > P}} \frac{1}{m \log x/m} = \frac{1}{\log x} + \sum_{\substack{P < m \leq 2x \\ p \mid m \implies p > P}} \frac{1}{m \log x/m} \ll \frac{\log \log X}{\log P}.
\]
Cases (ii) and (iii) follow similarly from Lemma~\ref{le:SparseHalaszComplex}(ii)--(iii).
\end{proof}

\section{Decomposition of Dirichlet polynomials}
In this section we establish a variant of~\cite[Lemma 12]{MainPaper} which gives a Buchstab/Ramar\'e-type decomposition of a Dirichlet polynomial.

\begin{lemma} \label{lem:decomp}
Let $H \geq 1$,  $X^{1/5} \geq Q \geq P \geq 1$. Let $g \colon \mathbb{N} \to [0, 1]$ be multiplicative and let $a_m, b_m$ and $c_p$ be sequences in $\mathbb{U}$ such that
\[
\begin{cases}
|a_m|, |b_m|, |c_m| \leq g(m) & \text{for all $m$;} \\
a_{mp} = b_m c_p & \text{whenever $P < p \leq Q$ and $p \nmid m$} \\
a_m = 0 & \text{if $m$ has no prime factor in the interval $(P, Q]$.}
\end{cases}
\]
Let
\begin{align*}
Q_{\nu,H}(s) & = \sum_{\substack{P < p \leq Q \\ e^{\nu/H} < p \leq e^{(\nu + 1)/H}}}
\frac{c_p}{p^s} \quad \text{and} \\
  R_{\nu,H}(s) & = \sum_{\substack{X e^{-\nu/H} < m \leq  2X e^{-\nu/H}}} \frac{b_m}{m^s} \cdot \frac{1}{\#\{P < q \leq Q: q | m,  q \in \mathbb{P}\} + 1} \\
\end{align*}
and let $\mathcal{T} \subseteq [-T, T]$. Then
\begin{align*}
  \int_{\mathcal{T}} \Big| \sum_{X < m \leq 2X} \frac{a_m}{m^{1+it}} \Big|^2 dt &\ll \Big(H \log \Big ( \frac{Q}{P} \Big ) + 1 \Big) \sum_{\nu \in \mathcal{I}} \int_{\mathcal{T}} \Big | Q_{\nu,H}(1 + it) R_{\nu,H}(1 + it) |^2 dt \\
& \qquad + \frac{T}{X}\Big(\frac{1}{H} + \frac{1}{P}\Big) \prod_{p \leq X} \Big(1+\frac{g(p)^2-1}{p}\Big) \\
& \qquad + \Big(\frac{1}{H} + \frac{1}{P}\Big) \prod_{p \leq X} \Big ( 1 + \frac{2g(p)-2}{p}\Big)
\end{align*}
where $\mathcal{I}$ is the interval $\lfloor H \log P \rfloor \leq \nu \leq  H \log Q$.
\end{lemma}

\begin{proof}
Here and later we write
\[
\omega_{(P,Q]}(m) := \sum_{\substack{P < p \leq Q \\ p | m}} 1.
\]
Let us also write $s=1+it$ and notice that since $a_m$ are supported on numbers having a prime factor in $(P, Q]$, we have
\[
\begin{split}
\sum_{\substack{X < m \leq 2X}} \frac{a_m}{m^s} &=
\sum_{P < p \leq Q} \sum_{\substack{X/p < m \leq 2X/p}} \frac{a_{mp}}{(mp)^s} \cdot \frac{1}{\omega_{(P,Q]}(m) + \mathbf{1}_{(p, m) = 1}} \\
& = \sum_{P < p \leq Q} \frac{c_p}{p^s} \sum_{\substack{X/p < m \leq 2X/p}} \frac{b_m}{m^s} \cdot \frac{1}{\omega_{(P, Q]}(m) + 1} + \sum_{\substack{X < mp \leq 2X \\ p \in (P, Q], \, (m, p) = p}} \frac{e_{m, p}}{(mp)^s},
\end{split}
\]
where
\[
|e_{m,p}| = \Big|\frac{a_{mp}}{\omega_{(P, Q]}(m)} - \frac{b_{m} c_p}{\omega_{(P, Q]}(m) + 1}\Big| \leq g(mp) + g(m)g(p).
\]

We split the first sum further into short intervals getting that it equals
$$
\sum_{\nu \in \mathcal{I} } \ \sum_{\substack{e^{\nu/H} < p \leq e^{(\nu+1)/H} \\ P < p \leq Q}}
\frac{c_p}{p^s} \ \sum_{\substack{X e^{-(\nu+1)/H} < m \leq 2X e^{-\nu/H} \\ X < mp \leq 2X}} \frac{b_m}{m^s} \cdot \frac{1}{\omega_{(P, Q]}(m) + 1}
$$
We remove the condition $X < mp \leq 2X$ overcounting at most by the integers $mp$ in the ranges $(X e^{-1/H}, X]$ and $(2X, 2X e^{1/H}]$. Similarly, removing numbers with $Xe^{-(\nu+1)/H} < m \leq Xe^{-\nu/H}$ we undercount at most by integers $mp$ in the range $(Xe^{-1/H}, Xe^{1/H}]$. Therefore, for some $d_m$ and $e_{m, p}'$ with $|d_m| \leq |g(m)|$ and $|e_{m,p}'| \ll |g(mp)| + |g(m) g(p)|$, we have
\[
\begin{split}
&\sum_{\substack{X < m \leq 2X}} \frac{a_m}{m^s} = \sum_{\nu \in \mathcal{I}} Q_{\nu, H}(s) R_{\nu, H}(s)  \\
& \qquad + \sum_{\substack{X e^{-1/H} < m \leq Xe^{1/H}}} \frac{d_m}{m^s} + \sum_{\substack{2X < m \leq 2X e^{1/H}}} \frac{d_m}{m^s} + \sum_{\substack{X e^{-1/H} < mp \leq 2X e^{1/H}\\ p \in (P, Q], \, (m, p) = p}} \frac{e'_{m, p}}{(mp)^s}.
\end{split}
\]
We square both sides of this equation, integrate over $t \in \mathcal{T}$ and then apply Cauchy-Schwarz on the first sum over $\nu$ and Lemma~\ref{le:MVTwithHenriot} on the second and third sums (with $y \asymp X/H$). These clearly lead to an acceptable contribution. For the fourth sum we get from Lemma~\ref{le:contMVT2}
\[
\begin{split}
&\int_{-T}^T \Big|\sum_{\substack{X e^{-1/H} < mp \leq 2X e^{1/H}\\ p \in (P, Q], \, (m, p) = p}} \frac{e'_{m, p}}{(mp)^s} \Big|^2 dt\\
&\ll T \sum_{X/3 < n \leq 6X} \frac{1}{n^2} \sum_{\substack{p_1, p_2 \in (P, Q] \\ m_1, m_2 \\ n = p_1^2 m_1 = p_2^2 m_2}} |e'_{m_1 p_1, p_1} e'_{m_2 p_2, p_2}| + T \sum_{0 < |k| \leq 6X/T} \sum_{X/3 < n \leq 6X} \sum_{\substack{p_1, p_2 \in (P, Q] \\ m_1, m_2 \\ p_1^2 m_1 = n \\ p_2^2 m_2 = n+k}} \frac{|e'_{m_1 p_1, p_1}||e'_{m_2p_2, p_2}|}{n(n+k)}.
\end{split}
\]
Defining now, for each pair $(p_1, p_2)$, the multiplicative function $g_{p_1, p_2}$ such that $g_{p_1, p_2}(p) = g(p)$ when $p \not\in\{p_1, p_2\}$ and $g_{p_1, p_2}(p^\alpha) = 1$ when $p \in \{p_1, p_2\}$ or $\alpha > 1$, we obtain that the previous expression is
\[
\begin{split}
&\ll \frac{T}{X^2} \sum_{p_1, p_2 \in (P, Q]} \sum_{\substack{\frac{X}{3[p_1, p_2]^2} < n \leq \frac{6X}{[p_1, p_2]^2}}} g_{p_1, p_2}(n)^2  \\
& \qquad + \frac{T}{X^2} \sum_{p_1, p_2 \in (P, Q]} \sum_{\substack{0 < |k| \leq 6X/T \\ (p_1^2, p_2^2) \mid k}} \sum_{\substack{X/3 < n \leq 6X \\ p_1^2 \mid n, p_2^2 \mid n+k}} g_{p_1, p_2}(n) g_{p_1, p_2}(n+k).
\end{split}
\]
Applying now, for each $p_1, p_2$, Lemmas~\ref{le:Shiu} and~\ref{le:Henriot}, we obtain that
\[
\begin{split}
&\int_{-T}^T \Big|\sum_{\substack{X e^{-1/H} < mp \leq 2X e^{1/H}\\ p \in (P, Q], \, (m, p) = p}} \frac{e'_{m, p}}{(mp)^s} \Big|^2 dt \\
&\ll \frac{T}{X} \sum_{p_1, p_2 \in (P, Q]} \frac{1}{[p_1, p_2]^2} \prod_{p \leq X} \Big(1+\frac{g(p)^2-1}{p}\Big)  + \sum_{p_1, p_2 \in (P, Q]} \frac{1}{p_1^2 p_2^2} \prod_{p \leq X} \Big(1+\frac{2g(p)-2}{p}\Big) \\
&\ll \frac{T}{XP} \prod_{p \leq X} \Big(1+\frac{g(p)^2-1}{p}\Big)  + \frac{1}{P} \prod_{p \leq X} \Big(1+\frac{2g(p)-2}{p}\Big).
\end{split}
\]
\end{proof}

\section{Moment computation}
In this section we prove an analogue of~\cite[Lemma 13]{MainPaper}. Let us first introduce some
relevant notation.
Let $X, Y_1, Y_2 \geq 2$, and consider,
$$
Q(s) = \sum_{Y_1 < p \leq 2Y_1} \frac{c_p}{p^s} \quad \text{and} \quad
A(s) = \sum_{\substack{X/Y_2 < m \leq 2X/Y_2}} \frac{a_m}{m^s}.
$$
\begin{lemma}
\label{le:moment}
Let $\ell = \lceil \frac{\log Y_2}{\log Y_1} \rceil$ and let $g \colon \mathbb{N} \to [0, 1]$ be multiplicative.
Assume that $|c_p| \leq 1$ for all $p$, $|a_m| \leq g(m)$ for all $m$, and $Y_2 \leq X^{1/5}$. Then
\[
\begin{split}
&\int_{-T}^{T} |Q(1 + it)^{\ell} \cdot A(1 + it) |^2 dt \\
&\ll \ell!^2 \Big(\frac{T}{X} \prod_{p \leq X} \Big(1+\frac{|g(p)|^2-1}{p}\Big) + \prod_{p \leq X} \Big ( 1 + \frac{2|g(p)|-2}{p}\Big)\Big)
\end{split}
\]
\end{lemma}
\begin{proof}
The coefficients of the Dirichlet polynomial $Q(s)^{\ell} A(s)$ are supported on the interval
$$
(Y_1^\ell \cdot X/Y_2, (2Y_1)^\ell\cdot 2 X/Y_2] \subseteq (X, 2^{\ell+1}Y_1 X].
$$
Let $g^\ast$ be a multiplicative function such that $g^\ast(p^\alpha) = g(p^\alpha)$ if $p \not \in (Y_1, 2Y_1]$ and $g^\ast(p^\alpha) = 1$ if $p \in (Y_1, 2Y_1]$.
Using the improved mean-value theorem for Dirichlet polynomials (first inequality of Lemma~\ref{le:contMVT2}) and then splitting the $n$ sums dyadically, we see that
\[
\begin{split}
&\int_{-T}^{T} |Q(1 + it)^{\ell} \cdot A(1 + it) |^2 dt \\
&\ll \sum_{\substack{X/2 < y \leq 2^{\ell+1}Y_1 X \\ y = 2^r}} \Bigl( T \sum_{y < n \leq 2y} \frac{1}{n^2} \Big ( \sum_{\substack{n = m p_1 \ldots p_{\ell} \\ Y_1 < p_1, \ldots, p_{\ell} \leq 2Y_1 \\ X/Y_2 < m \leq 2 X/Y_2 }} g^\ast(n) \Big )^2 \\
& \qquad + T\sum_{\substack{0 < |k| \leq 2y / T \\ y < n \leq 2y}} \sum_{\substack{n = m p_1 \ldots p_{\ell} \\ Y_1 < p_1, \ldots, p_{\ell} \leq 2Y_1 \\ X /Y_2 < m \leq 2 X /Y_2}} \frac{g^\ast(n)}{n} \sum_{\substack{n+k = m' p'_1 \ldots p'_{\ell} \\ Y_1 < p'_1, \ldots, p'_{\ell} \leq 2Y_1 \\ X /Y_2 < m' \leq 2 X /Y_2}} \frac{g^\ast(n+k)}{n+k} \Bigr)\\
&\ll \sum_{\substack{X/2 < y \leq 2^{\ell+1}Y_1 X \\ y = 2^r}} \frac{T}{y^2}  \Bigl( \sum_{\substack{Y_1 < p_1, \dotsc, p_{\ell} \leq 2Y_1 \\ Y_1 < p_1', \dotsc, p_{\ell}' \leq 2Y_1}} \sum_{\frac{y}{[p_1 \dotsm p_{\ell}, p_1' \dotsm p_{\ell}']} < n \leq \frac{2y}{[p_1 \dotsm p_{\ell}, p_1' \dotsm p_{\ell}']}} g^\ast(n)^2  \\
& \qquad + \sum_{\substack{Y_1 < p_1, \dotsc, p_{\ell} \leq 2Y_1 \\ Y_1 < p_1', \dotsc, p_{\ell}' \leq 2Y_1}} \sum_{\substack{0 < |k| \leq 2y / T \\ (p_1 \dotsm p_{\ell}, p_1' \dotsm p_{\ell}') \mid k}} \sum_{\substack{y < n \leq 2y \\ p_1 \dotsm p_\ell \mid n \\ p_1' \dotsm p_\ell' \mid n+k}} g^\ast(n) g^\ast(n+k) \Bigr).
\end{split}
\]

By Lemmas~\ref{le:Shiu} and~\ref{le:Henriot} we see that this is
\[
\begin{split}
&\ll \sum_{\substack{X/2 < y \leq 2^{\ell+1}Y_1 X \\ y = 2^r}} \Bigl(\frac{T}{y} \sum_{\substack{Y_1 < p_1, \dotsc, p_{\ell} \leq 2Y_1 \\ Y_1 < p_1', \dotsc, p_{\ell}' \leq 2Y_1}} \frac{1}{[p_1 \dotsm p_{\ell}, p_1' \dotsm p_{\ell}']} \prod_{p \leq X} \Big(1+\frac{g(p)^2-1}{p}\Big) \\
& \qquad + \sum_{\substack{Y_1 < p_1, \dotsc, p_{\ell} \leq 2Y_1 \\ Y_1 < p_1', \dotsc, p_{\ell}' \leq 2Y_1}} \frac{1}{p_1 \dotsm p_\ell \cdot p_1' \dotsm p_\ell'} \prod_{p \leq X} \Big(1+\frac{2g(p)-2}{p}\Big)\Bigr).
\end{split}
\]
The claim follows since
\[
\begin{split}
&\sum_{\substack{Y_1 < p_1, \dotsc, p_{\ell} \leq 2Y_1 \\ Y_1 < p_1', \dotsc, p_{\ell}' \leq 2Y_1}} \frac{1}{[p_1 \dotsm p_{\ell}, p_1' \dotsm p_{\ell}']} \\
&\leq \sum_{r=1}^\ell \sum_{Y_1 < p_1, \dotsc, p_r \leq 2Y_1} \frac{r!}{p_1 \dotsm p_r} {l \choose r}^2 \sum_{\substack{Y_1 < p_{r+1}, \dotsc, p_{\ell} \leq 2Y_1 \\ Y_1 < p'_{r+1}, \dotsc, p'_{\ell} \leq 2Y_1}} \frac{1}{ p_{r+1} \dotsm p_{\ell} \cdot p'_{r+1} \dotsm p'_{\ell}} \\
&\leq \sum_{r=1}^\ell r! {\ell \choose r}^2 \ll \ell!^2.
\end{split}
\]
\end{proof}

\section{Parseval bound}


Let us first state a standard Parseval type bound that we will use frequently.
\begin{lemma}
\label{le:Parseval}
Let $\mathcal{T} \subseteq \mathbb{R}$, let $X \geq y \geq 1$, and let $A \colon \mathbb{C} \to \mathbb{C}$ be such that $\sup_{t \in \mathcal{T}} |A(1+it)| \leq g(X)$ for some $g(X)$. Then
\[
\begin{split}
&\frac{1}{X y^2} \int_{X}^{2X} \Big| \int_{\mathcal{T}} A(1+it) \frac{(x+y)^{1+it}-x^{1+it}}{1+it} dt \Big|^2 dx \\
&\ll  \max_{T \geq X/y} \frac{X/y}{T} \int_{[-T, T] \cap \mathcal{T}} |A(1+it)|^2  dt.
\end{split}
\]
\end{lemma}
\begin{proof}
This follows e.g. from the arguments in~\cite[Proof of Lemma 14]{MainPaper}.
\end{proof}

\begin{assumption}
\label{as:ambmcp}
Let $X^{1/2-\eps} \geq Q \geq P \geq Q' \geq P' \geq 1$, and let $f \colon \mathbb{N} \to \mathbb{U}$ be a multiplicative function. Let $a_m, b_m, d_m$ and $c_p$ be sequences in $\mathbb{U}$ such that
\[
\begin{cases}
|a_m|, |b_m|, |c_m| \leq |f(m)| & \text{for all $m$;} \\
a_{mpp'} = b_m c_{p} c_{p'} & \text{whenever $p \in (P, Q], p' \in (P', Q']$ and $(pp', m) = 1$;} \\
a_m = 0 & \text{if $\omega_{(P, Q]}(m) = 0$ or $\omega_{(P', Q']}(m) = 0$;} \\
a_m = d_m & \text{if $X/4 < m \leq 4X$.}
\end{cases}
\]
\end{assumption}
We set
$$
F_{W}(s) := \sum_{n} \frac{a_n}{n^s} \cdot W \Big ( \frac{n}{X} \Big )
$$
with $W(\cdot)$ a smooth function such that $W(x) = 1$ for $1 < x \leq 5/2$ and $W$ is compactly supported on $(1/2, 4]$. This Dirichlet polynomial will be used for studying averages over short intervals. Let us first collect some properties of $F_W(s)$.

Similarly to the Buchstab/Ramar\'e-type identity in Lemma~\ref{lem:decomp} we get
\[
\begin{split}
F_W(s) &= \sum_{\substack{m \\ P < p \leq Q \\ P' < p' \leq Q'}} \frac{a_{m p p'} W(mp p'/X)}{(mp p')^{s} \omega_{(P, Q]}(mpp') \omega_{(P', Q']}(mpp') } \\
&= \sum_{\substack{m \\ P < p \leq Q \\ P' < p' \leq Q'}} \frac{b_m c_p c_{p'} W(mp p'/X)}{(mp p')^{s} (\omega_{(P, Q]}(m)+1) (\omega_{(P', Q']}(m) + 1) } + F_{\square}(s),
\end{split}
\]
where
\begin{equation}
\label{eq:Fsquaredef}
F_\square(s) = \sum_{\substack{X/2 < mpp' \leq 4X \\ P < p \leq Q \\ P' < p' \leq Q' \\ (pp', m) > 1}} \frac{e_{m, p, p'}}{(m p p')^s}
\end{equation}
with $|e_{m,p, p'}| \ll |f(mpp')| + |f(m)f(p)f(p')|$. Hence we obtain by Mellin inversion
\begin{equation}
\label{eq:FWsdecomp}
\begin{split}
F_W(s) = \sum_{\substack{A,B,C}} \frac{1}{2\pi} \int_{-\infty}^{\infty} Q_{1,A}(s + i u) Q_{2,B}(s + iu) R_{C}(s + iu) \widetilde{W}(i u) \cdot X^{i u} du + F_{\square}(s),
\end{split}
\end{equation}
where $A, B, C$ traverses powers of two such that $X / 16 < ABC \leq 4X, P/2 < A \leq Q$, $P'/2 < B \leq Q'$, and where
\begin{equation}
\label{eq:WQRdef}
\begin{split}
\widetilde{W}(s) &:= \int_{0}^{\infty} W(x) x^{s -1} dx \\
Q_{1,D}(s) &:= \sum_{\substack{D < p \leq 2D \\ P < p \leq Q}} \frac{c_p}{p^s} \quad \text{and} \quad Q_{2,D}(s) = \sum_{\substack{D < p' \leq 2D \\ P' < p' \leq Q'}} \frac{c_{p'}}{p'^s} \\
\text{and} \quad R_{D}(s) &:= \sum_{D < m \leq 2D} \frac{b_m}{m^s (\omega_{(P, Q]}(m)+1) (\omega_{(P', Q']}(m) + 1) }.
\end{split}
\end{equation}

The following proposition, which is a variant of~\cite[Lemma 14]{MainPaper}, shows that the behaviour of a sequence satisfying Assumption~\ref{as:ambmcp} in almost all very short intervals can be approximated by its behaviour in a long interval (when appropriately twisted) if the mean square of the corresponding Dirichlet polynomial is small. Compared to~\cite[Lemma 14]{MainPaper}, the following proposition allows us to discard some parts of the mean square integral which will allow us to get power-type savings in the exceptional set.
\begin{proposition}
\label{prop:L1work}
Let $\eps > 0$ be small, $X^{1/2-\eps} \geq Q \geq P \geq Q' \geq P' \geq 1$, $f$, $a_m, b_m, d_m$ and $c_p$ be as in Assumption~\ref{as:ambmcp}, and let $W(x), F_W(s), Q_{j,D}(s)$ and $R_D(s)$ be as above. Assume $\nu_1 := \log P'/\log X \geq (\log X)^{-1/4}$. Let $1 \leq y_1 \leq y_2 \leq X$ and $1 \leq T_0 \leq X/10$. Assume that $y_2 > X^\eps$. Let $\mathcal{U} \subseteq [-X, X]$. Then, for any $|t_0| \leq X$ and $X < x \leq 2X$,
\begin{equation}
\label{eq:Sxy1y2diff}
\begin{split}
&\frac{1}{y_1} \sum_{x < n \leq x+y_1} a_n - \frac{1}{y_1} \int_x^{x+y_1} u^{i t_0} du \cdot \frac{1}{y_2} \sum_{x < n \leq x+y_2} a_n n^{-it_0} \\
&= A(x, y_1, y_2, t_0, T_0, \mathcal{U}) + O_\varepsilon\Big( \Big(\frac{y_2 T_0^2}{X} + \frac{1}{y_1} + \frac{\Big ( 1 + \log \frac{\log Q}{\log P'} \Big )^2}{T_0^{10}}\Big) \cdot \prod_{p \leq X} \Big (1 + \frac{|f(p)| - 1}{p} \Big)\Big) \\
& \qquad  + O \Big( \Big ( 1 + \log \frac{\log Q}{\log P'} \Big )^2 \cdot \sup_{\substack{|t| \leq X/2, |t-t_0| \geq T_0 \\ X/(16QQ') < C \leq 16X/(PP')}} |R_C(1+it)| \Big),
\end{split}
\end{equation}
where $A(x, y_1, y_2, t_0, T_0, \mathcal{U})$ is such that
\begin{equation}
\label{eq:A2bound}
\begin{split}
&\frac{1}{X} \int_{X}^{2X} |A(x, y_1, y_2, t_0, T_0, \mathcal{U})|^2 dx \ll \mathcal{I}(X, y_1) +  \frac{1}{y_1} \prod_{p \leq X} \Big(1+\frac{|f(p)|^2-1}{p}\Big) \\
&\qquad + \frac{(\log X)^{5}}{X^{\nu_1^3/160}} \max_{X/y_1 \leq T \leq X} \Big(\frac{|(\mathcal{U} + [-X^\varepsilon, X^\varepsilon]) \cap [-T, T]| \cdot Q'}{y_1 \sqrt{T}} + 1 \Big),
\end{split}
\end{equation}
with
\[
\mathcal{I}(X, y_1) = \max_{X/y_1 \leq T \leq X} \frac{X/y_1}{T} \int_{[-T, T] \setminus \mathcal{U}} \Big | \sum_{\substack{X/100 < n \leq 100X}} \frac{d_n}{n^{1 + it}} \Big |^2 dt,
\]
and where, for two sets $A, B \subset \mathbb{R}$, we define $A + B := \{ a + b: a \in A, b \in B\}$. 
\end{proposition}

\begin{proof} 
By Perron's formula (see e.g.~\cite[Lemma 1.1]{Harman06})
\[
\begin{split}
\frac{1}{y_1} \sum_{x < n \leq x+y_1} a_n &= \frac{1}{y_1} \cdot \frac{1}{2\pi i} \int_{-X^{10}}^{X^{10}} F_W(1+it) \cdot \frac{(x + y_1)^{1+it} - x^{1+it}}{1+it} dt \\
& \qquad \qquad+  O\left(\frac{1}{y_1}\min\left\{1, \frac{1}{X^{5} \min\{\Vert x \Vert, \Vert x+ y_1\Vert \}}\right\}\right).
\end{split}
\]
The square mean over $x \in [X, 2X]$ of the error term is very small, and thus the error term can be included to $A(x, y_1, y_2, t_0, T_0, \mathcal{U})$.

We split the integration range $[-X^{10}, X^{10}]$ into three parts:
\[
\begin{split}
\mathcal{V}_1 &=
\begin{cases}
\{t \colon |t-t_0| \leq T_0 \} &\text{if $|t_0| \leq X/3$} \\
\emptyset & \text{otherwise.}
\end{cases} \\
\mathcal{V}_2 &= [-X/8, X/8] \setminus \mathcal{V}_1 \\
\mathcal{V}_3 &= \{X/8 \leq |t| \leq X^{10}\} \setminus \mathcal{V}_1.
\end{split}
\]
Writing, for $\mathcal{T} \subset \mathbb{R}$,
\[
I_{\mathcal{T}}(x) = \frac{1}{y_1} \cdot \frac{1}{2\pi i} \int_{\mathcal{T}} F_W(1+it) \cdot \frac{(x + y_1)^{1+it} - x^{1+it}}{1+it} dt,
\]
we thus need to prove the claim with $\frac{1}{y_1} \sum_{x < n \leq x+y_1} a_n$ replaced by
\[
I_{\mathcal{V}_1}(x) + I_{\mathcal{V}_2}(x) + I_{\mathcal{V}_3}(x).
\]

By Lemmas~\ref{le:Parseval} and~\ref{le:MVTwithHenriot},
\[
\begin{split}
&\frac{1}{X} \int_X^{2X} \left|I_{\mathcal{V}_3}(x)\right|^2 dx \ll \max_{X/8 \leq T \leq X^{10}} \frac{X/y_1}{T} \int_{-T}^T |F_W(1+it)|^2 dt \\
&\ll \max_{X/8 \leq T \leq X^{10}} \frac{X/y_1}{T} \Big(\frac{T}{X} \prod_{p \leq X} \Big(1+\frac{|f(p)|^2-1}{p}\Big) + \prod_{p\leq X} \Big(1+\frac{2|f(p)|-2}{p}\Big)\Big) \\
& \ll \frac{1}{y_1} \prod_{p \leq X} \Big(1+\frac{|f(p)|^2-1}{p}\Big),
\end{split}
\]
so that $I_{\mathcal{V}_3}(x)$ can be included into $A(x, y_1, y_2, t_0, T_0, \mathcal{U})$.

In case $\mathcal{V}_1$ is empty, one has $|t_0| > X/3$ and so
\[
\frac{1}{y_1} \int_x^{x+y_1} u^{it_0}du = \frac{1}{y_1} \frac{(x+y_1)^{1+it_0} - x^{1+it_0}}{1+it_0} \ll \frac{1}{y_1}.
\]
Hence in this case the second term on the left hand side of~\eqref{eq:Sxy1y2diff} leads by Shiu's bound (Lemma~\ref{le:Shiu}) to an acceptable contribution
\[
\ll \frac{1}{y_1} \prod_{p \leq X} \Big(1+\frac{|f(p)|-1}{p}\Big).
\]
Therefore in case $\mathcal{V}_1 = \emptyset$, it suffices to show that $I_{\mathcal{V}_2}(x)$ can be written in the form on the right hand side of~\eqref{eq:Sxy1y2diff}.

When $\mathcal{V}_1$ is non-empty, we borrow arguments from the on-going work \cite{GHMR}. In this case
\[
\begin{split}
I_{\mathcal{V}_1}(x) &= \frac{1}{y_1} \cdot \frac{1}{2\pi i} \int_{-T_0}^{T_0} F_W(1+it_0+it) \cdot \frac{(x + y_1)^{1+it_0+it}
  - x^{1+it_0+it}}{1+it_0 + it} dt \\
  &=\frac{1}{y_1} \cdot \frac{1}{2\pi i} \int_{-T_0}^{T_0} F_W(1+it_0+it) \cdot \int_{x}^{x+y_1} u^{it_0 + it} du \, dt.
\end{split}
\]
For $|t| \leq T_0$ and $u \in [x, x+y_1]$, we have $u^{it} = x^{it} + O(y_1 T_0 /x)$, so that, using also Lemma~\ref{le:Shiu},
\[
I_{\mathcal{V}_1}(x) = \frac{1}{y_1} \int_x^{x+y_1} u^{it_0} du \cdot \frac{1}{2\pi i} \int_{- T_0}^{T_0} F_W(1+it_0+it) x^{it} dt + O\Big(\frac{y_1 T_0^2}{X} \prod_{p \leq X}\Big(1+\frac{|f(p)|-1}{p}\Big)\Big).
\]

On the other hand, again by Perron's formula, the sum $\frac{1}{y_2} \sum_{x < n \leq x+y_2} a_n n^{-it_0}$ in the second term on the left hand side of~\eqref{eq:Sxy1y2diff} equals
\[
\frac{1}{y_2} \cdot \frac{1}{2\pi i} \int_{-X^{10}}^{X^{10}} F_W(1+it_0 +it) \cdot \frac{(x + y_2)^{1+it} - x^{1+it}}{1+it} dt
\]
apart from an error term which has acceptable square mean over $x \in [X, 2X]$.
Arguing as above, the region $|t| \geq X/8$ leads to an acceptable contribution, whereas the region $|t| \leq T_0$ contributes a main term
\[
\frac{1}{2\pi i} \int_{- T_0}^{T_0} F_W(1+it_0+it) x^{it} dt + O\Big(\frac{y_2 T_0^2}{x} \prod_{p \leq x}\Big(1+\frac{|f(p)|-1}{p}\Big)\Big).
\]
Since this main term corresponds to the main term from $I_{\mathcal{V}_1}(x)$, it remains to study $I_{\mathcal{V}_2}(x)$ and
\[
\frac{1}{y_2} \cdot \frac{1}{2\pi i} \int_{T_0 \leq |t| \leq X/8} F_W(1+it_0 +it) \cdot \frac{(x + y_2)^{1+it} - x^{1+it}}{1+it} dt.
\]
These can be treated similarly and we concentrate on $I_{\mathcal{V}_2}(x)$.

Write
\[
\mathcal{U}_{2} = \mathcal{V}_2 \cap \Big(\mathcal{U} + [-X^{\varepsilon/2}, X^{\varepsilon/2}] \Big) \quad \text{and} \quad \mathcal{T}_2 = \mathcal{V}_2 \setminus \mathcal{U}_{2}.
\]
Let us first consider $I_{\mathcal{T}_2}(x)$ which will be included to $A(x, y_1, y_2, t_0, T_0, \mathcal{U})$ and leads to a term of the type $\mathcal{I}(X,y_1)$ there. By Lemma~\ref{le:Parseval}
\[
\begin{split}
&\frac{1}{X} \int_{X}^{2X} \Big|I_{\mathcal{T}_2}(x)\Big|^2 dx \ll \max_{X/y_1 \leq T \leq X/8} \frac{X/y_1}{T} \int_{\mathcal{T}_2 \cap [-T, T]} |F_W(1 + it)|^2 dt.
\end{split}
\]
Now notice that, by Mellin inversion,
$$
F_W(1 + it) = \frac{1}{2\pi} \int_{-\infty}^{\infty} \sum_{\substack{X/ 100 < n \leq 100 X}} \frac{d_n}{n^{1 + it + iu}} \cdot X^{iu} \widetilde{W}(i u) du .
$$
Due to the rapid decay of $\widetilde{W}(iu)$ we can truncate at $|u| \leq X^{\varepsilon/3}$. Then applying Cauchy-Schwarz, we get that
\begin{align*}
\int_{\mathcal{T}_2 \cap [-T, T]} |F_W(1 + it)|^2 dt &\ll \int_{-X^{\varepsilon / 3}}^{X^{\varepsilon / 3}} \int_{\mathcal{T}_2 \cap [-T, T] + u} \Big | \sum_{\substack{X /100 < n \leq 100 X}} \frac{d_n}{n^{1 + it}} \Big |^2 dt  \cdot |\widetilde{W}(i u)| du  + O\left(\frac{1}{X^{10}}\right)\\
&\ll \int_{[-2T, 2T] \setminus \mathcal{U}} \Big | \sum_{\substack{X /100 < n \leq 100 X}} \frac{d_n}{n^{1 + it}} \Big |^2 dt + O\left(\frac{1}{X^{10}}\right)
\end{align*}
since, by the definition of $\mathcal{T}_2$, we have $t + u \not \in \mathcal{U}$ whenever $t\in \mathcal{T}_2$ and $|u| \leq X^{\varepsilon / 3}$.

Now we are left with studying $I_{\mathcal{U}_2}$.
Recalling~\eqref{eq:FWsdecomp}, it equals, up to an error $O(X^{-10})$,
\begin{equation}
\label{eq:tuint}
\begin{split}
&\frac{1}{4\pi^2 i \cdot y_j} \sum_{A, B, C} \int_{-X^{\varepsilon/2}}^{X^{\varepsilon/2}} \widetilde{W}(i u) X^{iu} \int_{\mathcal{U}_2}  Q_{1,A}(1 + it + i u) Q_{2,B}(1 + it + iu) R_{C}(1 + it + iu)  du \\
&\qquad  \qquad \cdot \frac{(x + y_1)^{1+it} - x^{1+it}}{1+it} dt \  + \frac{1}{y_1} \cdot \frac{1}{2\pi i} \int_{\mathcal{U}_2} F_\square(1 + it) \cdot \frac{(x + y_1)^{1+it} - x^{1+it}}{1+it} dt,
\end{split}
\end{equation}
where $F_\square(s)$ is as in \eqref{eq:Fsquaredef} and $A, B, C$ traverse powers of two such that $X/16 < ABC \leq 4X$ and $P/2 < A \leq Q$ and $P'/2 < B \leq Q'$. Let us first deal with the second term. By Lemma~\ref{le:Parseval}
\[
\begin{split}
&\frac{1}{X}\int_X^{2X} \Big|\frac{1}{y_1} \cdot \frac{1}{2\pi i} \int_{\mathcal{U}_2} F_\square(1+it) \cdot \frac{(x + y_1)^{1+it} - x^{1+it}}{1+it} dt \Big|^2 dx \\
&\ll \max_{X/y_1 \leq T \leq X} \frac{X/y_1}{T} \int_{-T}^{T} |F_\square(1+it)|^2 dt.
\end{split}
\]
Arguing as in the end of proof of Lemma~\ref{lem:decomp}, this is
\[
\ll \frac{1}{y_1 P'} \prod_{p \leq X} \Big(1+\frac{|f(p)|^2-1}{p}\Big)  + \frac{1}{P'} \prod_{p \leq X} \Big(1+\frac{2|f(p)|-2}{p}\Big) \ll \frac{1}{X^{\nu_1}}.
\] 
Hence the term involving $F_\square(s)$ can be included to $A(x, y_1, y_2, t_0, T_0)$

Now we can concentrate on the first term in~\eqref{eq:tuint}. In the region where $|t+u-t_0| < T_0/2$, we necessarily have $|u| > T_0/2$, and in this region we can use fast decay of $\widetilde{W}$, trivial estimates for $Q_{1,A}(s)$ and $Q_{1,B}(s)$ and Lemma~\ref{le:Shiu} for $R_C(s)$. This way we obtain that the contribution of the region $|t+u-t_0| < T_0/2$ is
\[
\ll \frac{\Big ( 1 + \log \frac{\log Q}{\log P'}\Big )^2}{T_0^{10}} \prod_{p \leq X} \Big(1 + \frac{|f(p)|-1}{p}\Big).
\]

Hence we can restrict our attention to
\[
t+u \in \mathcal{U}^\ast := \{t \colon |t-t_0| > T_0/2\} \cap (\mathcal{U} + [-2 X^{\eps/2}, 2X^{\eps/2}])
\]
We write
\[
\mathcal{T}^\ast = \{t \in \mathcal{U}^\ast \colon \max_B |Q_{2,B}(1+it)| \leq X^{-\nu_1^3/320} \},
\]
where $B$ runs through powers of $2$ in $(P'/2, Q']$.
We split the integral in~\eqref{eq:tuint} according to whether $t + u \in \mathcal{T}^\ast$ or $t+u \in \mathcal{U}^\ast \setminus \mathcal{T}^\ast$. In the first case we can use the point-wise bound in the definition of $\mathcal{T}^\ast$ together with a Dirichlet polynomial large value result (Lemma~\ref{le:Hallargevalint}). The second case is very rare and in that case we use a point-wise bound for $R_C(s)$ and another large value result (Lemma~\ref{lem:primeshalasz}) which is applicable thanks to the rareness of large values of $Q_{2,B}(s)$.

For $t+u \in \mathcal{T}^\ast$, Cauchy-Schwarz and Lemma~\ref{le:Parseval} give
\[
\begin{split}
&\frac{1}{X}\int_{X}^{2X} \Big|\frac{1}{y_j} \sum_{A, B, C} \int_{-X^{\varepsilon/2}}^{X^{\varepsilon/2}} \widetilde{W}(i u) X^{iu} \int_{\substack{t \in \mathcal{U}_2 \\ t+u \in \mathcal{T}^\ast}} Q_{1, A}(1 + it + i u) Q_{2, B}(1 + it + iu) R_{C}(1 + it + iu)  du \\
&\qquad  \qquad \cdot \frac{(x + y_1)^{1+it} - x^{1+it}}{1+it} dt \Big|^2 dx \\
&\ll \Big(1 + \log \frac{Q}{P'}\Big)^4 \max_{\substack{A, B, C \\ X/y_1 \leq T \leq X}} \frac{X/y_1}{T} \int_{\mathcal{T}^\ast \cap [-T, T]} \Big|Q_{1,A}(1 + it) Q_{2,B}(1 + it) R_{C}(1 + it)\Big|^2 dt \\
&\ll (\log X)^4 X^{-\nu_1^3/160} \max_{\substack{A, C \\ X/y_1 \leq T \leq X}} \frac{X/y_1}{T} \int_{\mathcal{T}^\ast \cap [-T, T]} |Q_{1, A}(1+it) R_{C}(1+it)|^2 dt.
\end{split}
\]
Recalling that $\mathcal{T}^\ast$ is contained in $\mathcal{U}^\ast$, we get, discretising and applying Lemma~\ref{le:Hallargevalint}
\[
\int_{\mathcal{T}^\ast \cap [-T, T]} |Q_{1, A}(1+it) R_{C}(1+it)|^2 dt \ll \Big(AC + T^{1/2} \cdot |(\mathcal{U} + [-X^\varepsilon, X^\varepsilon]) \cap [-T, T]| \Big) \frac{\log X}{AC}.
\]
Since $AC \geq X/(16Q')$, those $t+u \in \mathcal{T}^\ast$ give an acceptable contribution.

Since
\begin{equation}
\label{eq:(x+y)it-xit}
\Big|\frac{(x+y)^{1+it}-x^{1+it}}{1+it}\Big| \ll \min \Big\{ \frac{x}{1+|t|}, y \Big\},
\end{equation}
the contribution of $t+u \in \mathcal{U}^\ast \setminus \mathcal{T}^\ast$ to the first term in~(\ref{eq:tuint}) is
\[
\begin{split}
&\ll \max_{X/y_1 \leq T \leq X/2} \frac{x/y_1}{T} \sum_{A, B, C} \int_{(\mathcal{U}^\ast \setminus \mathcal{T}^\ast) \cap [-T, T]} \Big|Q_{1,A}(1 + it) Q_{2,B}(1 + it) R_{C}(1 + it)\Big| dt \\
&\ll (\sup_{\substack{|t| \leq X/2, |t-t_0| > T_0/2 \\ X/(16Q Q') < C \leq 16X/(P P')}} |R_C(1+it)| ) \cdot  \sum_{A, B} \int_{\mathcal{U}^\ast \setminus \mathcal{T}^\ast} |Q_{1,A}(1 + it) Q_{2,B}(1 + it)| dt.
\end{split}
\]
Here, for certain one-spaced $\mathcal{W}^\ast \subset \mathcal{U}^\ast \setminus \mathcal{T}^\ast$,
\[
\begin{split}
&\sum_{A, B} \int_{\mathcal{U}^\ast \setminus \mathcal{T}^\ast} |Q_{1,A}(1 + it) Q_{2,B}(1 + it)| dt \ll \sum_{A, B} \Big(\sum_{t \in \mathcal{W}^\ast} |Q_{1,A}(1 + it)|^2 \sum_{t \in \mathcal{W}^\ast} |Q_{2,B}(1 + it)|^2 \Big)^{1/2} \\
&\ll \Big ( 1 + \log \frac{\log Q}{\log P'} \Big ) \cdot \Big(\sum_A \log A \sum_{t \in \mathcal{W}^\ast} |Q_{1,A}(1 + it)|^2 \sum_B \log B \sum_{t \in \mathcal{W}^\ast} |Q_{2,B}(1 + it)|^2\Big)^{1/2}.
\end{split}
\]
Applying Lemma~\ref{lem:primeshalasz}, we see that, for any $\eta \in (0, 1/2)$,
\[
\begin{split}
X^{-\nu_1^3/160} |\mathcal{W}^\ast| &\leq \sum_B \log B \sum_{t \in \mathcal{W}^\ast} |Q_{2,B}(1+it)|^2 \ll \sum_B \Big(\frac{1}{\log B} + |\mathcal{W}^\ast| X^{\frac{9}{2} \eta^{3/2}} (\log X)^2 B^{-\frac{99}{100}\eta}\Big) \\
&\ll 1 + \log \frac{\log Q}{\log P'}  + |\mathcal{W}^\ast| X^{\frac{9}{2} \eta^{3/2}} (\log X)^3 \cdot P'^{-\frac{99}{100} \eta} 
\end{split}
\]
Taking $\eta = \nu_1^2/45$, the first two terms must dominate. In particular 
\[
|\mathcal{W}^\ast| \ll \left(1 + \log \frac{\log Q}{\log P'}\right) X^{\nu_1^3/160}.
\]
As a result with this choice of $\eta$  
\[
  \sum_{B} \log B \sum_{t \in \mathcal{W}^{\ast}} |Q_{2,B}(1 + it)|^2 \ll 1 + \log \frac{\log Q}{\log P'} \]
and
\[
\begin{split}
\sum_A \log A \sum_{t \in \mathcal{W}^\ast} |Q_{1,A}(1+it)|^2 &\ll 1 + \log \frac{\log Q}{\log P'}  + |\mathcal{W}^\ast| X^{\frac{9}{2} \eta^{3/2}} (\log X)^3 \cdot P^{-\frac{99}{100} \eta} \\
& \ll 1 + \log \frac{\log Q}{\log P'}.
\end{split}
\]

Hence we obtain
\[
\sum_{A, B} \int_{\mathcal{U}^\ast \setminus \mathcal{T}^\ast} |Q_{A}(1 + it) Q_{B}(1 + it)| dt  \ll \Big ( 1 + \log \frac{\log Q}{\log P'} \Big )^2
\]
and so the total contribution of $t+u \in \mathcal{U}^\ast \setminus \mathcal{T}^\ast$ is
\[
\ll \Big ( 1 + \log \frac{\log Q}{\log P'} \Big )^2 \cdot \sup_{\substack{|t| \leq X/2, |t-t_0| > T_0/2 \\ X/(16(Q Q')) < C \leq 16X/(PP')}} |R_C(1+it)|.
\]
\end{proof}

\section{Theorem inside a set $\mathcal{S}$}
\label{sec:RedSM*}
We will deduce Theorem~\ref{th:MT} from a variant where $n$ is restricted to $\mathcal{S} \subset \mathbb{N}$ for which $\mathcal{S} \cap (X, 2X]$ is dense and which contains only those $n$ which have prime divisors from certain convenient ranges.

To define $\mathcal{S}$, we need to introduce some notation following~\cite[Section 2]{MainPaper}.
Let $1/6 > \nu_2 > \nu_1 > (\log X)^{-1/10}$. Let $\eta \in (0, 1/6-\nu_2/3)$. Consider a sequence of increasing intervals $(P_j, Q_j]$ with $j = 1, \dotsc, J + 2$ with $J \geq 1$, such that
\begin{itemize}
\item $P_{J + 1} = X^{\nu_1}$, $Q_{J+1} = P_{J+2} = X^{\sqrt{\nu_1 \nu_2}}$ and $Q_{J + 2} = X^{\nu_2}$, and $P_1 \geq 3/2$.
\item $P_{J}$ is large enough and $Q_{J}$ is small enough, precisely
\begin{equation}
\label{eq:PJQJsize}
P_J \geq (\log X)^{2/\eta} \quad \text{and either $J=1$ or $Q_J \leq \exp((\log X)^{1/2})$};
\end{equation}
\item The intervals are not too far from each other, precisely
\begin{equation}
\label{eq:PjQjnottoofar}
\frac{\log\log Q_{j}}{\log P_{j-1}-1} \leq \frac{\eta}{4j^2} \quad \text{for $j = 2, \dotsc, J$.}
\end{equation}
\item The intervals are not too close to each other, precisely
\begin{equation}
\label{eq:PjQjnottooclose}
\frac{\eta}{j^2} \log P_{j} \geq 16 \log Q_{j-1} + 16 \log j \quad \text{for $j = 2, \dotsc, J$.}
\end{equation}
\end{itemize}
Note that~\eqref{eq:PjQjnottooclose} implies that necessarily $J \ll \log \log X / \log \log \log X$.

For example, given $(\log X)^{-1/10} < \nu_1 < \nu_2 < 1 / 6$, $\eta \in (0, 1/6-\nu_2/3)$ and $\beta_0 \in (0, 1]$,
choose first any $Q_1$ and $P_1$ with $X^{1/6} \geq Q_1 \geq P_1 \geq (\log Q_1)^{40/\eta}$ large enough in terms of $\eta$ and $\beta_0$. Then if $Q_1 \geq \exp((\log X)^{1/2})$, take $J = 1$ and otherwise choose $P_j$ and $Q_j$ for $j= 2, \dotsc, J$ to be
\begin{equation}
\label{eq:PjQjchoice}
P_j = \exp(j^{8j/\beta_0} (\log Q_1)^{j-1} \log P_1) \quad \text{and} \quad
Q_j = \exp(j^{(8j + 6)/\beta_0} (\log Q_1)^j)
\end{equation}
with $J$ the largest index $j$ such that $Q_j \leq \exp((\log X)^{1/2}).$

Let $\mathcal{S} = \mathcal{S}_X$ be the set of integers $\leq 100X$ having at least one prime factor in each of the intervals $(P_j, Q_j]$ for $j \leq J+2$.

In order to apply results on multiplicative functions to sums with the additional restriction $n \in \mathcal{S}$, we use the following immediate consequence of the inclusion-exclusion principle (which is also~\cite[Lemma 5]{MainPaper}).
\begin{lemma}
\label{le:Sinclexcl}
For $\mathcal{J} \subseteq \{1, \dotsc, {J+2}\}$, let $g_{\mathcal{J}}$ be the completely multiplicative function
\[
g_{\mathcal{J}}(p^j) =
\begin{cases}
1 & \text{if $p \not \in \bigcup_{j \in \mathcal{J}} (P_j, Q_j]$} \\
0 & \text{otherwise}.
\end{cases}
\]
Then, for any complex numbers $a_n$,
\[
\sum_{\substack{X <  n \leq 2X \\n \in \mathcal{S}}} a_n = \sum_{\substack{X < n \leq 2 X}} a_n \prod_{j=1}^{J+2} (1-g_{\{j\}}(n)) = \sum_{\mathcal{J} \subseteq \{1, \dotsc, {J+2}\}} (-1)^{\# \mathcal{J}} \sum_{X  < n \leq 2X} g_{\mathcal{J}}(n) a_n.
\]
\end{lemma}

We shall prove the following variant of~\cite[Theorem 3]{MainPaper}.
\begin{theorem}
\label{th:ThminS}
Let $X \geq 2$ and let $\mathcal{S}$ be as above with $\eta \in (0, 1 / 6 - \nu_2/3)$. Let $2 \leq h_0 \leq X^{1/2}$ and assume that $(P_1, Q_1] \subset (3/2,h_0]$. Write $h_1 = H(f; X)$. Let $f: \mathbb{N} \to \mathbb{U}$ be a multiplicative function. 
\begin{enumerate}[(i)]
\item
Assume that $f$ is $(\alpha, X^\theta)$-non-vanishing for some $\alpha, \theta \in (0, 1]$, and let $0 < \rho < \rho_\alpha$. Then, for $\delta \in (0, 1/1000)$,
\[
\begin{split}
&\Big | \frac{1}{h_0 h_1} \sum_{\substack{x < n \leq x + h_0 h_1 \\ n \in \mathcal{S}}} f(n) - \frac{1}{h_0 h_1} \int_x^{x+h_0 h_1} u^{i \widehat{t}_{f, X}} du \cdot \frac{1}{X} \sum_{\substack{X < n \leq 2X \\ n \in \mathcal{S}}} f(n) n^{-i\widehat{t}_{f, X}} \Big |  \\
& \qquad \leq \left(\delta + \frac{1}{\nu_1 (\log X)^{\rho/12}}\right) \prod_{p \leq X} \Big ( 1 + \frac{|f(p)| - 1}{p} \Big )
\end{split}
\]
for all but at most
\[
\ll_{\eta, \rho, \theta} \frac{X}{\delta^2} \cdot \frac{(\log h_0)^2}{P_1^{1 / 2 - \nu_2 - 3\eta}} + X^{1-\nu_1^3/200}
\]
integers $x \in [X, 2X]$.
Moreover, if
\begin{equation} \label{eq:cancel}
\sum_{\substack{u < p \leq v}} \frac{f(p)}{p^{1/2 + it}} \ll_{\varepsilon} \sum_{p \leq v} \frac{|f(p)|}{p} \cdot \frac{1}{1 + |t|} + u^{-1/2 + \varepsilon}
\end{equation}
for all $\varepsilon > 0$ and 
for all $(u,v] \subset (P_J, Q_J]$, then the exceptional set is
$$
\ll_{\eta, \rho, \theta} \frac{X}{\delta^2} \cdot \frac{(\log h_0)^2}{P_1^{1 - 3 \eta}} + X^{1-\nu_1^{3}/200}. 
$$
\item Let $0 < \rho < \rho_1$. Then, for $\delta \in (0, 1/1000)$, 
\begin{align*}
&\Big | \frac{1}{h_0} \sum_{\substack{x < n \leq x + h_0 \\ n \in \mathcal{S}}} f(n) - \frac{1}{h_0} \int_x^{x+h_0} u^{i t_{f, X}} du \cdot \frac{1}{X} \sum_{\substack{X < n \leq 2X \\ n \in \mathcal{S}}} f(n) n^{-it_{f, X}} \Big | \leq \delta+ \frac{1}{\nu_1 (\log X)^{\rho/12}}
\end{align*}
for all but
\[
\ll_{\eta, \rho} \frac{X}{\delta^2} \cdot \frac{(\log h_0)^2}{P_1^{1 / 2 - \nu_2 - 3\eta}} + X^{1-\nu_1^3/200}
\]
integers $x \in [X, 2X]$.
\item If $f$ is almost real-valued, (i) and (ii) also hold with $\widehat{t}_{f, X}$ and $t_{f, X}$ replaced by $0$. In this case the main terms become simply
\[
\frac{1}{X} \sum_{\substack{X < n \leq 2X \\ n \in \mathcal{S}}} f(n)
\]
\end{enumerate}
\end{theorem}

\begin{remark} 
Note that if $|\widehat{t}_{f, X}| > X/2$, then 
\[
\frac{1}{h_0 h_1} \int_x^{x+h_0 h_1} u^{i \widehat{t}_{f, X}} du \ll \frac{1}{h_0h_1},
\]
and using also Shiu's bound (Lemma~\ref{le:Shiu}), the main term in Theorem~\ref{th:ThminS}(i) becomes essentially negligible. On the other hand, for $|\widehat{t}_{f, X}| \leq X/2$, Hal\'asz's theorem (Lemma~\ref{le:SparseHalaszComplex} together with Lemma~\ref{le:Sinclexcl}) implies that the main term has size
\[
\ll 2^J \left(\frac{\widehat{M}(f; X)}{\exp(\frac{1}{2}\widehat{M}(f; X))} + \frac{1}{(\log X)^\alpha}\right) \prod_{p \leq X} \Big ( 1 + \frac{|f(p)| - 1}{p} \Big ).
\]
This is $o(\prod_{p \leq X} ( 1 + \frac{|f(p)| - 1}{p} ) )$ if $\widehat{M}(f; X)$ tends to infinity faster than $2J$. In case $\widehat{M}(f; X)$ tends to infinity but more slowly, it seems to be more complicated to show in general that the main term is small. However, even in this case Corollary~\ref{cor:convenient} provides a non-trivial bound without a main term. Similar remarks apply to part (ii).
\end{remark}

\begin{proof}[Proof of Theorem~\ref{th:ThminS}]
Let us first prove the case (i) and then discuss the differences in the cases (ii)--(iii) at the end. Let $\rho' = (\rho + \rho_\alpha)/2$. We shall apply Proposition~\ref{prop:L1work} with $\eps = \eta/2$, $t_0 = \widehat{t}_{f, X}$, $T_0 = (\log X)^{\rho'/6}, y_1 = h_0 h_1, y_2 = X/(\log X)^{5\rho'/12}$, $P' = P_{J+1}, Q' = Q_{J+1}$, $P = P_{J+2},$ $Q = Q_{J+2},$ $a_m = f(m) \mathbf{1}_{m \in \mathcal{S}}$, $b_m = f(m) \mathbf{1}_{m \in \mathcal{S}^\ast}$ and $c_p = f(p)$. Here $\mathcal{S}^\ast$ is the set of integers that have prime factor in each interval $(P_j, Q_j]$ for $j = 1, \dotsc, J$. We shall choose $d_n$ and $\mathcal{U}$ later. The set $\mathcal{U}$ will be chosen so that, for any $T \in [X/H, X]$, one has 
\begin{equation}
\label{eq:UBound}
|(\mathcal{U} + [-X^{\eta/2}, X^{\eta/2}]) \cap [-T, T]| \ll \frac{h_0 h_1 \sqrt{T}}{Q_{J+2}}
\end{equation}
which guarantees that the corresponding term in~\eqref{eq:A2bound} is of acceptable size.

Let us first show that the main term in Proposition~\ref{prop:L1work} corresponds to the main term in Theorem~\ref{th:ThminS}. With $g_{\mathcal{J}}$ as in Lemma~\ref{le:Sinclexcl}, we have by Lemma~\ref{le:Sinclexcl}, for $y \in \{X, y_2\}$,
\[
\frac{1}{y} \sum_{\substack{X < n \leq X + y \\ n \in \mathcal{S}}} f(n)n^{-i\widehat{t}_{f, X}} = \frac{1}{y} \sum_{\mathcal{J} \subseteq \{1, \dotsc, J+2\}} (-1)^{\# \mathcal{J}} \sum_{X < n \leq X + y} g_{\mathcal{J}}(n) f(n) n^{-i\widehat{t}_{f, X}}.
\]
Applying Lemma~\ref{le:Lipschitz}(i) to the sum over $n$ on the right hand side, we see that
\begin{align*}
\frac{1}{X} \sum_{\substack{X < n \leq 2X}} f(n) g_{\mathcal{J}}(n) & n^{-i\widehat{t}_{f, X}}
= \frac{1}{y_2} \sum_{\substack{x < n \leq x+y_2}} f(n) g_{\mathcal{J}}(n) n^{-i\widehat{t}_{f, X}} \\ & + O\Big(\frac{1}{(\log X)^{\rho'/12}} \prod_{\substack{p \leq X}} \Big ( 1 + \frac{|f(p)| - 1}{p} \Big )\Big)
\end{align*}
for all $x \in [X, 2X]$. Summing back over subsets $\mathcal{J} \subset \{1, \ldots, J + 2\}$ weighted by $(-1)^{\# \mathcal{J}}$ we conclude that,
$$
\frac{1}{X} \sum_{\substack{X < n \leq 2X \\ n \in \mathcal{S}}} f(n) n^{- i \widehat{t}_{f,X}} = \frac{1}{y_2} \sum_{\substack{x < n \leq x + y_2 \\ n \in \mathcal{S}}} f(n) n^{- i \widehat{t}_{f, X}} + O \Big ( \frac{2^J}{(\log X)^{\rho' / 12}} \prod_{p \leq X} \Big ( 1 + \frac{|f(p)| - 1}{p} \Big ) \Big ). 
$$

Letting $R_C(s)$ be as in~\eqref{eq:WQRdef}, we get, using Lemmas~\ref{le:Halappl}(i) and~\ref{le:Sinclexcl} in similar fashion, that
\[
\begin{split}
&\sup_{\substack{|t| \leq X/2, |t-\widehat{t}_{f, X}| > T_0/2 \\ X/(16Q_{J+1} Q_{J+2}) < C \leq 16 X/(P_{J+1} P_{J+2})}} |R_C(1+it)| \\
&\ll \frac{2^J}{\log P_{J+1}} \Big(\frac{1}{\log^{\rho'/2} X}+ \frac{(\log \log X)^2}{T_0^{1/2}}\Big) \prod_{p \leq X} \Big(1+\frac{|f(p)|}{p}\Big) \\
&\ll \frac{2^J (\log\log X)^2}{\nu_1 (\log X)^{\rho'/12}} \cdot \prod_{p \leq X} \Big(1+\frac{|f(p)|-1}{p}\Big).
\end{split}
\]
Recall that $J \ll \log \log X/\log \log \log X$, $\rho' > \rho$ and that we can assume that $X$ is large in terms of $\rho$. Combining the above, Proposition~\ref{prop:L1work} implies that, for $x \in (X, 2X]$,
\[
\begin{split}
&\Bigl| \frac{1}{h_0 h_1} \sum_{\substack{x < n \leq x + h_0 h_1 \\ n \in \mathcal{S}}} f(n)  - \frac{1}{h_0 h_1} \int_x^{x+h_0 h_1} u^{i \widehat{t}_{f, X}} du \cdot \frac{1}{X} \sum_{\substack{X < n \leq 2X \\ n \in \mathcal{S}}} f(n) n^{-i\widehat{t}_{f, X}}  \\
&\qquad \qquad \qquad \qquad \qquad \qquad \qquad - A(x, h_0h_1, X/(\log X)^{5\rho'/12}, \widehat{t}_{f, X}, T_0, \mathcal{U})\Bigr| \\
&\leq\left(\frac{1}{2\nu_1 (\log X)^{\rho/12}} + O\left(\frac{1}{h_0 h_1}\right) \right) \cdot \prod_{p \leq X} \Big(1+\frac{|f(p)|-1}{p}\Big),
\end{split}
\]
where $A(x, y_1 , y_2, \widehat{t}_{f, X}, T_0, \mathcal{U})$ satisfies~\eqref{eq:A2bound}.

Now our claim is trivial unless $\delta \geq h_0^{-1/2}$, and so it suffices to show that
\[
|A(x, h_0h_1, X/(\log X)^{5\rho'/12}, \widehat{t}_{f, X}, T_0, \mathcal{U})| \leq \left(\frac{1}{2\nu_1 (\log X)^{\rho/12}} + \frac{\delta}{2}\right) \cdot \prod_{p \leq X} \Big(1+\frac{|f(p)|-1}{p}\Big)
\]
for all but at most $\ll \frac{X}{\delta^2}\mathcal{R}$ integers $x \in [X, 2X]$, where $\mathcal{R} = (\log h_0)^2/P_1^{1-3\eta}$ or $\mathcal{R} = (\log h_0)^2/P_1^{1/2-\nu_2-3\eta}$ depending on whether we assume~\eqref{eq:cancel} or not.

Thus by~\eqref{eq:A2bound} it suffices to show that, for some choice of $d_n$ and $\mathcal{U}$ such that $\mathcal{U}$ satisfies~\eqref{eq:UBound} and $d_n = f(n)\mathbf{1}_{n \in \mathcal{S}}$ for all $X/4 < n \leq 4X$, we have
\[
\int_{[-T, T] \setminus \mathcal{U}} \Big | \sum_{\substack{X / 100 < n \leq 100X}} \frac{d_n}{n^{1 + it}} \Big |^2 dt \ll \frac{T}{X/(h_0 h_1)} \prod_{p \leq X} \Big ( 1+ \frac{2 |f(p)| - 2}{p} \Big ) \mathcal{R} 
\]
whenever $X/(h_0 h_1) \leq T \leq X$. Writing
\[
\mathfrak{S}(X, T, f) := \frac{T}{X} \prod_{p \leq X} \Big (1 + \frac{|f(p)|^2 - 1}{p} \Big ) + \prod_{p \leq X} \Big ( 1+ \frac{2 |f(p)| - 2}{p} \Big )
\]
and recalling the definition of $h_1$ and that $Q_1 \leq h_0$, it thus suffices to show that
\begin{equation}
\label{eq:TL2intclaim}
\int_{[-T, T] \setminus \mathcal{U}} \Big | \sum_{\substack{X / 100 < n \leq 100X}} \frac{d_n}{n^{1 + it}} \Big |^2 dt \ll \mathfrak{S}(X/Q_1, T, f) \mathcal{R} 
\end{equation}
for all $X/(h_0 h_1) \leq T \leq X$.

We write $\mathcal{I}_1 := \{v \colon \lfloor \log P_1 / \log 2 \rfloor \leq v \leq \log Q_1 / \log 2\}$ and choose
\[
d_n := \sum_{v \in \mathcal{I}_1} \sum_{\substack{n = pm \\ 2^v < p \leq 2^{v+1} \\ P_1 < p \leq Q_1 \\ m \in \mathcal{S}_1 \\ X/2^{v+3} < m \leq 4X/2^v}} \frac{f(mp)}{\omega_{(P_1, Q_1]}(m) + \mathbf{1}_{(p, m) = 1}},
\]
where $\mathcal{S}_1$ is the set of those $n$ that have at least one prime factor in each of intervals $(P_j, Q_j]$ for $j =2, \dotsc, J+2$. Now $d_n = f(n) \mathbf{1}_{n \in \mathcal{S}}$ for all $n \in (X/4, 4X]$ as requested.

We have
\begin{equation}
\label{eq:dnsumcn}
d_n = \sum_{v \in \mathcal{I}_1} \sum_{\substack{n = mp \\ 2^v < p \leq 2^{v+1} \\ P_1 < p \leq Q_1 \\ m \in \mathcal{S}_1 \\ X/2^{v+3} < m \leq 4X/2^v}} \frac{f(m)f(p)}{\omega_{(P_1, Q_1]}(m) + 1} + \sum_{\substack{n = mp \\ P_1 < p \leq Q_1 \\ m \in \mathcal{S}_1, \, (m, p) = p}} e_{m,p},
\end{equation}
where
\[
|e_{m,p}| \leq |f(mp)| + |f(m)||f(p)|.
\]
The contribution of $e_{m,p}$ to $\mathcal{R}$ in \eqref{eq:TL2intclaim} is acceptable (of the order of $1/P_1$) as in end of proof of Lemma~\ref{lem:decomp}.

To handle the rest of~(\ref{eq:TL2intclaim}) and define the set $\mathcal{U}$ we need to introduce some additional notation following~\cite[Proof of Theorem 3]{MainPaper}.
Pick a sequence $\alpha_j$ for $1 \leq j \leq J$ differently depending on whether \eqref{eq:cancel} holds or not. If \eqref{eq:cancel} does not hold, choose
\begin{equation}
\label{eq:alphajdef}
\alpha_j = \frac{1}{4} - \frac{\nu_2}{2} - \eta\Big(1+ \frac{1}{2j}\Big).
\end{equation}
Notice that
\[
\frac{1}{4} - \frac{\nu_2}{2} -\frac{3}{2}\eta = \alpha_1 < \alpha_2 < \dotsc < \alpha_J \leq \frac{1}{4} - \frac{\nu_2}{2}  -\eta.
\]
If \eqref{eq:cancel} holds then pick
$$
\alpha_j = \frac{1}{2} - \eta \Big ( 1 + \frac{1}{2j} \Big ) .
$$
Let
\begin{equation}
\label{eq:Hjdef}
H_1 = 1/\log 2, \quad \text{and for} \quad j = 2, \dotsc, J, \quad \text{let} \quad  H_j := j^2 P_1.
\end{equation}
For $j = 1, \dotsc, J$, let
\[
Q_{v, j}(s) := \sum_{\substack{e^{v/H_j} < p \leq e^{(v+1)/H_j} \\ P_j < p \leq Q_j}} \frac{f(p)}{p^s}.
 \]
Notice that this can be non-zero only when
\[
v \in \mathcal{I}_j := \{v: \lfloor H_j \log P_j \rfloor \leq v \leq H_j \log Q_j\}
\]
We write
\[
[-T, T] = \bigcup_{j = 1}^{J} \mathcal{T}_j \cup \mathcal{U}
\]
as a disjoint union where $t \in \mathcal{T}_j$ when $j$ is the smallest index such that
\[
\text{for all } v \in \mathcal{I}_j: |Q_{v,j}(1+it)| \leq e^{-\alpha_j v / H_j}
\]
and $t \in \mathcal{U}$ if this does not hold for any $j \leq J$.

Let us first show that $\mathcal{U}$ satisfies~\eqref{eq:UBound}. Note that if \eqref{eq:cancel} holds then $\mathcal{U} \subset [- Q_{J}^{1/2}, Q_{J}^{1/2}]$ and $Q_{J} \leq Q_{J + 2} \leq X^{1/6}$. As a result \eqref{eq:UBound} is automatically satisfied. On the other hand if \eqref{eq:cancel} does not hold then it suffices to show that for any $X/(h_0 h_1) \leq T \leq X$ and any one-spaced subset $\mathcal{W} \subset \mathcal{U} \cap [-T, T]$, we have $|\mathcal{W}| \ll T^{1/2} h_0 h_1/ X^{\nu_2+\eta/2}$. For each $t \in \mathcal{W}$ there exists $v \in \mathcal{I}_{J}$ such that $|Q_{v,J}(s)| > e^{-\alpha_J v / H_J}$ . By~ \eqref{eq:PJQJsize} we have $Q_J \leq X^{o(1)} h_0$, so applying~\cite[Lemma 8]{MainPaper} to $Q_{v, J}(s)$ for every $v \in \mathcal{I}_J$ we get
\[
\begin{split}
|\mathcal{W}| &\ll |\mathcal{I}_J| \cdot T^{2 \alpha_{J}} Q_J^{2\alpha_J} X^{\eta} \ll (T h_0 h_1)^{2 \alpha_{J}}  X^{\eta + o(1)} \\
&\ll (T h_0 h_1)^{1/2} (T h_0 h_1)^{-\nu_2 - 2\eta} X^{\eta+o(1)} \ll T^{1/2} (h_0 h_1)^{1/2} X^{-\nu_2 - \eta+o(1)}
\end{split}
\]
which is sufficient. Hence our choice for $\mathcal{U}$ is legitimate and it suffices to show~\eqref{eq:TL2intclaim}.

Write
$$
R_{v, 1}(s) = \sum_{\substack{X/(8e^{v/H_1}) < m \leq 4 X /e^{v/H_1} \\ m \in \mathcal{S}_1}} \frac{f(m)}{m^s} \cdot \frac{1}{\omega_{(P_1, Q_1]}(m) +1}.
$$
Recalling~\eqref{eq:dnsumcn} and that the contribution from $e_n$ is acceptable, it suffices to show that
\begin{equation}
\label{eq:QvRvmeanclaim}
\max_{v \in \mathcal{I}_1} \int_{[-T, T] \setminus \mathcal{U}} \Big |Q_{v, 1}(1+it) R_{v, 1}(1+it) \Big |^2 dt \ll \mathfrak{S}(X/Q_1, T, f) P_1^{-2\alpha_1}.
\end{equation} 

Our argument is similar to that in \cite[Proof of Theorem 3]{MainPaper} though we have taken out the smallest prime in $p \in (P_1, Q_1]$ earlier to get better bounds (because this way we avoid splitting $p \in (P_1, Q_1]$ into very short intervals) and use refined estimates for Dirichlet polynomias to take into account the average size of $f(n)$.

Let us first consider the contribution of $\mathcal{T}_1$. We have
\[
\begin{split}
\max_{v \in \mathcal{I}_1} \int_{\mathcal{T}_1} \Big | Q_{v, 1}(1+it) R_{v, 1}(1+it)\Big|^2 dt \ll P_1^{-2\alpha_1} \cdot \mathfrak{S}(X/Q_1, T, f)
\end{split}
\]
by the definition of $\mathcal{T}_1$ and mean value theorem (Lemma~\ref{le:MVTwithHenriot}) which is acceptable.

Let us now consider the contribution of $\mathcal{T}_j$ to~\eqref{eq:QvRvmeanclaim} for $j \geq 2$. We use Lemma~ \ref{lem:decomp} with $H = H_j$, $P = P_j, Q = Q_j$,
\[
a_m = \frac{f(m) \mathbf{1}_{\mathcal{S}_1}}{\omega_{(P_1, Q_1]}(m) + 1}, \qquad \quad b_m = \frac{f(m) \mathbf{1}_{\mathcal{S}_{1, j}}}{\omega_{(P_1, Q_1]}(m) + 1}, \qquad \text{and} \qquad c_p = f(p),
\]
where $\mathcal{S}_{1, j}$ is the set of those integers which have at least one prime factor in every
interval $(P_i,Q_i]$ with $i \in \{2, \dotsc, J+2\} \setminus \{j\}$. Estimating $Q_{v, 1}(1+it)$ trivially, we see that
\[
\begin{split}
&\max_{v \in \mathcal{I}_1} \int_{\mathcal{T}_j} \Big |Q_{v, 1}(1+it) R_{v, 1}(1+it) \Big |^2 dt \ll \Big(\frac{1}{H_j} + \frac{1}{P_j}\Big)\mathfrak{S}(X/Q_1, T, f) \\
& \qquad + (H_j \log Q_j)^2 \max_{v_1 \in \mathcal{I}_1, v \in \mathcal{I}_j} \int_{\mathcal{T}_j} |Q_{v,j}(1 + it)R_{v_1, v,j}(1 + it)|^2 dt,
\end{split}
\]
where
$$
R_{v_1, v, j}(s) = \sum_{\substack{X e^{-v_1/H_1-v/H_j}/8 < m \leq 4 X e^{-v_1/H_1-v/H_j} \\ m \in \mathcal{S}_{1,j}}} \frac{f(m)}{m^s} \cdot \frac{1}{(\omega_{(P_j, Q_j]}(m) +1)(\omega_{(P_1, Q_1]}(m) +1)}.
$$

Here, by~\eqref{eq:PjQjnottooclose} and~\eqref{eq:Hjdef},
\[
\sum_{j = 2}^J \Big(\frac{1}{H_j} + \frac{1}{P_j}\Big) \mathfrak{S}(X/Q_1, T, f) \ll \mathfrak{S}(X/Q_1, T, f) \frac{1}{P_1}.
\]
Thus the claim follows once we have shown
\begin{equation}
\label{eq:Ejsumclaim}
\sum_{2 \leq j \leq J} E_j \ll \frac{1}{P_1} \mathfrak{S}(X/Q_1, T, f),
\end{equation}
where, for $2 \leq j \leq J$,
\[
\begin{split}
E_j &:= (H_j \log Q_j)^2 \cdot \max_{\substack{v_1 \in \mathcal{I}_1 \\ v \in \mathcal{I}_j}} \int_{\mathcal{T}_j} |Q_{v, j}(1+it) R_{v_1, v, j}(1+it)|^2 dt \\
&\ll (H_j \log Q_j)^2 \max_{\substack{v_1 \in \mathcal{I}_1 \\ v \in \mathcal{I}_j}} e^{-2\alpha_j v / H_j} \int_{\mathcal{T}_j} |R_{v_1, v, j}(1+it)|^2 dt.
\end{split}
\]

Now we split further
\begin{equation}
\label{eq:Tjsplit}
\mathcal{T}_j = \bigcup_{r \in \mathcal{I}_{j-1}} \mathcal{T}_{j, r},
\end{equation}
where
$$
\mathcal{T}_{j, r} = \{t \in \mathcal{T}_j \colon |Q_{r,j-1}(1 + it)| > e^{-\alpha_{j-1} r/ H_{j-1}} \}
$$
Note that~\eqref{eq:Tjsplit} indeed holds since, by the definition of $\mathcal{T}_j$, for any $t \in \mathcal{T}_j$ there exists an index $r \in \mathcal{I}_{j-1}$ such that $|Q_{r, j-1}(1 + it)| > e^{-\alpha_{j-1} r / H_{j-1}}$. Therefore, for some $v_1 = v_1(j) \in \mathcal{I}_1$, $v = v(j) \in \mathcal{I}_j$ and $r = r(j) \in \mathcal{I}_{j-1}$,
\[
E_j \ll (H_j \log Q_j)^2 (H_{j-1} \log Q_{j-1}) \cdot e^{-2 \alpha_j v / H_j}
\cdot \int_{\mathcal{T}_{j,r}} |R_{v_1, v, j}(1 + it)|^2 dt
\]
On $\mathcal{T}_{j,r}$ we have $|Q_{r, j-1}(1 + it)| > e^{-\alpha_{j-1} r/H_{j-1}}$.
Therefore, for any $\ell_{j,r} \geq 1$, multiplying by the term $(|Q_{r, j-1}(1 + it)| e^{\alpha_{j-1} r/H_{j-1}})^{2\ell_{j,r}} \geq 1$, we get
\begin{align*}
E_{j} \ll &(H_j \log Q_j)^2 (H_{j-1} \log Q_{j-1}) \cdot e^{-2\alpha_j v/ H_j} \\ & \cdot  \exp \Big (2 \ell_{j,r} \cdot \alpha_{j-1} r  / H_{j-1} \Big )
\int_{\mathcal{T}_{j,r}} |Q_{r, j-1}(1 + it)^{\ell_{j,r}} R_{v_1,v,j}(1 + it)|^2 dt.
\end{align*}
Choosing
\[
\ell_{j,r} = \Big\lceil \frac{v/ H_{j}}{r/H_{j-1}}\Big \rceil \leq \frac{H_{j-1}}{r} \cdot \frac{v}{H_j} + 1,
\]
we get
\begin{equation}
\label{eq:Ejbound1}
\begin{split}
E_j &\ll H_j^3 (\log Q_j)^3 \cdot \exp \Big ( 2v (\alpha_{j-1} - \alpha_{j}) / H_{j} + 2 \alpha_{j-1} r / H_{j-1} \Big ) \\
& \quad \cdot\int_{-T}^{T} |Q_{r,j-1}(1 + it)^{\ell_{j,r}} R_{v_1,v,j}(1 + it)|^2 dt.
\end{split}
\end{equation}
Now we are in the position to use Lemma~\ref{le:moment} which gives
\begin{align*}
&\int_{-T}^{T}  |Q_{r,H_{j-1}}(1 + it)^{\ell_{j,r}}  R_{v_1,v,j}(1 + it)|^2  dt \\
&\ll \ell_{j,r}!^2 \Big(\frac{T}{X/e^{v_1/H_1}} \prod_{p \leq X} \Big(1+\frac{|f(p)|^2-1}{p}\Big) +  \prod_{p \leq X} \Big ( 1 + \frac{2|f(p)|-2}{p}\Big)\Big) \\
&\ll \exp(2\ell_{j,r} \log \ell_{j,r}) \mathfrak{S}(X/Q_1, T, f).
\end{align*}
Here by the mean value theorem and the definition of $\ell_{j,r}$
\[
\begin{split}
\ell_{j,r} \log \ell_{j,r} &\leq \frac{v/ H_{j}}{r/H_{j-1}} \log \frac{v/ H_{j}}{r/H_{j-1}} + \log \log Q_j + 1 \\
&\leq \frac{v}{H_j} \cdot \frac{\log \log Q_j}{\log P_{j-1}-1} + \log \log Q_j + 1,
\end{split}
\]
so that
\begin{equation}
\label{eq:EjintBound}
\begin{split}
&\int_{-T}^{T}  |Q_{r,j-1}(1 + it)^{\ell_{j,r}}  R_{v_1,v,j}(1 + it)|^2  dt \\
& \ll \exp \Big (2\frac{v}{H_j} \cdot \frac{\log \log Q_j}{\log P_{j-1}-1}\Big) (\log Q_j)^2  \mathfrak{S}(X/Q_1, T, f) \\
& \ll (\log Q_j)^2 \exp \Big (\frac{\eta}{2j^2} \cdot \frac{v}{H_j} \Big ) \mathfrak{S}(X/Q_1, T, f)
\end{split}
\end{equation}
by \eqref{eq:PjQjnottoofar}. Note that~\eqref{eq:PjQjnottoofar} also implies that, for $j \geq 2$,
\[
\log \log Q_j \leq \frac{1/6}{4 \cdot 2^2}\log P_{j-1} \leq \log Q_{j-1}^{1/96} \implies \log Q_j \leq Q_{j-1}^{1/96},
\]
so that we have the bound
\[
\begin{split}
H_j^3 (\log Q_j)^5 \exp(2\alpha_{j-1} r / H_{j-1})  &\ll H_j^3 (\log Q_j)^5 Q_{j-1} \\
&\ll H_j^3 Q_{j-1}^{3/2} \ll j^6 P_1^{3} Q_{j-1}^{3/2} \ll j^6 Q_{j-1}^{5}.
\end{split}
\]
Therefore, recalling also~\eqref{eq:Ejbound1} and~\eqref{eq:EjintBound} we see that, for $2 \leq j \leq J$,
\begin{align*}
E_j & \ll j^6 Q_{j-1}^5 \exp \Big ( \frac{2 v}{H_j} \Big (\alpha_{j-1} - \alpha_{j} + \frac{\eta}{4j^2} \Big ) \Big) \mathfrak{S}(X/Q_1, T, f).
\end{align*}
By~\eqref{eq:alphajdef} and~(\ref{eq:PjQjnottooclose}) we see that
\begin{align*} 
E_j & \ll j^6 Q_{j-1}^5 \exp \Big ( - \frac{\eta}{2j^2} \log P_j \Big ) \mathfrak{S}(X/Q_1, T, f) \\
&\ll  \frac{1}{j^2 Q_{j-1}} \mathfrak{S}(X/Q_1, T, f) \ll  \frac{1}{j^2 P_1} \mathfrak{S}(X/Q_1, T, f)
\end{align*}
This immediately implies~\eqref{eq:Ejsumclaim} and thus part (i) of the theorem.

Parts (ii) and (iii) follow similary, replacing applications of Lemmas~\ref{le:Lipschitz}(i) and~\ref{le:Halappl}(i) by parts (ii) and (iii) of those lemmas, and in case (ii) replacing applications of Lemma~\ref{le:MVTwithHenriot} by applications of~\eqref{eq:contMVT}. 
\end{proof}

\section{Sieve estimates}\label{se:sieve}
In order to deduce Theorems~\ref{th:MTDense} and~\ref{th:MT} from Theorem~\ref{th:ThminS} we need to show that, once the parameters defining the set $\mathcal{S}$ are appropriately selected, the numbers outside $\mathcal{S}$ make an acceptable contribution. In case of Theorem~\ref{th:MTDense} this is quite straight-forward and was essentially shown in~\cite{MainPaper}, but in case of Theorem~\ref{th:MT} we need to use some sophisticated sieve majorants in order to obtain power savings in the exceptional set. In this section we provide the needed sieve bounds.

When $g : \mathbb{N} \rightarrow [0, 1]$ is a multiplicative function, we define another multiplicative function $g^{\star}$ by requiring that $g^{\star}(p^\alpha) = 1 - g(p^\alpha)$. Note that, for any square-free $n$, $g^\ast(n) = \sum_{d \mid n} \mu(d) g(d)$.
Let us first state a result of Alladi that we use for constructing majorants from sieves. It follows immediately from~\cite[Corollary 1]{Alladi}.
\begin{lemma}
\label{le:Alladi}
Suppose $H\colon \mathbb{N} \to [0, 1]$ is multiplicative and suppose that $\mathcal{P} \subset \mathbb{P}$ and $\chi \colon \mathbb{N} \to [0, 1]$ are such that, for all $n \mid \prod_{p \in \mathcal{P}} p$, one has
\[
H^\ast(n) = \sum_{d \mid n} \mu(d) H(d) \leq \sum_{d \mid n} \mu(d) \chi(d) H(d).
\]
Then for any multiplicative $h \colon \mathbb{N} \to [0, 1]$ satisfying $h(n) \leq H(n)$ for all $n \mid \prod_{p \in \mathcal{P}} p$, one has, for all $n \mid \prod_{p \in \mathcal{P}} p$,
\[
h^\ast(n) = \sum_{d \mid n} \mu(d) h(d) \leq \sum_{d \mid n} \mu(d) \chi(d) h(d).
\]
\end{lemma}

Let us now define the Brun-Hooley-type sieve that we are going to use. Let $K$ be the largest integer for which\footnote{$\log_k x$ denotes the $k$-fold logarithm} $\log_{K} X > 10^{10}$. Let $I_1 = (1, X^{1/(\log\log X)^2} ]$ and, for $2 \leq k < K$, take $I_{k} = (X^{1/(\log_{k} X)^2} , X^{1/(\log_{k + 1} X)^2}]$, and let $I_{K} = (X^{1/(\log_{K} X)^2} , X^{1/7}]$ and $I_{K+1} = (X^{1/7}, 4X]$. For $1 \leq k \leq K-1$, let
\[
\chi_k(n) =
\begin{cases}
0 & \text{if $\omega_{I_k}(n) > 30 \lfloor \log_{k+1} X\rfloor $;} \\
1 & \text{otherwise,}
\end{cases}
\]
where $\omega_{I_k}(n) := \sum_{\substack{p | n \\ p \in I_k}} 1$. Furthermore, let $\chi_K(n) = 1_{S^+}(n)$, where $S^+$ is as in Lemma~\ref{le:linearsieve} with $D = X^{2/5-1/1000}$, $\mathcal{P} = \mathbb{P} \cap I_K$ and $z = X^{1/7}$, and let
\[
\chi_{K+1}(n) =
\begin{cases}
0 & \text{if $\omega_{I_{K+1}}(n) \geq 1$} \\
1 & \text{otherwise.}
\end{cases}
\]
Write $P(I_k) = \prod_{p \in I_k} p$, and define
\begin{equation}
\label{eq:chi(n)def}
\chi(n) = \prod_{k=1}^{K+1} \chi_k(n).
\end{equation}
Note that we use the linear sieve for primes in $I_K$ just to get a numerically better result.

Now according to Lemma~\ref{le:linearsieve}(ii) and the Brun-Hooley sieve (see e.g. \cite{FordHalberstam}), we have, for all $n \in \mathbb{N}$,
\[
\sum_{d | n} \mu(d)  \leq \prod_{k=1}^{K+1} \mathbf{1}_{(n, P(I_k)) = 1} \leq \prod_{k = 1}^{K+1} \sum_{d \mid (n, P(I_k))} \mu(d) \chi_k(d) = \sum_{d | n} \mu(d) \chi(d).
\]
Furthermore, by definition, $\chi(d)$ is supported on $d \leq X^{2/5}$.

By Lemma~\ref{le:Alladi} (with $H^\ast(p) = 0$ and $h^\ast(p) = g(p)$ for all $p$), we have, for all $n$,
\begin{equation}
\label{eq:AlladiBH}
|\mu(n)| g(n) \leq \sum_{d | n} \mu(d) g^{\star}(d) \chi(d).
\end{equation}
This is the sieve majorant we shall use for multiplicative functions (different sieve majorants for multiplicative functions were constructed by Matthiesen in~\cite{Matthiesen}).

We will need the following result of Friedlander (we follow \cite{Opera}).
\begin{lemma}
\label{le:Fried}
Let $\lambda_d$ be arbitrary complex numbers with $|\lambda_d | \leq 1$ and $\lambda_1 = 1$, and supported on $[1, D]$. Then, for any $h \geq 1$,
\[
\begin{split}
&\int_X^{2X} \Big | \sum_{\substack{x < n \leq x + h}} \Big ( \sum_{d | n} \lambda_d \Big ) - h \sum_{d} \frac{\lambda_d}{d} \Big |^2 dx \\
&\ll_{A} X h \cdot \sum_{d} \frac{1}{d} \Big ( \sum_{e} \frac{\lambda_{d e}}{e} \Big )^2  + h^2 D^2 (\log X)^{A+8}+ hX(\log X)^{-A}
\end{split}
\]
for any $A \geq 1$.
\end{lemma}
\begin{proof}
We get from~\cite[Proposition 6.25]{Opera} that the left hand side of the claim equals
\[
2X\sum_{d} \gamma_d \Big(\sum_{\substack{m \leq D \\ m \equiv 0 \pmod{d}}} \frac{\lambda_m}{m}\Big)^2 + O_{A}\Big(h^2 D^2(\log X)^{A+8}+ hX(\log X)^{-A}\Big),
\]
where
\[
\gamma_d = \sum_{\substack{k \geq 1 \\ (k, d) = 1}} \Big(\frac{d}{\pi k} \sin \frac{\pi h k }{d}\Big)^2 \ll \sum_{\substack{1 \leq k \leq \frac{d}{h}}} \Big(\frac{d}{k} \cdot \frac{h k}{d}\Big)^2
+ \sum_{\substack{k > \frac{d}{h}}} \Big(\frac{d}{k}\Big)^2 \ll dh.
\]
Now
\[
\begin{split}
2X \sum_{d} \gamma_d \Big(\sum_{\substack{m \leq D \\ m \equiv 0 \pmod{d}}} \frac{\lambda_m}{m}\Big)^2 &\ll hX \sum_{d \leq D} d \Big(\sum_{m \leq D/d} \frac{\lambda_{dm}}{dm}\Big)^2 \ll hX \sum_{d} \frac{1}{d} \Big(\sum_m \frac{\lambda_{dm}}{m}\Big)^2
\end{split}
\]
and the claim follows.
\end{proof}

\begin{lemma}
\label{le:SieveCoeffs}
Let $X$ be large. 
Let $g \colon \mathbb{N} \to [0, 1]$ be a multiplicative function and let $\chi$ be as in~\eqref{eq:chi(n)def}. Set $\lambda_{d} = \mu(d) g^{\star}(d) \chi(d)$. Then
\[
\sum_{d} \frac{\lambda_{d}}{d} \leq 9 \prod_{p \leq X} \Big (1 + \frac{g(p) - 1}{p} \Big )
\]
and
\[
\sum_{d} \frac{1}{d} \Big ( \sum_{e} \frac{\lambda_{d e}}{e} \Big )^2 \ll \prod_{p \leq X} \Big ( 1 + \frac{g(p)^2 - 1}{p} \Big ).
\]
\end{lemma}
\begin{proof}
Let us first notice that, for $d \leq 4X$, $\lambda_d = \prod_{k=1}^{K+1} \lambda_{(d, P(I_k))}$. For the proof of the first claim, we notice that by this multiplicativity property and definition of $\chi(n)$, we have
\[
\sum_{d} \frac{\lambda_{d}}{d} = \prod_{k \leq K-1} \Big ( \sum_{\substack{n \\ \omega_{I_k}(n) \leq 30 \lfloor \log_{k + 1} X \rfloor \\ p | n \implies p \in I_k}} \frac{\mu(n) g^{\star}(n)}{n} \Big ) \cdot \Big ( \sum_{\substack{n \\ p | n \implies p \in I_K}} \frac{\mu(n) \mathbf{1}_{S^+}(n) g^\ast(n)}{n} \Big )
\]
Now, for $k \leq K-1$
\[
\begin{split}
\sum_{\substack{n \\ \omega_{I_k}(n) > 30 \lfloor \log_{k + 1} X \rfloor \\ p | n \implies p \in I_k}} \frac{|\mu(n)|}{n} &\leq \exp(- 30 \log_{k + 1} X) \sum_{\substack{n \\ p | n \implies p \in I_k}} \frac{e^{\omega(n)}|\mu(n)|}{n} \\
&\leq \exp(- 30 \log_{k + 1} X) \prod_{p \in I_k} \Big ( 1 + \frac{e}{p} \Big ) \\
&\leq 2 \cdot \exp(- 30 \log_{k + 1} X + 2 e \log_{k + 1} X) \leq \exp( -  15 \log_{k + 1} X).
\end{split}
\]
and by Lemma~\ref{le:linearsieve}
\[
\begin{split}
\sum_{\substack{n \\ p \mid n \implies p \in I_K}} \frac{\mu(n) 1_{S^+(n)} g^\ast(n) }{n}
&\leq \frac{2e^\gamma + o(1)}{(\frac{2}{5}-\frac{1}{1000})/(1/7)} \cdot \prod_{p \in I_K} \Big(1-\frac{g^\ast(p)}{p}\Big) \leq 8.93  \prod_{p \in I_K \cup I_{K+1}} \Big(1-\frac{g^\ast(p)}{p}\Big).
\end{split}
\]
Hence, for some $|\theta_k| \leq 1$, recalling $g^\ast(p) \in [0, 1]$ for all $p$, we get
\[
\begin{split}
\sum_{d} \frac{\lambda_{d}}{d} &=  \prod_{k \leq K-1} \Big ( \prod_{p \in I_k} \Big ( 1 - \frac{g^\ast(p)}{p} \Big )  + \frac{\theta_k}{(\log_{k} x)^{15}} \Big ) \cdot 8.93  \prod_{p \in I_K \cup I_{K+1}} \Big(1-\frac{g^\ast(p)}{p}\Big) \\
&\leq 8.93 \prod_{p \leq X} \Big ( 1 - \frac{g^\ast(p)}{p} \Big ) \cdot \prod_{k \leq K-1} \Big(1+ \frac{2}{(\log_{k} x)^{10}} \Big) \\
&\leq 9 \prod_{p \leq X} \Big ( 1 - \frac{g^\ast(p)}{p} \Big )
\end{split}
\]
as claimed.

For the proof of the second claim, write $\theta := \lambda \ast 1$ (i.e. $\theta_b = \sum_{d \mid b} \lambda_d$), so that $\lambda = \mu \ast \theta$, and write $P = \prod_{p \leq 4X} p$. The argument in \cite[Proof of Lemma 6.18]{Opera} shows that
\begin{equation}
\label{eq:edlambdasum1ststep}
\sum_{d} \frac{1}{d} \Big ( \sum_{e} \frac{\lambda_{d e}}{e} \Big )^2  = \prod_{p \mid P} \Big(1-\frac{1}{p} + \frac{1}{p^2}\Big) \sum_{b \mid P} \frac{\theta_b^2}{b}.
\end{equation} 
Recalling the multiplicativity property of $\lambda_d$ from the beginning of the proof and noting that, as a convolution, $\theta_b$ inherits this property, we obtain
\begin{equation}
\label{eq:theta2sum}
\sum_{b \mid P} \frac{\theta_b^2}{b} = \prod_{k =1}^{K+1} \Big(\sum_{\substack{n \\ p | n \implies p \in I_k}} \frac{\theta_n^2}{n} \Big).
\end{equation}
Now $|\theta_n^2| \leq \tau(n)^2$ and, for $k \leq K-1$, we have
\[
\begin{split}
\sum_{\substack{\omega_{I_k}(n) > 30 \lfloor \log_{k + 1} X \rfloor \\ p | n \implies p \in I_k}} \frac{\tau(n)^2}{n} &\ll 2^{- \frac{3}{2} \cdot 30  \log_{k + 1} X} \sum_{\substack{n \\ p | n \implies p \in I_k}} \frac{2^{\frac{3}{2} \omega(n) + 2\Omega(n)}}{n} \\
&\ll 2^{- 45 \log_{k + 1} X} \prod_{p \in I_k} \Big ( 1 + \frac{2^{7/2}}{p} \Big ) \\
&\ll \exp(- 45 (\log 2) \log_{k + 1} X + 2^{9/2} \log_{k + 1} X) \ll \exp( -  8 \log_{k + 1} X).
\end{split}
\]
and furthermore, for $k = K, K+1$,
\[
\sum_{\substack{n \\ p | n \implies p \in I_k}} \frac{\tau(n)^2}{n} = O(1).
\]

When $n$ is such that $p \mid n \implies p \in I_k$ and $\omega_{I_k}(n) \leq 30 \lfloor \log_{k+1} X \rfloor$, we have $\theta_n = (1 \ast (\mu g^\ast))(n)$. Hence~\eqref{eq:theta2sum} and the bounds after it imply that
\[
\begin{split}
\sum_{b \mid P} \frac{\theta_b^2}{b}  &\ll \prod_{k \leq K-1} \Big(\sum_{\substack{n \\ p \mid n \implies p \in I_k \\ \omega_{I_k}(n) \leq 30 \lfloor \log_{k + 1} X \rfloor}} \frac{(1 \ast (\mu g^\ast))(n)^2}{n} + \sum_{\substack{\omega_{I_k}(n) > 30 \lfloor \log_{k + 1} X \rfloor \\ p | n \implies p \in I_k}} \frac{\tau(n)^2}{n}\Big) \\
&\ll \prod_{k \leq K-1} \Big(\prod_{p \in I_k} \Big(1+\frac{(1 \ast (\mu g^\ast))(p)^2}{p}\Big) + O((\log_k X)^{-2})\Big) \\
&\ll \prod_{k \leq K-1} \Big(\prod_{p \in I_k} \Big(1+\frac{(1-g^\ast(p)  )^2}{p}\Big) + O((\log_k X)^{-2})\Big) \\
&\ll \prod_{p \leq X} \Big(1+\frac{g(p)^2}{p}\Big).
\end{split}
\]
Combining this with~\eqref{eq:edlambdasum1ststep}, the claim follows.
\end{proof}

Now we are ready to prove the result that we will use to estimate the contribution of $n \not \in \mathcal{S}$ in Theorem~\ref{th:MT}. 

\begin{proposition}
\label{prop:sieveaaintervals}
There exists an absolute constant $C' > 0$ such that the following holds. Let $1 \leq P \leq Q \leq X^{3/4}$, and let $f \colon \mathbb{N} \to \mathbb{U}$ be a multiplicative function. Write $h_1 = H(f; X)$, and let $2 \leq h \leq X^{1/6}$ and $\Delta > 0$.

Let $\mathcal{E}(h)$ denote the set of integers $x \in [X, 2X]$ for which
\begin{equation}
\label{eq:E(h)def}
\frac{1}{h} \sum_{\substack{\substack{x < n \leq x + h \\ p \mid n \implies p \not \in (P, Q]}}} |f(n)|  \geq \prod_{p \leq X} \Big ( 1 + \frac{|f(p)| - 1}{p} \Big ) \left(\Delta + 20 \prod_{P < p \leq Q} \Big ( 1 - \frac{|f(p)|}{p} \Big ) \right) .
\end{equation}
Then
\[
\begin{split}
|\mathcal{E}(h)| &\leq C' \frac{X}{\Delta^2} \frac{\log^2(2+h/h_1)}{1+ h/h_1} \cdot \prod_{p \in (P, Q]} \Big ( 1 - \frac{|f(p)|^2}{p} \Big).
\end{split}
\]
\end{proposition}

\begin{proof}
Write $\mathcal{K} := \{n \in \mathbb{N} \colon p \mid n \implies p^2 \mid n\}$. Note that, for any $K \geq 1$, we have $\#(\mathcal{K} \cap (K, 2K]) \ll K^{1/2}$. Furthermore
\begin{equation}
\label{eq:Ksum}
\sum_{k \in \mathcal{K}} \frac{1}{k} = \prod_{p \in \mathbb{P}} \left(1+\frac{1}{p^2} + \frac{1}{p^3} + \dotsb \right) = \prod_{p \in \mathbb{P}} \left( 1+ \frac{1}{p^2-p}\right) \leq 1.95.
\end{equation}
Any $n$ can be uniquely written as $n = km$ where $(k, m) = 1$, $k \in \mathcal{K}$, and $|\mu(m)| = 1$, so that
\[
\sum_{\substack{x < n \leq x + h \\ p \mid n \implies p \not \in (P, Q]}} |f(n)| = \sum_{\substack{x < km \leq x + h \\ p \mid km \implies p \not \in (P, Q] \\ k \in \mathcal{K}, \, (k, m) = 1}} |\mu(m) f(km)|
\]
Write $K_0 := \max\{h/h_1, 1\}$. We split the sum over $k \in \mathcal{K}$ into three parts $\mathcal{K}_0, \mathcal{K}_1$ and $\mathcal{K}_2$, where 
\[
\mathcal{K}_0 := \mathcal{K} \cap [1, K_0], \quad \mathcal{K}_1 := \mathcal{K} \cap (K_0, X^{2/5}] \quad \text{and} \quad \mathcal{K}_2 := \mathcal{K} \cap (X^{2/5}, 2X+h].
\]
First note that
\[
\begin{split}
&\frac{1}{X} \int_{X}^{2X} \sum_{\substack{x < km \leq x + h \\ p \mid km \implies p \not \in (P, Q] \\ k \in \mathcal{K}_2, \, (k, m) = 1}} |\mu(m) f(km)| \, dx \leq  \frac{h}{X} \sum_{\substack{k \in \mathcal{K}_2}}  \sum_{\substack{X/k < m \leq (2X + h)/k}} 1 \ll \frac{h}{X^{1/5}}.
\end{split}
\]
Since we can clearly assume that $\Delta \leq \log X$ (otherwise~\eqref{eq:E(h)def} can never hold for large $X$), this implies that
\[
\sum_{\substack{x < km \leq x + h \\ p \mid km \implies p \not \in (P, Q] \\ k \in \mathcal{K}_2, \, (k, m) = 1}} |\mu(m) f(km)| > \frac{\Delta}{100} h \prod_{p \leq X} \Big(1+\frac{|f(p)|-1}{p}\Big)
\]
for at most $\ll X^{0.81}$ values $x \in [X, 2X]$. This exceptional set is acceptable since $h \leq X^{1/6}$.

Furthermore
\begin{equation}
\label{eq:SieveL2}
\begin{split}
&\int_{X}^{2X} \Bigl(\sum_{\substack{x < km \leq x + h \\ p \mid km \implies p \not \in (P, Q] \\ k \in \mathcal{K}_1, \, (k, m) = 1}} |\mu(m) f(km)| \Bigr)^2 dx \\
&\ll h \sum_{\substack{X < km \leq 2X+h \\ p \mid m \implies p \not \in (P, Q] \\ k \in \mathcal{K}_1}} |\mu(m) f(m)|^2  + h \sum_{\substack{k_1, k_2 \in \mathcal{K}_1}} \sum_{\substack{|r| \leq h \\ (k_1, k_2) \mid r}} \sum_{\substack{X < n \leq 2X + h \\ k_1 \mid n, k_2 \mid n+r \\ p \mid n(n+r) \implies p \not \in (P, Q]}} |f(n)f(n+r)|.
\end{split}
\end{equation}
Applying Lemma~\ref{le:Shiu} to the first term, we see that it is
\[
\begin{split}
&\ll h \sum_{\substack{k \in \mathcal{K}_1}} \frac{X}{k} \prod_{p \leq X} \Big(1+\frac{|f(p)|^2-1}{p}\Big) \cdot  \prod_{p \in (P, Q]} \Big(1-\frac{|f(p)|^2}{p}\Big) \\
&\ll \frac{h X}{K_0^{1/2}} \cdot h_1 \prod_{p \leq X} \Big(1+\frac{2|f(p)|-2}{p}\Big) \prod_{p \in (P, Q]} \Big(1-\frac{|f(p)|^2}{p}\Big).
\end{split}
\]
Applying Lemma~\ref{le:Henriot} to the second term on the right hand side of~\eqref{eq:SieveL2}, we see that it is
\[
\begin{split}
&\ll h^2 X \sum_{\substack{k_1, k_2 \in \mathcal{K}_1 }} \frac{1}{k_1 k_2} \prod_{p \leq X} \Big(1+\frac{2|f(p)|-2}{p}\Big) \cdot  \prod_{p \in (P, Q]} \Big(1-\frac{2 |f(p)|}{p}\Big) \prod_{p \mid k_1 k_2} \left(1+\frac{1}{p}\right) \\
&\ll \frac{h^2 X}{K_0} \prod_{p \leq X} \Big(1+\frac{2|f(p)|-2}{p}\Big) \cdot  \prod_{p \in (P, Q]} \Big(1-\frac{2 |f(p)|}{p}\Big)
\end{split} 
\]
Plugging these bounds into~\eqref{eq:SieveL2}, we see that
\[
\sum_{\substack{x < km \leq x + h \\ p \mid km \implies p \not \in (P, Q] \\ k \in \mathcal{K}_1, \, (k, m) = 1}} |\mu(m) f(km)| > \frac{\Delta}{100} h \prod_{p \leq X} \Big(1+\frac{|f(p)|-1}{p}\Big).
\]
for at most 
\[
\ll \frac{X}{\Delta^2 (h/h_1 + 1)} \prod_{p \in (P, Q]} \Big(1-\frac{|f(p)|^2}{p}\Big)
\]
integers $x \in [X, 2X]$.

Now we can concentrate to $\mathcal{K}_0$ (note that if $h < h_1$, we are done now). Let $g \colon \mathbb{N} \to [0, 1]$ be the multiplicative function which is defined by $g(p^\alpha) = |f(p^\alpha)| \mathbf{1}_{p \not \in (P, Q]}$. We shall apply Lemma~\ref{le:Fried}. Let us take $\lambda_d = \mu(d) g^\ast(d) \chi(d)$ with $\chi$ as in~\eqref{eq:chi(n)def}. Note that, by~\eqref{eq:AlladiBH}, 
\[
\sum_{\substack{x < km \leq x + h \\ p \mid km \implies p \not \in (P, Q] \\ k \in \mathcal{K}_0, \, (k, m) = 1}} |\mu(m) f(km)| \leq \sum_{\substack{x < km \leq x + h \\ k \in \mathcal{K}_0}} |\mu(m) g(m)| \leq \sum_{\substack{x < dkn \leq x + h \\ k \in \mathcal{K}_0}} \lambda_d
\]
and, by Lemma~\ref{le:SieveCoeffs} and~\eqref{eq:Ksum},
\[
\begin{split}
\sum_{k \in \mathcal{K}_0} \frac{h}{k} \sum_{d} \frac{\lambda_d}{d} &\leq 18 h \prod_{p \leq X} \Big(1+\frac{|f(p)|-1}{p}\Big) \cdot \prod_{p \in (P, Q]} \Big(1-\frac{|f(p)|}{p}\Big).
\end{split}
\]
Hence it suffices to show that
\[
\begin{split}
&\int_{X}^{2X} \Bigl|\sum_{\substack{k \in \mathcal{K}_0}} \Bigl(\sum_{\substack{x/k < dn \leq (x + h)/k}} \lambda_d - \frac{h}{k} \sum_{d} \frac{\lambda_d}{d}\Bigr) \Bigr|^2 dx \\
& \ll \log^2(2K_0) X h h_1\prod_{p \leq X} \left(1+\frac{2|f(p)|-2}{p}\right) \prod_{p \in (P, Q]} \left(1-\frac{|f(p)|^2}{p}\right).
\end{split}
\]
Now, by Cauchy-Schwarz, the left hand side is
\[
\begin{split}
&\ll \sum_{\substack{k \in \mathcal{K}_0}} \frac{1}{k^{1/2}} \sum_{\substack{k \in \mathcal{K}_0}} k^{1/2} \int_{X}^{2X} \Bigl|\sum_{\substack{x/k < dn \leq (x + h)/k}} \lambda_d - \frac{h}{k} \sum_{d} \frac{\lambda_d}{d} \Bigr|^2 dx \\
&\ll \log (2K_0) \sum_{\substack{k \in \mathcal{K}_0}} k^{3/2} \int_{X/k}^{2X/k} \Bigl|\sum_{\substack{x < dn \leq x + h/k}} \lambda_d - \frac{h}{k} \sum_{d} \frac{\lambda_d}{d} \Bigr|^2 dx.
\end{split}
\]
By Lemmas~\ref{le:Fried} and~\ref{le:SieveCoeffs} this is
\[
\begin{split}
&\ll \log(2K_0) X h \sum_{\substack{k \in \mathcal{K}_0}} \frac{1}{k^{1/2}} \prod_{p \leq X} \left(1+\frac{|f(p)|^2-1}{p}\right) \prod_{p \in (P, Q]} \left(1-\frac{|f(p)|^2}{p}\right) \\
& \ll \log^2(2K_0) X h h_1\prod_{p \leq X} \left(1+\frac{2|f(p)|-2}{p}\right) \prod_{p \in (P, Q]} \left(1-\frac{|f(p)|^2}{p}\right)
\end{split}
\]
as desired.
\end{proof}

We shall estimate long sums of multiplicative functions using the following lemma. Results in a similar spirit can be found in a recent paper of Elliott~\cite{Elliott17} but part (iii) is not implied by those results.
\begin{lemma}
\label{le:GrKoMa}
\begin{enumerate}[(i)]
Let $X$ be large and let $f: \mathbb{N} \rightarrow [0,1]$ be a multiplicative function.
\item One has
\[
\frac{1}{X} \sum_{X < n \leq 2X} f(n) \leq 10 \prod_{\substack{p \leq X}} \Big(1+\frac{f(p)-1}{p}\Big).
\]
\item One has
\[
\frac{1}{3} \cdot \prod_{p \leq X} \Big(1+\frac{f(p)}{p}\Big)  \leq \sum_{n \leq X} \frac{f(n)}{n} \leq  \prod_{p \leq X} \Big(1+\frac{f(p)}{p}+\frac{f(p^2)}{p^2} + \dotsb \Big).
\]
\item
There exists a positive constant $\lambda$ such that the following holds for any $\delta > 0$. Assume that
\[
\sum_{X^{\delta} < p \leq X^{1/4}} \frac{f(p)}{p} \geq \lambda.
\]
Then
\[
\frac{1}{X} \sum_{X < n \leq 2X} f(n) \geq c \prod_{\substack{p \leq X}} \Big(1+\frac{f(p)-1}{p}\Big)
\]
for a positive constant $c$ depending only on $\delta$.
\end{enumerate}
\end{lemma}
\begin{remark}
  An alternative proof of (i) can be obtained by appealing to \cite[Theorem 3.5, Chapter III.3]{Tenenbaum}. 
\end{remark}
\begin{proof}
Part (i): We take $\mathcal{K} = \{n \in \mathbb{N} \colon p \mid n \implies p^2 \mid n\}$, and write
\[
\sum_{X < n \leq 2X} f(n) \leq \sum_{k \in \mathcal{K}} f(k) \sum_{\substack{X/k < n \leq 2X/k}} |\mu(n)| f(n).
\]
Those $k$ with $k > X^{1/1000}$ make an acceptable contribution $O(X^{1-1/2000})$. Hence we concentrate on $k \leq X^{1/1000}$. Write $f^{\star}$ for the completely multiplicative function such that $f^{\star}(p) = 1 - f(p)$. Similarly to~\eqref{eq:AlladiBH}, we have, according to Lemma~\ref{le:Alladi}, with $S^+$ as in Lemma~\ref{le:linearsieve} with $D = z = X^{9/10}$ and $\mathcal{P} = \mathbb{P} \cap [1, z]$,
\[
\sum_{\substack{X/k < n \leq 2X/k}} |\mu(n)| f(n) \leq \sum_{\substack{X/k < n \leq 2X/k}} \sum_{\substack{d \mid n \\ d \in S^+}} \mu(d) f^\ast(d) \leq \frac{X}{k}\sum_{d \in S^+} \frac{\mu(d) f^{\ast}(d)}{d} + O(X^{9/10}).
\]
By Lemma~\ref{le:linearsieve}(iii) this equals
\[
\frac{X}{k}(2e^\gamma+o(1)) \prod_{p \leq X^{9/10}} \Big(1+\frac{f(p)-1}{p}\Big),
\]
and the claim follows since $\frac{10}{9} \cdot 2 e^\gamma \cdot \sum_{k \in \mathcal{K}} \frac{1}{k} < 10$.

Part (ii): We generalise the argument in~\cite[Proof of Lemma 2.1]{GrKoMa12}. The upper bound is immediate by multiplicativity. For the lower bound, let $f^\ast$ be multiplicative function such that $f^\ast(p) = 1- f(p)$ and $f^\ast(p^k) = 0$ for $k \geq 2$. Note that
\[
\sum_{m \leq X} \frac{f^\ast(m)}{m} \leq \prod_{p \leq X} \Big(1+\frac{f^\ast(p)}{p}\Big )
\]
and so
\[
\begin{split}
\sum_{n \leq X} \frac{f(n)}{n} &\geq \sum_{n \leq X} \frac{f(n)}{n} \cdot \sum_{m \leq X} \frac{f^\ast(m)}{m} \prod_{p \leq X} \Big(1+\frac{f^\ast(p)}{p}\Big )^{-1} \\
& \geq \sum_{n \leq X} \frac{f \ast f^\ast(n)}{n} \prod_{p \leq X} \Big(1+\frac{f^\ast(p)}{p}\Big )^{-1}\\
& = \sum_{n \leq X} \frac{|\mu(n)|}{n} \prod_{p \leq X} \Big(1+\frac{f(p)}{p} \Big ) \cdot \prod_{p \leq X} \Big(1-\frac{1}{p}\Big)\\
& \geq \Bigl(\frac{6}{\pi^2} + o(1)\Bigr) \log X \prod_{p \leq X} \Big(1+\frac{f(p)}{p} \Big ) \cdot \prod_{p \leq X} \Big(1-\frac{1}{p}\Big) \\
& = \Bigl(\frac{6e^{-\gamma}}{\pi^2} + o(1)\Bigr) \prod_{p \leq X} \Big(1+\frac{f(p)}{p} \Big ),
\end{split}
\]
where we used Mertens' theorem.

Part (iii): This is a variant of~\cite[Theorem 1]{GrKoMa12} and we quickly deduce it from that result. By choosing $\lambda$ large enough, we can assume that
\begin{equation}
\label{eq:pxdelta1/4sum}
\prod_{X^\delta < p \leq X^{1/4}} \Big(1+\frac{f(p)}{p}\Big) \geq 20.
\end{equation}
Furthermore we also have, with $\eta = \lambda/(2 \log(1/\delta))$, and for any integer $m \in [X^{\delta}, X^{1/4}]$, 
\[
\sum_{\substack{X^{\delta} < p \leq X^{1/4} \\ f(p) \geq \eta \\ p \nmid m}} \frac{1}{p} \geq \sum_{\substack{X^{\delta} < p \leq X^{1/4}}} \frac{f(p)}{p} - \eta \sum_{\substack{X^{\delta} < p \leq X^{1/4}}} \frac{1}{p} - \sum_{\substack{X^{\delta} \leq p \leq X^{1/4} \\ p | m}} \frac{1}{p} \geq \lambda/3
\]
since $\sum_{\substack{X^{\delta} \leq p \leq X^{1/4} \\ p | m}} \frac{1}{p} = o(1)$. 
Hence, once $\lambda$ is large enough, \cite[Theorem 1 together with Remark 1.4]{GrKoMa12} implies that, for any $y \in [X^{1/2}, X]$, and any integer $m \in [X^{\delta}, X^{1/4}]$, 
\[
\frac{1}{y} \sum_{\substack{y < n \leq 2y \\ p \mid n \implies p \in (X^\delta, X^{1/4}] \\ (n,m) = 1}} f(n) \geq \frac{\eta^{1/\delta}}{y} \sum_{\substack{y < n \leq 2y \\ p \mid n \implies f(p) \geq \eta, p \in (X^\delta, X^{1/4}] \\ (n,m) = 1}} 1 + O(X^{-\delta/2}) \gg_\delta \frac{1}{\log X}.
\]
Now, for some constants $c_1 = c_1(\delta)$ and $c_2 = c_2(\delta)$, we get
\[
\begin{split}
\frac{1}{X} \sum_{X < n \leq 2X} f(n) &\geq c_1 \frac{1}{X} \sum_{\substack{X^{\delta} < m \leq X^{1/4}}} f(m) \sum_{\substack{X/m < n \leq 2X/m \\ p \mid n \implies p \in (X^{\delta}, X^{1/4}] \\ (n,m) = 1}} f(n) \\
&\geq c_2 \frac{1}{\log X} \Big(\sum_{m \leq X^{1/4}} \frac{f(m)}{m} - \sum_{m \leq X^\delta} \frac{f(m)}{m}\Big).
\end{split}
\]
By part (ii) the right hand side is at least
\[
c_2 \frac{1}{\log X} \prod_{p \leq X^{\delta}} \Big(1+\frac{f(p)}{p}\Big) \Big(\frac{1}{3} \cdot \prod_{X^{\delta} < p \leq X^{1/4}} \Big(1+\frac{f(p)}{p}\Big)  - 1 \Big) \gg_\delta \prod_{p \leq X} \Big(1+\frac{f(p)-1}{p}\Big)
\]
by~\eqref{eq:pxdelta1/4sum}.
\end{proof}

\section{Proof of Theorems~\ref{th:MT} and \ref{th:MTDense} and Corollaries \ref{cor:convenient} and \ref{cor:main1}}
\begin{proof}[Proof of Theorem~\ref{th:MT}]
By adjusting $\rho$, we can clearly assume that
\begin{equation}
\label{eq:deltaasThMT}
\delta \geq C_1 \Big(\frac{\log \log h_0}{\log h_0}\Big)^\alpha + \frac{C_1}{(\log X)^{\alpha \rho/36}}
\end{equation}
for a large constant $C_1$. Furthermore we can assume without loss of generality that $h_0 \leq X^{1/10}$ as  Theorem \ref{th:MT} for $h_0 > X^{1/10}$ is implied by Theorem \ref{th:MT} with $h_0 = X^{1/10}$. 

Let us begin by choosing $P_j$ and $Q_j$ defining the set $\mathcal{S}$. Let $\delta' = \delta/200$. We take $\nu_1 = \theta \delta'^{2/\alpha}, \nu_2 = \theta$, $\eta = 1/12$, $Q_1 = \min\{h_0, X^{\nu_1}\}$, and $P_1 = Q_1^{\delta'^{1/\alpha}}$. Note that by our assumption on $\delta$, we have $P_1 \geq (\log Q_1)^{40/\eta}$. If $Q_1 \geq \exp((\log X)^{1/2})$, we take $J = 1$. Otherwise we take, for $j=2, \dotsc, J$, $P_j$ and $Q_j$ as in~\eqref{eq:PjQjchoice} with $\beta_0 = \alpha^2$. Let $\mathcal{S}$ be as in Section~\ref{sec:RedSM*}.

We first reduce to averages over the set $\mathcal{S}$. We have
\[
\frac{1}{h_0 H(f; X)}\sum_{\substack{x < n \leq x+h_0 H(f; X) \\ n \not \in \mathcal{S}}} |f(n)| \leq \sum_{j = 1}^{J+2} \frac{1}{h_0 H(f; X)} \sum_{\substack{x < n \leq x+ h_0 H(f; X)  \\ p \mid n \implies p \not \in (P_j, Q_j]}} |f(n)|
\]
Write $\delta_j = \delta'/(4j^2)$ for $j = 1, \dotsc, J$ and $\delta_{J+1} = \delta_{J+2} = \delta'/4$. If, for each $j = 1, \dotsc, J+2$ one has
\begin{equation}
\label{eq:NoDivPjQjShortBound}
\frac{1}{h_0 H(f; X)} \sum_{\substack{x < n \leq x+ h_0 H(f; X)  \\ p \mid n \implies p \not \in (P_j, Q_j]}} |f(n)| \leq \Big(\delta_j + 20 \prod_{p \in (P_j, Q_j]} \Big(1-\frac{|f(p)|}{p}\Big) \Big) \prod_{p \leq X} \Big(1+\frac{|f(p)|-1}{p} \Big),
\end{equation}
then, recalling definitions of $P_j$ and $Q_j$ and that $f$ is $(\alpha, X^\theta)$-non-vanishing,
\[
\begin{split}
&\frac{1}{h_0 H(f; X)} \sum_{\substack{x < n \leq x+h_0 H(f; X) \\ n \not \in \mathcal{S}}} |f(n)| \\
&\leq \Big(\frac{\delta'}{2} + \sum_{j = 1}^J \frac{\delta'}{4j^2} + 20 \sum_{j = 1}^{J+2} \prod_{P_j < p \leq Q_j} \Big(1-\frac{|f(p)|}{p}\Big)\Big) \prod_{p \leq X} \Big(1 + \frac{|f(p)|-1}{p}\Big) \\
&\leq \frac{\delta}{2} \prod_{p \leq X} \Big(1 + \frac{|f(p)|-1}{p}\Big).
\end{split}
\]
On the other hand, by Proposition~\ref{prop:sieveaaintervals} the measure of $x \in [X, 2X]$ for which~\eqref{eq:NoDivPjQjShortBound} fails for some $j$ is 
\[
\ll \sum_{j = 1}^{J+2} \frac{X (\log h_0)^2}{\delta_j^2 h_0} \cdot \prod_{p \in (P_j, Q_j]} \Big ( 1 - \frac{|f(p)|^2}{p} \Big) \ll \frac{X}{\delta^2 h_0^{9/10}} \ll \frac{X}{h_0^{4/5}}.
\]
by \eqref{eq:deltaasThMT}.

Furthermore Lemma~\ref{le:GrKoMa}(i) implies that
\[
\begin{split}
\sum_{\substack{X < n \leq 2X \\ n \not \in \mathcal{S}}} |f(n)| &\leq 10 X \prod_{p \leq X}\Big(1+\frac{|f(p)|-1}{p}\Big)\cdot \sum_{j = 1}^{J+2} \prod_{P_j < p \leq Q_j} \Big(1-\frac{|f(p)|}{p}\Big) \\
&\leq \frac{\delta}{4} X \prod_{p \leq X}\Big(1+\frac{|f(p)|-1}{p}\Big).
\end{split}
\]
Hence the contribution of $n \not \in \mathcal{S}$ to the left hand side of~\eqref{eq:ress} is acceptable and it suffices to bound the cardinality of the set of $x \in [X, 2X]$ for which
\begin{align*}
\Big | \frac{1}{h_0 H(f; X)} \sum_{\substack{x < n \leq x + h_0 H(f; X) \\ n \in \mathcal{S}}} f(n) & - \frac{1}{h_0 H(f; X)} \int_{x}^{x+h_0 H(f; X)} u^{i \widehat{t}_{f, X}} du \cdot \frac{1}{X} \sum_{\substack{X < n \leq 2X \\ n \in \mathcal{S}}} f(n) n^{-i\widehat{t}_{f, X}} \Big | \\
&> \frac{\delta}{4} \prod_{p \leq x} \Big ( 1 + \frac{|f(p)| - 1}{p} \Big ).
\end{align*}
Recalling~\eqref{eq:deltaasThMT} and the definition of $\nu_1$, Theorem~\ref{th:ThminS} gives the bound
\begin{equation}
\label{eq:excSetFin}
\ll_{\rho, \theta} \frac{X}{\delta^2} \cdot \frac{(\log h_0)^2}{P_1^{1 / 2 - \nu_2 - 3\eta}} + X^{1-\nu_1^3/200} \ll X \cdot \frac{(\log h_0)^4}{Q_1^{3\delta'^{1/\alpha}/16}}   +  X^{1-\theta^3 (\delta/2000)^{6/\alpha}}.
\end{equation}
Recall~\eqref{eq:deltaasThMT} and that $Q_1 = \min\{h_0, X^{\nu_1}\}$. If $h_0 \leq X^{\nu_1}$, then the first term on the right hand side of~\eqref{eq:excSetFin} is $\ll X/h_0^{\delta'^{1/\alpha}/8} \ll X/h_0^{(\delta/2000)^{1/\alpha}}$. On the other hand if $h_0 > X^{\nu_1}$, then the first term on the right hand side of~\eqref{eq:excSetFin} is $\ll X (\log X)^4/X^{3\theta \delta'^{3/\alpha}/16} \ll  X^{1-\theta^3 (\delta/2000)^{6/\alpha}}$. Hence the claim follows.

The last claim concerning almost real-valued $f$ follows similarly using part (iii) of Theorem \ref{th:ThminS}. 
\end{proof}

\begin{proof}[Proof of Theorem \ref{th:MTDense}]
By adjusting $\rho$, we can clearly assume that
\begin{equation}
\label{eq:deltaasThMTDense}
\delta \geq C_1 \frac{\log \log h}{\log h} + \frac{C_1}{(\log X)^{\rho/36}}
\end{equation}
for a large constant $C_1$.

This time we handle $n \in \mathcal{S}$ following~\cite[Proof of Theorem 1 in Section 9]{MainPaper} which is more efficient but only works in the dense setting. Let $\mathcal{S}$ be any set as in Section~\ref{sec:RedSM*} satisfying the assumptions of Theorem~\ref{th:ThminS}(ii). Arguing as in~\cite[Proof of Theorem 1 in Section 9]{MainPaper} we see that
\begin{equation}
\label{eq:denseSnotSsplit}
\begin{split}
&\Big | \frac{1}{h} \sum_{x < n \leq x + h} f(n) - \frac{1}{h} \int_{x}^{x+h} u^{i t_{f, X}} du \cdot \frac{1}{X} \sum_{\substack{X < n \leq 2X }} f(n) n^{-it_{f, X}} \Big | \\
&\leq \Big| \frac{1}{h} \sum_{\substack{x < n \leq x + h \\ n \in \mathcal{S}}} f(n) - \frac{1}{h} \int_{x}^{x+h} u^{i t_{f, X}} du \cdot \frac{1}{X} \sum_{\substack{X < n \leq 2X \\ n \in \mathcal{S}}} f(n) n^{-it_{f, X}} \Big| \\
&\quad + \Big| \frac{1}{h} \sum_{\substack{x < n \leq x + h \\ n \in \mathcal{S}}} 1 - \frac{1}{X} \sum_{\substack{X < n \leq 2X \\ n \in \mathcal{S}}} 1 \Big|
 + \frac{2}{X} \sum_{\substack{X < n \leq 2X \\ n \not \in \mathcal{S}}} 1 + O(1/h).
\end{split}
\end{equation}
Theorem \ref{th:ThminS}(ii)-(iii) applied to $f(n)$ and to $1$ implies that the first and second terms on the right hand side are both at most $\delta/100$ with at most
\begin{equation}
\label{eq:excthminS}
\ll_{\eta, \rho} \frac{X (\log h)^2}{\delta^2 P_1^{1/2 - \nu_2 - 3 \eta}}+X^{1-\nu_1^3/200}
\end{equation}
exceptions.

By Lemma~\ref{le:linearsieve} (since $\nu_2 < 1/6$, we can take e.g. $D = X^{\nu_2 (3-1/1000)}$ and $z = X^{\nu_2}$ so that $s = 3-1/1000$ and $F(s) \leq 1.19$), for all large enough $X$,
\begin{equation}
\label{eq:DenseNotInSContr}
\begin{split}
\sum_{\substack{X < n \leq 2X \\ n \not \in \mathcal{S}}} 1 &\leq \sum_{j = 1}^{J+2} \sum_{\substack{X < n \leq 2X \\ p \mid n \implies p \not \in (P_j, Q_j]}} 1 \leq 1.195 X \cdot \sum_{j =1}^{J+2} \prod_{P_j < p \leq Q_j} \Big ( 1 - \frac{1}{p} \Big ) \\
\end{split}
\end{equation}

Let us now choose $P_j$ and $Q_j$ defining the set $\mathcal{S}$. We take $\nu_1 = \delta^{2}/4000, \nu_2 = 1/10$, and $\eta = 1/150$. We choose $Q_1 = \min\{h_0, X^{\nu_1}\}$ and $P_1 = Q_1^{\delta/4}$. For $j=2, \dotsc, J$, we choose $P_j$ and $Q_j$ as in~\eqref{eq:PjQjchoice} with $\beta_0 = 1$. Then~\eqref{eq:denseSnotSsplit} and~\eqref{eq:DenseNotInSContr} imply that, with at most~\eqref{eq:excthminS} exceptions, we have
\[
\begin{split}
&\Big | \frac{1}{h} \sum_{x < n \leq x + h} f(n) - \frac{1}{X} \sum_{X < n \leq 2X} f(n) \Big | \\
&\leq \delta/50 + \frac{12}{5} \cdot \frac{\log P_1}{\log Q_1} \cdot \sum_{j=1}^J \frac{1}{j^6} + \frac{12}{5} \cdot 2 \cdot \sqrt{\frac{\nu_1}{\nu_2}} \\
&\leq \delta/50 + \frac{12}{5} \cdot 1.02 \cdot \frac{\delta}{4} + \frac{6}{25}\delta \leq \delta.
\end{split}.
\]
The exceptional set~\eqref{eq:excthminS} is acceptable thanks to the assumption~\eqref{eq:deltaasThMTDense}.

The last claim concerning almost real-valued $f$ follows similarly using part (iii) of Theorem \ref{th:ThminS}. 
\end{proof}

\begin{proof}[Proof of Corollary \ref{cor:convenient}]
Let us concentrate on part (i). By Theorem~\ref{th:MT} it clearly suffices to show that
\[
\begin{split}
&\Big| \frac{1}{h} \int_{x}^{x+ h} u^{i \widehat{t}_{f, X}} du \cdot \frac{1}{X} \sum_{X < n \leq 2X} f(n) n^{-i\widehat{t}_{f, X}} \Big | \\
&\ll \left(\frac{1}{\alpha (\log X)^{\alpha}} + \frac{\widehat{M}(f; X)}{\alpha \exp(\widehat{M}(f; X))}\right) \prod_{p \leq X} \Big ( 1 + \frac{|f(p)| - 1}{p} \Big ).
\end{split}
\]
In case $|\widehat{t}_{f, X}| \leq X/2$ this follows from Lemma~\ref{le:SparseHalaszComplex}(iv). On the other hand in case $|\widehat{t}_{f, X}| > X/2$ we have
\[
\frac{1}{h} \int_{x}^{x+h} u^{i\widehat{t}_{f, X}} du \ll \frac{1}{h},
\]
and the claim follows from Shiu's bound (Lemma~\ref{le:Shiu}).

Part (ii) follows similarly from Theorem~\ref{th:MTDense}.
\end{proof}

\begin{proof}[Proof of Corollary \ref{cor:main1}]

  Recall that $\delta$ is assumed to lie in $(0, 1/1000)$. We can assume that 
  that 
  \begin{equation}
  \label{eq:CorMain1deltas}
  \delta \geq C' \left(\frac{\log \log h_0}{\log h_0}\right)^\alpha + C' \frac{1}{2(\log X)^{\alpha \rho_\alpha/40}}
  \end{equation}
  for any large constant $C'$ since otherwise the claim is trivial taking $\kappa \geq 50/(\alpha \rho_\alpha)$. We can also assume that $h_0 \leq X^{1/100}$.
  
In case 
\begin{equation} \label{eq:holds}
    \sum_{\substack{p \leq X \\ p \in \mathcal{N}}} \frac{1 - |f(p)|}{p} > 5 + \log \frac{1}{\delta},
\end{equation}
we shall apply Proposition~\ref{prop:sieveaaintervals} with $P = Q$. Note that in this case
\[
\begin{split}
\frac{\delta}{2} \delta(\mathcal{N}; X) & \geq \frac{\delta}{2} \prod_{\substack{p \leq X \\ p \not \in \mathcal{N}}} \left(1-\frac{1}{p}\right) \cdot \prod_{p \in \mathcal{N}} \left(1-\left(\frac{|f(p)|-1}{p}\right)^2\right) \\
&= \prod_{p \leq X} \left(1+\frac{|f(p)|\mathbf{1}_{p \in \mathcal{N}} - 1}{p}\right) \cdot \frac{\delta}{2} \prod_{\substack{p \leq X \\ p \in \mathcal{N}}} \left(1+\frac{1-|f(p)|}{p}\right) \\
&\geq \cdot \prod_{p \leq X} \left(1+\frac{|f(p)|\mathbf{1}_{p \in \mathcal{N}} - 1}{p}\right) \cdot \left(\frac{\delta}{10} \prod_{\substack{p \leq X \\ p \in \mathcal{N}}} \left(1+\frac{1-|f(p)|}{p}\right) + 20\right)
\end{split}
\]
and
\[
\frac{h_0 \delta(\mathcal{N}; X)^{-1}}{H(f\mathbf{1}_{n \in \mathcal{N}}; X)} \geq \frac{h_0}{\prod_{p \leq X} \left(1+\frac{1-|f(p)|\mathbf{1}_{p \in \mathcal{N}}}{p}\right) \prod_{\substack{p \leq X \\ p \not \in \mathcal{N}}} \left(1-\frac{1}{p}\right)} \gg \frac{h_0} {\prod_{\substack{p \leq X \\ p \in \mathcal{N}}} \left(1+\frac{1-|f(p)|}{p}\right)}.
\]
Hence Proposition~\ref{prop:sieveaaintervals} shows that the cardinality of the set of $x \in [X, 2X]$ for which,
    $$
   \frac{1}{h_0 \delta(\mathcal{N}; X)^{-1}} \sum_{\substack{x < n \leq x + h_0 \delta(\mathcal{N}; X)^{-1} \\ n \in \mathcal{N}}} |f(n)| \geq \frac{\delta}{2} \delta(\mathcal{N}; X) 
   $$
   is bounded by
   \[
   \ll \frac{X}{\delta^2 h_0^{1/2} \prod_{\substack{p \leq X \\ p \in \mathcal{N}}} \left(1+\frac{1-|f(p)|}{p}\right)^{3/2}} \ll \frac{X}{h_0^{1/4}}.
   \]
  Furthermore, when \eqref{eq:holds} holds, one has by Lemma \ref{le:GrKoMa}(i),  
   $$
   \frac{1}{X} \sum_{\substack{X < n \leq 2X \\ n \in \mathcal{N}}} |f(n)| \leq 10 \prod_{p \leq X} \Big ( 1 + \frac{\mathbf{1}_{p \in \mathcal{N}} |f(p)| - 1}{p} \Big ) < \frac{\delta}{2} \delta(\mathcal{N}; X).
   $$
   
   Therefore we can assume that \eqref{eq:holds} does not hold. Then
   $$\sum_{p \leq X}\frac{(\mathbf{1}_{p \in \mathcal{N}} |f(p)| - 1)^2}{p} \leq \sum_{p \leq X} \frac{1 - \mathbf{1}_{p \in \mathcal{N}} |f(p)|}{p} \leq 5 + \log \frac{1}{\delta} + \sum_{\substack{p \leq X \\ p \not \in \mathcal{N}}} \frac{1}{p} $$
   As a result,
   $$
H(f \mathbf{1}_{n \in \mathcal{N}}; X) = \prod_{p \leq X} \Big ( 1 + \frac{(|f(p)| \mathbf{1}_{p \in \mathcal{N}} - 1)^2}{p} \Big ) \leq \frac{500}{\delta} \delta(\mathcal{N}; X)^{-1}
$$
Now the result follows from Theorem~\ref{th:MT} with $h_0 \delta(\mathcal{N}; X)^{-1} / H(f \mathbf{1}_{n \in \mathcal{N}}; X) \geq h_0^{1/2}$ in place of $h_0$.  
\end{proof}

\section{Proof of Theorem~\ref{th:LowerBound} and Corollary~\ref{cor:HooleyGen}}

We will split the proof of Theorem \ref{th:LowerBound} into two parts, depending on the size of $h_0$.
When $h_0 \leq X^{\varepsilon^3 / 20000}$ we can simply appeal to Theorem \ref{th:MT}. However when $h_0 > X^{\varepsilon^3 / 20000}$ we need the following technical result.

\begin{proposition}\label{pr:largevalues}
Let $\varepsilon > 0$ be small. Let $f : \mathbb{N} \rightarrow \mathbb{U}$ be a multiplicative function that is $(\alpha, X^{\theta})$ non-vanishing for some $\alpha, \theta \in (0, 1]$. Let $X$ be large enough.

Let $k =\lfloor \frac{1-2\varepsilon^{10}}{\varepsilon^{10}}\rfloor$, and let 
\[
\mathcal{P} = \mathbb{P} \cap (X^{\varepsilon^{10}(1-\varepsilon^{20})}, X^{\varepsilon^{10}(1+\varepsilon^{20})}].
\] 
Let $0 < \rho < \rho_\alpha$ and let $h \in ( X^{\varepsilon^3 / 20000}, X/2]$. 
\begin{enumerate}[(i)]
\item If $f$ is almost real-valued, then for $x \in (X, 2X]$, we have
\[
\begin{split}
&\Big | \frac{1}{h} \sum_{\substack{x < p_1 \ldots p_k m \leq x + h \\ p_j \in \mathcal{P}, m \in \mathbb{N}}} f(p_1 \ldots p_k m) - \frac{1}{X} \sum_{\substack{X < p_1 \ldots p_k m \leq 2X \\ p_j \in \mathcal{P}, m \in \mathbb{N}}} f(p_1 \ldots p_k m) \Big | \\
& \qquad \ll_{\varepsilon, \rho, \theta} \frac{1}{(\log x)^{\rho/12}} \prod_{p \leq X} \Big ( 1 + \frac{|f(p)| - 1}{p} \Big )
\end{split}
\]
with at most
$$
\leq C \frac{X}{h^{1/2 - \varepsilon}}
$$
exceptions, with $C$ a constant depending only on $\alpha, \varepsilon, \theta$.
\item For $x \in (X, 2X]$, we have
\[
\begin{split}
&\Big | \frac{1}{h} \sum_{\substack{x < p_1 \ldots p_k m \leq x + h \\ p_j \in \mathcal{P}, m \in \mathbb{N}}} f(p_1 \ldots p_k m) \Big | \ll_{\varepsilon, \theta} \left(\frac{\widehat{M}(f; X)}{\exp(\widehat{M}(f, X))} + \frac{1}{(\log X)^\alpha} \right) \frac{1}{\alpha} \prod_{p \leq X} \Big ( 1 + \frac{|f(p)| - 1}{p} \Big )
\end{split}
\]
with at most
$$
\leq C \frac{X}{h^{1/2 - \varepsilon}}
$$
exceptions, with $C$ a constant depending only on $\alpha, \varepsilon, \theta$. 
\end{enumerate}
\end{proposition}
\begin{proof}
We restrict to the case $h \in (X^{\varepsilon^3/20000}, X^{1-\varepsilon^2}]$, showing that in this range the claims hold with $\ll X/h^{-1/2+3\varepsilon/4}$ exceptions. The claims for $h \in (X^{1-\varepsilon^2}, X]$ follow directly from this.

(i) Write $\rho' := (\rho + \rho_\alpha)/2$. We shall compare
\[
S(x, y_j) :=  \frac{1}{y_j}\sum_{\substack{x < p_1 \dotsc p_k m \leq x+y_j \\ p_j \in \mathcal{P}, m \in \mathbb{N}}} f(p_1) \dotsm f(p_k) f(m) =
\frac{1}{y_j} \sum_{\substack{x < p_1 \dotsc p_k m \leq x+y_j \\ p_j \in \mathcal{P}, m \in \mathbb{N}}} f(p_1 \dotsm p_k m ) + O(1/X^{\varepsilon^{10}/2})
\]
for $y_1 := h$ and $y_2 := x/(\log X)^{5\rho'/12}$. The convolution of many short factors will give us lot of flexibility in dealing with Dirihclet polynomials. 

Let us first show that $S(x, y_2)$ corresponds to the main term of the claim with an acceptable error: Moving the $m$ sum inside and applying Lemma~\ref{le:Lipschitz}(iii) (with $X$ replaced by $X/(p_1 \dotsm p_k)$) to it, we see that $S(x, y_2)$ indeed equals
\begin{equation}
\label{eq:S2lowbound}
\begin{split}
& \sum_{p_j \in \mathcal{P}}  f(p_1) \dotsm f(p_k) \Big(\frac{1}{X} \sum_{\substack{m \in \mathbb{N} \\ X < p_1 \dotsc p_k m \leq 2X \\}} f(m) \Big) + O_{\varepsilon, \rho, \theta}\Big(\frac{1}{(\log X)^{\rho'/12}} \prod_{\substack{p \leq X}} \Big(1+\frac{|f(p)|-1}{p}\Big) \Big) \\
&= \frac{1}{X} \sum_{\substack{X < p_1 \ldots p_k m \leq 2X \\ p_j \in \mathcal{P}, m \in \mathbb{N}}} f(p_1 \ldots p_k m) + O_{\varepsilon, \rho, \theta}\Big(\frac{1}{(\log X)^{\rho'/12}} \prod_{\substack{p \leq X}} \Big(1+\frac{|f(p)|-1}{p}\Big) \Big)\Big)
\end{split}
\end{equation}

Let $T_0 = (\log X)^{\rho'/6}$. Arguing as in the beginning of the proof of Proposition~\ref{prop:L1work}, we see that, apart from an acceptable error, $S(x, y_1) - S(x, y_2)$ equals the difference of $j= 1$ and $j=2$ cases of
\begin{equation}
\label{eq:Sy1y2diffFW}
\begin{split}
&\frac{1}{y_j} \cdot \frac{1}{2\pi i}  \sum_{A_1, \dotsc, A_{k}} \int_{\substack{T_0 \leq |t| \leq X/2}} P_{A_1}(1+it) \dotsm P_{A_k}(1+it) M_{X/(A_1 \dotsm A_k)}(1+it) \\
& \qquad \qquad \qquad \cdot \frac{(x + y_j)^{1+it} - x^{1+it}}{1+it} dt
\end{split}
\end{equation}
where $A_i$ traverse through powers of two such that
\[
\frac{X^{\eps^{10}(1-\eps^{20})}}{2} < A_i \leq X^{\eps^{10}(1+\eps^{20})},
\]
\[
P_{A}(s) = \sum_{\substack{A < p \leq 2A \\ p \in \mathcal{P}}} \frac{f(p)}{p^s}, \quad \text{and} \quad M_A(s) = \sum_{\substack{A/2^{k+1} < m \leq 4 A}} \frac{f(m)}{m^s}.
\]

We concentrate on the more difficult case $j = 1$, and consider now $A_1, \dotsc, A_k$ fixed and write
\[
F(s) = P_{A_1}(s) \dotsc P_{A_k}(s) M_{X/(A_1 \dotsm A_k)}(s).
\]
Let
\[
\begin{split}
\mathcal{T}_1 &= \{t \in [-X/2, X/2] \colon |F(1+it)| \leq X^{-1/4+\varepsilon/8} \} \\
\mathcal{T}_2 &= \{t \in [-X/2, X/2] \colon X^{-1/4+\varepsilon/8} < |F(1+it)| \leq h^{-1/2+\varepsilon/2} \} \\
\mathcal{U}_1 &= \{t \in [-X/2, X/2] \colon h^{-1/2+\varepsilon/2} < |F(1+it)| \leq X^{-\eps^{100}}\} \setminus \mathcal{T}_1 \\
\text{and} \quad \mathcal{U}_2 &= \{t \in [-X/2, X/2] \colon |t| \geq T_0, |F(1+it)| \geq X^{-\eps^{100}}\}.
\end{split}
\]
We split the integration range over $t$ in~\eqref{eq:Sy1y2diffFW} into these four sets. By Lemma~\ref{le:Parseval} and~\eqref{eq:(x+y)it-xit}, we see that it suffices to show, for $j = 1, 2$, the $L_1$-bound
\begin{equation}
\label{eq:Thm2L1}
\int_{\mathcal{U}_j} \min\Big\{1, \frac{X/h}{|t|}\Big\} |F(1+it)| dt = O\Big(\frac{1}{(\log X)^{k+\rho/12}} \prod_{p \leq X} \Big(1+\frac{|f(p)|-1}{p}\Big)\Big),
\end{equation}
and the $L_2$ bound
\begin{equation}
\label{eq:Thm2L2}
\max_{X/h < T \leq X/2} \frac{X/h}{T} \int_{\mathcal{T}_j \cap [-T, T]} |F(1+it)|^2 dt = O\Big(h^{-1/2+2\varepsilon/3}\Big).
\end{equation}

Let us first consider the integral over $\mathcal{U}_2$. Estimating $P_{A_3}(1+it), \dotsc, P_{A_k}(1+it)$ trivially and $M_{X/(A_1 \dotsm A_k)}(1+it)$ using Lemma~\ref{le:SparseHalaszComplex}, we see that
\[
\begin{split}
&\int_{\mathcal{U}_2} |F(1+it)| dt = \int_{\mathcal{U}_2} |P_{1,A_1}(1+it) \dotsm P_{k,A_k}(1+it) M_{X/(A_1 \dotsm A_k)}(1+it)| dt \\
&\ll \frac{1}{(\log X)^{k-2+\rho/12}} \prod_{p \leq X} \Big(1+\frac{|f(p)|-1}{p}\Big) \cdot \sum_{t \in \mathcal{W}_2} \Big(|P_{A_1}(1+it)|^2+|P_{A_2}(1+it)|^2\Big),
\end{split}
\]
where $\mathcal{W}_2 \subset \mathcal{U}_2$ is one-spaced. By Lemma~\ref{lem:primeshalasz},
\[
|\mathcal{W}_2| X^{-2\eps^{100}} \ll \sum_{t \in \mathcal{W}_2} |P_{A_1}(1+it)|^2 \leq \Big( \frac{A_1}{\log A_1} + |\mathcal{W}_2| X^{\frac{9}{2} \eta^{3/2}} (\log X)^{2} A_1^{1-\eta/2}\Big) \frac{1}{A_1 \log A_1},
\]
and similarly for $P_{A_2}$. Taking $\eta = \eps^{30}$ we obtain $|\mathcal{W}_2| \ll X^{2 \varepsilon^{100}}$ and re-inserting this into the bound on the right-hand side we conclude that the right-hand side is $\ll (\log X)^{-2}$. 
Hence the contribution from the set $\mathcal{U}_2$ to~\eqref{eq:Thm2L1} is acceptable.

We now turn to $\mathcal{U}_1.$ For 
\begin{equation}
\label{eq:U1betaTrange}
\beta \in [\eps^{100}, \min\{(1/2-\varepsilon/2) \log h / \log X, 1/4-\varepsilon/8\}] \quad \text{and} \quad T \in [X/h, X],
\end{equation}
we write
\begin{equation}
\label{eq:UbTdef}
\mathcal{U}_{\beta, T} = \{ t \in \mathcal{U} \cap [-T, T] \colon |F(1+it)| \in (X^{-\beta}, 2X^{-\beta}]\}.
\end{equation}
Then, splitting~\eqref{eq:Thm2L1} dyadically into $O((\log X)^2)$ integrals, it suffices to show that, for any such $\beta$ and $T$
\[
|\mathcal{W}_{\beta, T}| \ll \frac{T}{X/h} X^{\beta(1-\eps^{6})}
\]
whenever $\mathcal{W}_{\beta, T} \subseteq \mathcal{U}_{\beta, T}$ is one-spaced. Now, for any $t \in \mathcal{W}_{\beta, T}$, among $P_{A_j}$ and $M_{X/(A_1 \dotsm A_k)}$ there exists a Dirichlet polynomial $A(s)$ of length $A \in [X^{\varepsilon^{10}(1-\varepsilon^{20})}/3, 4X^{3\eps^{10}}]$ such that $|A(1+it)| \gg A^{- \beta}$. Since there is a bounded number of options for $A(s)$, by Lemma~\ref{lem:Huxley} applied for $A(s)^\ell$ with $\ell \in \mathbb{N}$, we see that
\[
\begin{split}
|\mathcal{W}_{\beta, T}| &\ll \Big(A^{2 \beta \ell} +\frac{T A^{6 \beta \ell}}{A^{2\ell}}\Big) (\log T)^{O_\eps(1)}
\end{split}
\]
for any positive integer $\ell = O_\eps(1)$. To balance, we choose $\ell \geq 1$ to be the smallest integer such that 
\[
A^{2 \beta \ell} \geq \frac{T A^{6 \beta \ell}}{A^{2\ell}} \iff A^{2\ell\beta(1-2\beta)} \geq T^\beta  \iff   A^{2\ell\beta} \geq T^{\beta/(1-2\beta)},
\]
so that
\begin{equation}
\label{WbTbound}
|\mathcal{W}_{\beta, T}| \ll T^{\beta/(1-2\beta)} A^{2\beta} (\log T)^{O(1)}.
\end{equation}
Hence it suffices to show that
\begin{equation}
\label{eq:U1req}
T^{\beta/(1-2\beta)} A^{2\beta} (\log T)^{O(1)} \ll \frac{T}{X/h} X^{\beta(1-\eps^{6})}
\end{equation}
whenever $\beta$ and $T$ satisfy \eqref{eq:U1betaTrange}. First notice that in this range of $\beta$, we have $\beta/(1-2\beta) \leq 1$, so that it suffices to show the claim for $T = X/h$. In this case
\[
X^{-\beta} \cdot T^{\beta/(1-2\beta)} A^{2\beta} (\log T)^{O(1)} \leq \Big(\frac{X^{2\beta}}{h}\Big)^{\beta/(1-2\beta)} X^{10\beta \varepsilon^{10}}\leq h^{-\varepsilon \beta/(1-2\beta)} X^{10 \beta \varepsilon^{10}} \leq X^{-\beta \eps^6},
\]
since $h \geq X^{\eps^3/20000}$ and $\eps$ is assumed to be small, so~\eqref{eq:U1req} holds. 

Next we consider
\[
\int_{\mathcal{T}_1 \cap [-T, T]} |F(1+it)|^2 dt
\]
for $T \in [X/h, X]$. By definition of $\mathcal{T}_1$, for each $t \in \mathcal{T}_1$, one can find a subproduct $R(1+it)$ of $P_{A_1}(1+it) \dotsm P_{A_k}(1+it) M_{X/(A_1 \dotsm A_k)}(1+it)$ with length $\leq h$ such that $|R(1+it)| \leq h^{-1/4+\varepsilon/4}$. Write
\[
F(s) = R(s) N(s).
\]
Divide the set $\mathcal{T}_1 \cap [-T, T]$ into $O_\eps(1)$ subsets $\mathcal{T}_R$ according to this polynomial $R(s)$. For each $R(s)$ one has
\[
\int_{\mathcal{T}_R \cap [-T, T]} |F(1+it)|^2 dt \leq h^{-1/2+\varepsilon/2} \int_{[-T, T]} |N(1+it)|^2 dt \ll_\eps h^{-1/2+\varepsilon/2} \cdot \frac{T}{X/h}
\]
by the mean value theorem (see~\eqref{eq:contMVT}), so~\eqref{eq:Thm2L2} holds for $j = 1$.

Let us finally consider~\eqref{eq:Thm2L2} for $j = 2$. Similarly to the case of $\mathcal{U}_1$, it suffices to show that 
\[
|\mathcal{W}_{\beta, T}| \ll \frac{T}{X/h} X^{2\beta} h^{-1/2+\eps/2}
\]
whenever 
\[
\beta \in [(1/2-\varepsilon/2) \log h / \log X, 1/4-\varepsilon/8] \quad \text{and} \quad T \in [X/h, X],
\]
$\mathcal{W}_{\beta, T} \subseteq \mathcal{U}_{\beta, T}$ is one-spaced and $\mathcal{U}_{\beta, T}$ is as in~\eqref{eq:UbTdef}.

Using~\eqref{WbTbound}, it suffices to show
\begin{equation}
\label{eq:T2reqment}
T^{\beta/(1-2\beta)} A^{2\beta} (\log T)^{O(1)} \ll \frac{T}{X/h} X^{2\beta} h^{-1/2+\eps/2}
\end{equation}
for these values of $T$ and $\beta$. Again we can 
reduce to the case $T = X/h$. Taking $(1-2\beta)$th power, the claim follows if
\[
X^{10\eps^{10}} \left(\frac{X}{h}\right)^{\beta} \ll X^{2\beta - 4\beta^2} h^{-1/2+\beta + (1-2\beta)\eps/2}.
\]
This follows if
\begin{equation}
\label{eq:betathetareq}
\beta - 4\beta^2 - 10\eps^{10} + (2\beta +(1-2\beta)\eps/2 - 1/2) \frac{\log h}{\log X} \geq 0.
\end{equation}
As a function of $\beta$, this is a quadratic polynomial with negative leading coefficient. Hence it suffices to show the inequality at the endpoints.

In case $\beta = (1/2-\varepsilon/2) \log h / \log X$, we have $X^{2\beta} h^{-1/2 + \varepsilon / 2} = X^\beta$ and the claim~\eqref{eq:T2reqment} follows from \eqref{eq:U1req} which we already showed for this value of $\beta$. On the other hand, it is easy to see that~\eqref{eq:betathetareq} holds for $\beta = 1/4-\eps/8$.

Case (ii) follows similarly except now we do not need to separate the integral with $|t| \leq T_0$ as there is no main term and Lemma~\ref{le:SparseHalaszComplex} wins $\frac{\widehat{M}(f; X)}{\alpha \exp(\widehat{M}(f, X))} + \frac{1}{\alpha (\log X)^\alpha}$ for all $t$.
\end{proof}

We are now ready to prove Theorem \ref{th:LowerBound}.

\begin{proof}[Proof of Theorem~\ref{th:LowerBound}]
We can clearly assume that $h_0$ is large and $\varepsilon$ is small, in particular that $\varepsilon \in (0, \theta)$. 

Let us first consider the case $h_0 \leq X^{\varepsilon^3/20000}$. In this case we shall apply Theorem~\ref{th:ThminS}, so let us fix the parameters defining the set $\mathcal{S}$. We take $Q_1 = h_0$,  $P_1 = h_0^{1-\varepsilon/1000}$, $\nu_2 = \varepsilon/3, \nu_1 = \nu_2(1-\varepsilon/1000)$, and $\eta = \varepsilon/1000$. If $Q_1 \geq \exp((\log X)^{1/2})$ we take $J = 1$. Otherwise we let $\delta'$ to be a small parameter to be fixed later and, for $j = 2, \dotsc, J$, define
\[
P_j = \exp\Big(\Big(\frac{j}{\delta'}\Big)^{8j/\alpha} (\log h_0)^j\Big) \quad \text{and} \quad Q_j = \exp\Big(\Big(\frac{j}{\delta'}\Big)^{(8j+6)/\alpha} (\log h_0)^j\Big),
\]
where $J$ is the largest index such that $Q_J \leq\exp((\log X)^{1/2})$. These choices satisfy the conditions~\eqref{eq:PJQJsize}--\eqref{eq:PjQjnottooclose} once $h_0$ is large enough in terms of $\alpha$, $\delta', \varepsilon$ as we can assume. Now
\[
\sum_{\substack{X < n \leq 2X \\ n \in \mathcal{S}}} f(n) \geq \sum_{P_1 < p_1 \leq Q_1} f(p_1) \sum_{\substack{P_{J+1} < p_{J+1} \leq Q_{J+1} \\ P_{J+2} < p_{J+2} \leq Q_{J+2}}} f(p_{J+1})f(p_{J+2}) \sum_{\substack{\frac{X}{p_1 p_{J+1} p_{J+2}} < n \leq \frac{2 X}{p_1 p_{J+1} p_{J+2}} \\ n \in \mathcal{S}_{1, J+1, J+2} \\ p \mid n \implies p \not \in (P_1, Q_1] \cup (x^{\nu_1}, x^{\nu_2}]}} f(n),
\]
where $\mathcal{S}_{1, J+1, J+2}$ is the set of those $n$ that have at least one prime factor in each of intervals $(P_j, Q_j]$ with $2 \leq j \leq J$. Now we can lower bound the innermost sum over $n$ by
\[
\sum_{\substack{\frac{X}{p_1 p_{J+1} p_{J+2}} < n \leq \frac{2 X}{p_1 p_{J+1}p_{J+2}} \\ p \mid n \implies p \not \in (P_1, Q_1] \cup (X^{\nu_1}, X^{\nu_2}]}} f(n) - \sum_{j=2}^J \sum_{\substack{\frac{X}{p_1 p_{J+1} p_{J+2}} < n \leq \frac{2 X}{p_1 p_{J+1} p_{J+2}} \\ p \mid n \implies p \not \in (P_j, Q_j]}} f(n).
\]
Recalling that $f$ is $(\alpha, X^\theta)$-non-vanishing we can apply Lemma~\ref{le:GrKoMa}, and see that this is at least
\[
\begin{split}
&c_1 \cdot \frac{X}{p_1 p_{J+1} p_{J+2}} \cdot  \prod_{p \leq X} \Big(1+\frac{f(p)-1}{p}\Big) - 20 \frac{X}{p_1 p_{J+1} p_{J+2}} \prod_{p \leq X} \Big(1+\frac{f(p)-1}{p}\Big) \sum_{j=2}^J \Big(\frac{\delta'}{j}\Big)^6 \\
&\gg_{\alpha, \theta, \varepsilon} \frac{X}{p_1 p_{J+1} p_{J+2}} \cdot  \prod_{p \leq X} \Big(1+\frac{f(p)-1}{p}\Big)
\end{split}
\]
once $\delta'$ is small enough in terms of the constant $c_1$ which depended only on $\alpha$ and $\theta$. Hence, using again that $f$ is $(\alpha, X^\theta)$-non-vanishing and also Lemma~\ref{le:GrKoMa}(i), we get
\[
\sum_{\substack{X < n \leq 2X \\ n \in \mathcal{S}}} f(n) \geq c_{\alpha, \theta, \varepsilon} X \prod_{\substack{p \leq X}} \Big(1+\frac{f(p)-1}{p}\Big) \geq \frac{c_{\alpha, \theta, \varepsilon}}{10} \sum_{\substack{X < n \leq 2X}} f(n) 
\]
for some $c_{\alpha, \theta, \varepsilon} > 0 $ depending only on $\alpha$, $\theta$ and $\varepsilon$. Hence in case $h_0 \leq X^{\varepsilon^3/20000}$ the claim follows from Theorem~\ref{th:ThminS} with the above choices and $\delta = c_{\alpha, \theta, \varepsilon}/20$. 

Now in order to deal with $h_0 > X^{\varepsilon^3 / 20000}$ we choose $\mathcal{P}$ as in Proposition \ref{pr:largevalues}. Since $f$ is $(\alpha, X^{\theta})$ non-vanishing,
\[
\begin{split}
\frac{1}{X} \sum_{\substack{p_j \in \mathcal{P}, m \in \mathbb{N} \\ X < p_1 \ldots p_k m \leq 2X}} f(p_1 \ldots p_k m)  &\gg \prod_{p \leq X} \Big ( 1 + \frac{f(p) - 1}{p} \Big ) \sum_{p_j \in \mathcal{P}} \frac{f(p_1) \ldots f(p_k)}{p_1 \ldots p_k} + O(1/X^{\eps^{10}/2}) \\
&\gg_{\alpha, \theta, \varepsilon} \prod_{p \leq X} \Big (1 + \frac{f(p) - 1}{p} \Big ) \gg \frac{1}{X}\sum_{X < n\leq 2X} f(n)
\end{split}
\]
Therefore, for $\delta_0$ small enough in terms of $\alpha, \theta, \eps$, the set of $x \in (X, 2X]$ for which
\[
\frac{1}{h_0 H(f; X)} \sum_{\substack{p_j \in \mathcal{P}, m \in \mathcal{M} \\ x < p_1 \ldots p_k m \leq x + h_0 H(f; X)}} f(p_1 \ldots p_k m) \leq \delta_0 \cdot \frac{1}{X} \sum_{X < n n\leq 2X} f(n)
\]
is of cardinality $\ll X h_0^{-1/2 + \varepsilon}$. The claim follows since
\[
\frac{1}{h_0 H(f; X)} \sum_{x < n \leq x + h_0 H(f; X)} f(n) \gg_\eps \frac{1}{h_0 H(f; X)} \sum_{\substack{p_j \in \mathcal{P}, m \in \mathbb{N} \\ x < p_1 \ldots p_k m \leq x + h_0 H(f; X)}} f(p_1 \ldots p_k m)
\]
\end{proof}

\begin{proof}[Proof of Corollary~\ref{cor:HooleyGen}]
  The first part of Corollary \ref{cor:HooleyGen} follows immediately from Theorem \ref{th:LowerBound} so it is enough to prove the second part. 
Let $f: \mathbb{N} \rightarrow \{0, 1\}$ to be the multiplicative function taking value $1$ if $n \in \mathcal{N}$ and $0$ otherwise. Note that $f$ is $(\alpha, X^\alpha)$-non-vanishing. Now $1 = n_1 < n_2 < \dotsb$ is the sequence of integers such that $f(n_i) = 1$. Then it suffices to show that
    $$
    \sum_{n_i \leq X} (n_{i + 1} - n_i)^{\gamma} \asymp_{\alpha, \gamma} x \Big ( \prod_{\substack{p \leq X \\ f(p) = 0}} \Big ( 1 + \frac{1}{p} \Big ) \Big )^{\gamma - 1}.
    $$

Let us consider first the lower bound which is straight-forward to show. By Lemma~\ref{le:GrKoMa}(iii) we know that $\mathcal{N} \cap (X/2, X] \neq \emptyset$. Hence, by H\"older's inequality and the Shiu bound (Lemma~\ref{le:Shiu}),
\[
\begin{split}
    X &\ll \sum_{n_i \leq X} (n_{i+1} - n_i) \leq \Big ( \sum_{n_i \leq X} (n_{i + 1} - n_i)^{\gamma} \Big )^{1/\gamma} \cdot \Big( \sum_{n \leq X} f(n) \Big)^{(\gamma-1)/\gamma} \\
  &\ll \Big( \sum_{n_i \leq X} (n_{i + 1} - n_i)^{\gamma} \Big )^{1/\gamma} \Big ( X \prod_{\substack{p \leq X \\ f(p) = 0}} \Big (1 - \frac{1}{p} \Big ) \Big )^{(\gamma - 1) / \gamma}.
\end{split}
\]

    Therefore it is enough to prove the upper bound. 
We follow~\cite[Proof of Corollary 6.30]{Opera} with some modifications. Let $\gamma \in [1, 3/2)$. Writing $h_1 = \prod_{\substack{p \leq X \\ f(p) = 0}} \Big(1+\frac{1}{p}\Big)$, our claim is
\[
\sum_{n_i \leq X} (n_{i+1}-n_i)^\gamma \ll_{\alpha, \gamma} X h_1^{\gamma-1}.
\]
By dyadic splitting it suffices to show this with the summation condition replaced by $X < n_i \leq 2X$. We first estimate trivially the contribution of those $n_i$ for which $n_{i+1} - n_i < h_1$. This contribution is
\[
\sum_{\substack{X < n_i \leq 2X \\ n_{i+1}-n_i < h_1}} (n_{i+1}-n_i)^\gamma \ll h_1^{\gamma-1}\sum_{n_i \leq 2X} (n_{i+1}-n_i) \ll Xh_1^{\gamma-1}.
\]
Next we note that by Theorem~\ref{th:LowerBound}, we know that, for any $w \geq h_1$,
\begin{equation}
\label{eq:Thm2Cos}
M(w) := \sum_{\substack{X < n_i \leq 2X \\ n_{i+1}-n_i \geq w}} (n_{i+1}-n_i) \ll \frac{X}{(w/h_1)^{1/2-\eps}}.
\end{equation}
Now
\[
\begin{split}
\sum_{\substack{X < n_i \leq 2X \\ n_{i+1}-n_i \geq h_1}} (n_{i+1}-n_i)^\gamma &= -\int_{h_1}^X w^{\gamma-1} \, d M(w) \\
&= h_1^{\gamma-1} M(h_1) - X^{\gamma-1} M(X) + \int_{h_1}^X M(w) d w^{\gamma-1}
\end{split}
\]
by partial integration and the claimed bound follows from~\eqref{eq:Thm2Cos}.

\end{proof}

\section{Proofs of results for norm-forms}
\label{se:normforms}
Norm-forms pose an additional challenge in that they are not multiplicative. However we are able to reduce ourselves back to a multiplicative setting through Lemma~\ref{lem:normforms} below.

In this section we let, for a number field $K$, $g_K(n)$ be the indicator function of non-negative norm-forms of $K$ and $\Delta_K(n)$ be the multiplicative function such that 
\[
\Delta_{K}(p^{v}) = 
\begin{cases}
1 & \text{if $p^{v} = N_{K / \mathbb{Q}}(\mathfrak{a})$ for $\mathfrak{a}$ an integral ideal of $K$;} \\
0 & \text{otherwise.} 
\end{cases}
\]

\begin{lemma}\label{lem:normforms}
  Let $K$ be a number field over $\mathbb{Q}$. There exist positive constants $\rho = \rho(K)$ and $\alpha = \alpha(K)$, non-negative integers $M = M(K)$, $R = R(K)$ and complex coefficients $c_i = c_i(K) \in \mathbb{C}$ for $i = 0, 1, \ldots, R$, such that
  \begin{equation} \label{eq:normforms}
  g_K(n) = \sum_{0 \leq \ell \leq M} c_{\ell} f_{\ell}(n) + \sum_{M < \ell \leq R} c_{\ell} f_{\ell}(n),
  \end{equation}
  where each $f_i : \mathbb{N} \rightarrow \mathbb{U}$ is multiplicative and
  \begin{enumerate}[(i)]
  \item Each $f_{\ell}$ is $(\alpha, X)$ non-vanishing
  \item For all $\ell = 0,\ldots, R$ we have $|f_\ell(p)| = \Delta_K(p)$.  
  \item For all $\ell = 0, \ldots, M$ and $(p, \text{disc}(K / \mathbb{Q})) = 1$, we have
    $f_{\ell}(p) = \Delta_K(p)$ while
    for $p | \text{disc}(K / \mathbb{Q})$ the value of $f_{\ell}(p)$ is a root of unity.
    In particular each $f_\ell$ with $\ell = 0, \ldots, M$ is almost real-valued.
   \item For each $f_\ell$ with $\ell = M + 1, \ldots, R$ we have
     $$
     \widehat{M}(f_\ell; X) \geq \rho \log\log X
     $$
        \item We have $\sum_{0 \leq \ell \leq M} c_\ell > 0$. 
   \item Let $\varepsilon > 0$. Assume the Riemann Hypothesis for all Hecke $L$-functions and let $W$ be any smooth function  compactly supported in $[1/2, 3]$. Then, for every $\ell = 0, \dotsc, R$,
     $$
     \sum_{n} \frac{f_\ell(n)}{n^{1 + it}} W \Big ( \frac{n}{N} \Big ) \ll_{A, W} \frac{1}{(1 + |t|)^{A}} + O_\varepsilon(N^{-1/2} \cdot (N (1 + |t|))^{\varepsilon}).
     $$
     Consequently also
          $$
     \sum_{n} \frac{g(n)}{n^{1 + it}} W \Big ( \frac{n}{N} \Big ) \ll_{A, W} \frac{1}{(1 + |t|)^{A}} + O_\varepsilon(N^{-1/2} \cdot (N (1 + |t|))^{\varepsilon}).
     $$
\end{enumerate}
\end{lemma}

Before proving Lemma \ref{lem:normforms} we introduce some notation and prove a more technical result (Lemma ~\ref{lem:viko} below).
\begin{definition}
\label{def:p<->B}
Let $K$ be a number field over $\mathbb{Q}$ of degree $k$, and let $\mathcal{C}_1, \ldots, \mathcal{C}_h$ be all the narrow ideal classes in the narrow ideal class group of $K$. For a $k \times h$-matrix $B = (b_{i j})$ consisting of non-negative integers $b_{i j}$, we write $p \leftrightarrow B$ if, for every $i,j$, the number of prime ideals $\mathfrak{p}$ of $K$ such that $\mathfrak{p} | p$, $N_{K / \mathbb{Q}} \mathfrak{p} = p^i$, and $\mathfrak{p} \in \mathcal{C}_j$ is $b_{i j}$. Furthermore, for $p \in \mathbb{P}$, we write $B(p)$ for the (unique) $k \times h$-matrix $B$ for which $p \leftrightarrow B$.
\end{definition}

\begin{lemma} \label{lem:viko}
Let the notation be as in Definition~\ref{def:p<->B}. Suppose that $B$ is such that there exists an unramified prime $p$ with $p \leftrightarrow B$. Then there exist constants $c(B) \in (0, 1]$ depending only on $B$ and $K$ and $\alpha(K) \in (0, 1]$ depending only on $K$ such that the following hold.
  \begin{enumerate}[(i)]
  \item
    For $\Re s > 1$,
    $$
    \sum_{\substack{p \leftrightarrow B}} \frac{1}{p^s} = c(B) \log \zeta(s) + \sum_{\chi} c_{\chi} L(s, \chi) + H(s),
    $$
    where the sum over $\chi$ runs over a finite number of non-trivial primitive Hecke Gr\"o{\ss}encharacters $\chi$, $L(s, \chi)$ is the Hecke L-functions associated to $\chi$, the coefficients $c_{\chi} \in \mathbb{Q}$, and, for any given $\varepsilon > 0$, $H(s)$ is analytic and uniformly bounded in $\Re s > \tfrac 12 + \varepsilon$. 
  \item 
    Let $\varepsilon > 0$ and $N_{X, \varepsilon} := \exp(\log^{2/3 + \varepsilon} X)$. Uniformly in $|t| \leq 4X$
  $$
    \sum_{\substack{N_{X, \varepsilon} < p \leq x \\ p \leftrightarrow B}} \frac{p^{it}}{p} = c(B) \sum_{N_{X, \varepsilon} < p \leq X} \frac{p^{it}}{p} + O_K(1) 
  $$
  \item One has, for all $2 \leq w \leq z$, 
    $$
  \sum_{w < p \leq z} \frac{g_K(p)}{p} = \alpha(K) \sum_{w < p \leq z} \frac{1}{p} + O_{K}\left(\frac{1}{\log w}\right).
  $$
  \item Assuming the Riemann Hypothesis for Hecke $L$-functions, we have, uniformly in $t \in \mathbb{R}$, and $P \geq 1$, and $I \subset [P, 2P]$ any interval, 
    $$
    \sum_{p \in I} \frac{g_K(p)}{p^{1 + it}} \ll \frac{1}{\log P} \cdot \frac{1}{1 + |t|} + P^{-1/2} \log^2 (1 + |t|).  
    $$
  \end{enumerate}
\end{lemma}
\begin{proof}
  Let $\overline{K}$ be the normal hull of $K / \mathbb{Q}$.
  Let $H(\overline{K})$ denote the narrow Hilbert class-field of $\overline{K}$. 
  By \cite[Theorem 4.1]{Odoni75} given $B$ there exists a subset $\mathcal{S}$ of conjugacy classes of $G:=\text{Gal}(H(\overline{K}) / \mathbb{Q})$ such that the condition $p \leftrightarrow B$ is equivalent to requiring that the Frobenius conjugacy class $\sigma_{p}$ of $p$, relative to the Galois extension $H(\overline{K}) / \mathbb{Q}$ belongs to a conjugacy class $\mathcal{C}$ with $\mathcal{C} \in \mathcal{S}$. Note that $\mathcal{S}$ is non-empty because we assume that there is an unramified prime $p$ such that $p \leftrightarrow B$. 

  Let $\widehat{G}$ be the group of characters of $G$. 
  Since $\chi \in \widehat{G}$ are class functions we have $\chi(g) = \chi(u)$ for all $g, u \in \mathcal{C}$.
  In particular,
  \begin{equation} \label{eq:repr}
  - \sum_{\substack{\mathcal{C} \in \mathcal{S}}} \frac{1}{|G|} \sum_{g \in \mathcal{C}} \sum_{\psi \in \widehat{G}} \overline{\psi}(g) \log L(s, \psi, H(\overline{K}) / \mathbb{Q}) = - \sum_{\mathcal{C} \in \mathcal{S}} \frac{|\mathcal{C}|}{|G|} \sum_{\psi \in \widehat{G}} \overline{\psi}(g_{\mathcal{C}}) \log L(s, \psi, H(\overline{K}) / \mathbb{Q}). 
  \end{equation}
  where $g_{\mathcal{C}} \in \mathcal{C}$ is an arbitrary element of $\mathcal{C}$, $\psi$ runs over irreducible characters of $\widehat{G}$ and
  $L(s, \psi, H(\overline{K}) / \mathbb{Q})$ denotes the Artin $L$-function attached to the character $\psi$ (for a reference on Artin $L$-functions, see e.g. \cite[Section 3]{Heilbronn}).
  Recall that
  $$
  \log L(s, \psi, H(\overline{K}) / \mathbb{Q}) := - \sum_{\substack{p^{\ell} \\ \text{unramified}}} \frac{\psi(\sigma_p^{\ell})}{\ell p^{\ell s}}.
  $$
  Therefore by orthogonality of characters the left-hand side of \eqref{eq:repr} is equal to
  \begin{equation} \label{eq:mainobj}
\sum_{\substack{\mathcal{C} \in \mathcal{S}}}  \sum_{\substack{p^{k} : \sigma_p^{k} \in \mathcal{C}}} \frac{1}{k p^{k s}} 
  \end{equation}
  If $\psi$ is the trivial character of $G = \text{Gal}(H(\overline{K}) / \mathbb{Q})$
  then by \cite[§VII.10 Proposition 10.4 (i)]{Neukirch} we have $L(s, \psi, H(\overline{K}) / \mathbb{Q}) = \zeta(s)$.
  On the other hand for non-trivial $\psi$ we have by \cite[Theorem 7]{Heilbronn} 
  that
  $$
  \log L(s, \psi, H(\overline{K}) / \mathbb{Q}) = \sum_{\Omega} \sum_{\substack{\chi \neq \chi_0 \\ \chi \in \widehat{G}_{\Omega}}} c(\psi, \chi, \Omega) \log L(s, \chi, H(\overline{K}) / \Omega) 
  $$
  where the sum over $\Omega$ ranges over sub-fields $\Omega$ of $H(\overline{K})$ such that $H(\overline{K}) / \Omega$ is a cyclic extension, the sum over $\chi$ ranges over non-trivial characters of $\text{Gal}(H(\overline{K}) / \Omega)$ and the coefficients $c(\psi, \chi, \Omega) \in \mathbb{Q}$. By \cite{Brauer} the coefficients $c(\psi, \chi, \Omega)$ can be assumed to be integers, but we will not make use of this fact. 
  
Since the extension $H(\overline{K}) / \Omega$ is cyclic, every non-trivial character $\chi$ of $\text{Gal}(H(\overline{K}) / \Omega)$ is irreducible and injective. Thus by \cite[§VII.10 Theorem 10.6 and the remark following it]{Neukirch} each Artin $L$-function $L(s, \chi, H(\overline{K}) / \Omega)$
 with non-trivial $\chi$ corresponds to $L(s, \widetilde{\chi})$ with $\widetilde{\chi}$ a
  primitive Gr\"o{\ss}encharacter $\pmod{\mathfrak{f}}$ with $\mathfrak{f}$ the conductor of $H(\overline{K}) / \Omega$. This proves the first claim and shows that the coefficient of $\log \zeta(s)$ in (i) is given by
  $$
  c(B) = \sum_{\mathcal{C} \in \mathcal{S}} \frac{|\mathcal{C}|}{|G|} > 0
  $$

  To prove the second claim it follows from \cite{Coleman} (see also \cite{Bartz}) that each $L(s, \widetilde{\chi})$ admits a zero-free region of Vinogradov-Korobov type.
  Using the Hadamard product of each of the completed $L$-functions of $L(s, \widetilde{\chi})$ we see that $\log L(\sigma + it, \widetilde{\chi}) \ll_\varepsilon \log^2 (2 + |t|)$ as long as we stay within $\varepsilon/\log (2 + |t|)$ of the boundary of the zero free-region for any given $\varepsilon > 0$. In particular \eqref{eq:mainobj} admits an analytic continuation to a region of the form $\sigma > 1 - c(\varepsilon) (\log t)^{-2/3 - \varepsilon}$ for all $\varepsilon > 0$ and $c(\varepsilon) > 0$ and is bounded by $\ll_{K} \log^2 (2 + |t|)$ within this region.
Using this, it follows by a standard contour integration argument that for $|t| \leq 4X$ we have 
  $$
  \sum_{\substack{N_{X, \varepsilon} \leq p \leq X \\ p \leftrightarrow B}} \frac{p^{it}}{p} = c(B) \sum_{N_{X, \varepsilon} \leq p \leq X} \frac{p^{it}}{p} + O_K(1)  
  $$
  where $c(B) = |G|^{-1} \sum_{\mathcal{C} \in \mathcal{S}} |\mathcal{C}|$.
  Moreover we find that, for any $2 \leq w \leq z$,
  \begin{equation} \label{eq:above}
  \sum_{\substack{w < p \leq z \\ p \leftrightarrow B}} \frac{1}{p} = c(B) \sum_{w < p \leq z} \frac{1}{p} + O_K\left(\frac{1}{\log w}\right) 
  \end{equation}
  Given an unramified prime $p$ we denote by $R(p)$ the set of ideal classes occupied by integral ideals of norm $p$.
  We note that if $p \leftrightarrow B$ then $R(p)$ is determined.
  Moreover there exists a set $\mathcal{N}$ of narrow ideal classes such that $p$ is a norm form if and only if $R(p) \cap \mathcal{N} \neq \emptyset$. Summing \eqref{eq:above} over all $B$ for
  which $p \leftrightarrow B$ entails $R(p) \cap \mathcal{N} \neq \emptyset$ shows that
  $$
  \sum_{w < p \leq z} \frac{g_K(p)}{p} = \alpha \sum_{w < p \leq z} \frac{1}{p} + O_K\left(\frac{1}{\log w}\right)
  $$
  for some $\alpha \in [0, 1]$. It remains to show that $\alpha \neq 0$.
  Let $T$ denote the narrow Hilbert class field of $K$. By Chebotarev's density theorem for a positive density of primes $p$ there is a
  degree $1$ prime $\mathfrak{B}$ of $T$ above $p \mathbb{Z}$. Then $\mathfrak{B} \cap K$ is
  necessarily principal and $\sigma(\mathfrak{B} \cap K) > 0$ for all embeddings $\sigma$ of $K$ and therefore $p$ is a norm-form of $K$.  

  Finally to establish the last claim it remains to show that conditionally
  $$
  \sum_{\substack{p \in I \\ p \leftrightarrow B}} \frac{1}{p^{1 + it}} \ll \frac{1}{\log P} \cdot \frac{1}{1 + |t|} + P^{-1/2} \log^2 (1 + |t|)
  $$
  for every unramified pattern $B$ and interval $I \subset [P, 2P]$. This follows from the fact that on the Riemann Hypothesis we have the bound $\log L(\sigma + it, \chi) \ll \log X \cdot \log (2 + |t|)$ for all Hecke $L$-functions in the region $\sigma > \tfrac 12 + (\log X)^{-1}$ and $|t| > 10$, say, and a standard contour integration argument which we omit.
\end{proof}

We are now ready to prove Lemma \ref{lem:normforms}.

\begin{proof}[Proof of Lemma \ref{lem:normforms}]
The proof closely follows the original argument of Odoni~\cite{Odoni75}, which we more or less reproduce here. Recall Definition~\ref{def:p<->B}.
The set of possible matrices $B(p)$ is finite and we denote it by $\mathcal{B}$. 

For each $n$ such that $n = N_{K / \mathbb{Q}} \mathfrak{a}$ with $\mathfrak{a}$ an integral ideal of $K$ we define $R(n)$ to be the set of all ideal classes occupied by integral ideals $\mathfrak{a}$ of $K$ with $N_{K / \mathbb{Q}} \mathfrak{a} = n$. We define,
$$
R(n) R(n') := \{ a b : a \in R(n), b \in R(n')\}.
$$
Then as shown by Odoni for $(n,n') = 1$ we have $R(n n') = R(n) R(n')$ and in general $R(n) R(n') \subseteq R(n n')$. Moreover, there exists a set of narrow ideal classes $\mathcal{N}$ such that $n$ is a norm-form if and only if $R(n) \cap \mathcal{N} \neq \emptyset$. 

We now write
$$
n = \prod_{B \in \mathcal{B}} n_{\text{unram}, B}  \prod_{B' \in \mathcal{B}} n_{\text{ram}, B'},
$$
where if $p | n_{\text{unram}, B}$ then $B(p) = B$ and $p$ is unramified in $K$, while if $p | n_{\text{ram}, B'}$ then $B(p) = B'$ and $p$ is ramified in $K$. Write $n_{\text{unram}} := \prod_{B \in \mathcal{B}} n_{\text{unram}, B}$ and $n_{\text{ram}} := \prod_{B' \in \mathcal{B}} n_{\text{ram}, B'}$. Note that all prime divisors of $n_{\text{ram}}$ must also divide $\text{disc}(K / \mathbb{Q})$. Since $n_{\text{unram}}$ and $n_{\text{ram}}$ are co-prime, the set of $n$ that are norm-forms of $K$ correspond to those $n$ for which,
\begin{equation} \label{eq:intersect}
\mathcal{N} \cap \Big ( \prod_{B \in \mathcal{B}} R(n_{\text{unram}, B}) \prod_{B' \in \mathcal{B}} R(n_{\text{ram}, B}) \Big )  \neq \emptyset
\end{equation}

Let $\mathcal{A} = \mathcal{A}(K)$ be the set of all non-empty subsets of $\mathcal{C}$. Following Odoni \cite{Odoni75} we make $\mathcal{A}$ into a semigroup by defining $A B := \{a b : a \in A, b \in B\}$ for $A,B \in \mathcal{A}$. For each $A \in \mathcal{A}$ and for each $B \in \mathcal{B}$ we assign a complex variable $z_{\text{unram}}(A,B)$ and a complex variable $z_{\text{ram}}(A, B)$. We then define a multiplicative function $f_{\mathbf{z}}(n)$ in these complex variables by setting, for unramified primes $p$,
\begin{equation} \label{eq:fdef1}
  f_{\mathbf{z}}(p^{v}) = \begin{cases}
    z_{\text{unram}}(R(p^v), B(p)) & \text{ if } p^{v} = N_{K / \mathbb{Q}} \mathfrak{a} \text{ for an integral ideal } \mathfrak{a} \text{ of } K \\
    0 & \text{ otherwise}.
  \end{cases}
\end{equation}
and by setting for ramified primes $p$,
\begin{equation} \label{eq:fdef2}
  f_{\mathbf{z}}(p^v) = \begin{cases}
    z_{\text{ram}}(R(p^v), B(p)) & \text{ if } p^{v} = N_{K / \mathbb{Q}} \mathfrak{a} \text{ for an integral ideal } \mathfrak{a} \text{ of } K \\
    0 & \text{ otherwise}.
    \end{cases}
\end{equation}

Following Odoni we let $\mathcal{R}$ be the set of tuples $(r_{\text{ram}}(A,B), r_{\text{unram}}(A,B))_{A \in \mathcal{A}, B \in \mathcal{B}}$ of non-negative integers such that,
$$
\mathcal{N} \cap \Big ( \prod_{B \in \mathcal{B}} \prod_{A \in \mathcal{A}} A^{r_{\text{unram}}(A,B)} \prod_{B' \in \mathcal{B}} \prod_{A' \in \mathcal{A'}} (A')^{r_{\text{ram}}(A',B')} \Big )  \neq \emptyset
$$
Then given a tuple $\mathbf{r} = (r_{\text{ram}}(A,B), r_{\text{unram}}(A,B))_{A \in \mathcal{A}, B \in \mathcal{B}}$ and a tuple of complex numbers $\mathbf{z} = (z_{\text{ram}}(A,B), z_{\text{unram}}(A,B))_{A \in \mathcal{A}, B \in \mathcal{B}}$ we define
$$
\mathbf{z}^{\mathbf{r}} = \prod_{B \in \mathcal{B}} \prod_{A \in \mathcal{A}} z_{\text{unram}}(A,B)^{r_{\text{unram}}(A,B)} \prod_{B' \in \mathcal{B}} \prod_{A' \in \mathcal{A'}} z_{\text{ram}}(A',B')^{r_{\text{ram}}(A',B')}
$$
We set $r_{\text{ram}}(A, B) = 0$ (resp. $r_{\text{unram}}(A, B) = 0$) if there exists no ramified prime $p$ (resp. unramified prime $p$) and no $v \geq 1$ such that $B(p) = p$ and $R(p^{v}) = A$.  
Odoni then shows that, for some integers $t_{\text{unram}}(A, B) \geq 1$ and $t_{\text{ram}}(A,B) \geq 1$, 
$$
R(\mathbf{z}) := \sum_{\mathbf{r} \in \mathcal{R}} \mathbf{z}^{\mathbf{r}} = P(\mathbf{z}) \prod_{(A,B) \in S_{\text{unram}}} \frac{1}{1 - z_{\text{unram}}(A,B)^{t_{\text{unram}}(A,B)}} \prod_{(A',B') \in S_{\text{ram}}} \frac{1}{1 - z_{\text{ram}}(A',B')^{t_{\text{ram}}(A',B')}}
$$
where $S_{\text{unram}}$ is the set of those $(A,B)$ for which there exists an unramified prime $p$ and an exponent $v$ such that $B(p) = B$ and $R(p^v) = A$. Similarly $S_{\text{ram}}$ is the set of those $(A,B)$ for which there exists a ramified prime $p$ and an exponent $v$ such that $B(p) = B$ and $R(p^v) = A$. Finally $P(\mathbf{z})$ is a polynomial in the variables $(z_{\text{unram}}(A,B), z_{\text{ram}}(A,B))$ and such that the degree of each $z_{\text{unram}}(A, B)$ (resp. $z_{\text{ram}}(A,B)$) is strictly less than $t_{\text{unram}}(A,B)$ (resp. $t_{\text{ram}}(A,B)$). 

As we noticed before $n$ is a norm form if and only if \eqref{eq:intersect} holds.
Thus,
$$
g_{K}(n) = \oint_{\mathbf{z}} f_{1 / \mathbf{z}}(n) R(\mathbf{z}) d \mu(\mathbf{z})
$$
where $f_{1/\mathbf{z}}$ corresponds to the multiplicative function $f$ as defined in \eqref{eq:fdef1} and \eqref{eq:fdef2} but with each $z_{\text{unram}}(A,B)$ replaced by $z_{\text{unram}}(A,B)^{-1}$ and each $z_{\text{ram}}(A,B)$ replaced by $z_{\text{ram}}(A,B)^{-1}$, and where the measure $d \mu(\mathbf{z})$ corresponds to the product measure of $d z_{\text{ram}}(A,B) / (2\pi i z_{\text{ram}}(A,B))$ and $d z_{\text{unram}}(A,B) / (2\pi i z_{\text{unram}}(A,B))$. Each variable is integrated over the circle $|z| = 1 - \varepsilon$ with $0 < \varepsilon < 1$. We now shift the contour to $|z| \rightarrow \infty$ for each variable, and we collect a contribution of the poles at roots of unity. The remaining integral vanishes identically because the degree of the polynomial $P(\mathbf{z})$ in each variable is less than the degree of the denominator. 
Thus we see that $g_K(n)$ is a linear combination of multiplicative functions $f_{\mathbf{\zeta}}(n)$ evaluated at roots of unity $\mathbf{\zeta}$, say,
\begin{equation} \label{eq:normformspf}
g_K(n) = \sum_{\ell = 0}^{M} c_{\ell} f_{\ell}(n) + \sum_{\ell = M + 1}^{R} c_{\ell} f_{\ell}(n)
\end{equation}
where $f_{i}(n)$ with $i = 0, \ldots, M$ are the multiplicative functions that correspond to setting $z_{\text{unram}}(R(p),B(p)) = 1$ for every unramified prime $p$, while the multiplicative functions $f_{i}$ with $i = M + 1, \ldots, R$ are the multiplicative functions for which for at least one unramified prime $p$ we have $f_{i}(p) \neq 1$. Notice that if $f_i$ is such that $f_i(p) \neq 1$ for at least one unramified prime $p$ then $f_{i}(p) = \zeta$ for some root of unity and we have $f_i(q) = \zeta$ for all unramified primes $q$ such that $B(q) = B(p)$ and $R(q) = R(p)$. Note that $R(p)$ is determined by the condition $p \leftrightarrow B$ so we simply have $f_{i}(p) = \zeta \neq 1$ for all $p \leftrightarrow B$ (in the notation of Lemma \ref{lem:viko}).
By Lemma \ref{lem:viko} and the inequality $\Re f(p) p^{it} - |f(p)| \leq 0$, we obtain for $M < i \leq R$ and $|t| \leq 4X$, 
  $$
\Re \sum_{p \leq X} \frac{f_i(p) p^{it} - |f_i(p)|}{p} \leq  \Re \sum_{\substack{N_{X, \varepsilon} \leq p \leq X \\ p \leftrightarrow B}} \frac{\zeta p^{it} - 1}{p}
\leq c(B) \cdot \Re \sum_{N_{X, \varepsilon} \leq p \leq X} \frac{\zeta p^{it} - 1}{p} + O(1)
  $$
  for some root of unity $\zeta \neq 1$ and where $N_{X, \varepsilon} := \exp(\log^{2/3 + \varepsilon} X)$.
  In particular the above is always $\leq - \rho \log\log X$ for some $\rho > 0$ giving (iv).

  To prove (v) we notice that on square-free $n$ with $(n, \text{disc}(K / \mathbb{Q})) = 1$ we have $f_i(n) = \Delta_K(n)$ for $0 \leq i \leq M$ . Therefore, using Lemma \ref{lem:viko}(i), Lemma~\ref{le:GrKoMa}(iii) and part (iv) together with Lemma~\ref{le:SparseHalaszComplex}(i) it follows that, for $N \geq 1$, 
  \begin{equation} \label{eq:Odo}
  \sum_{\substack{n \leq X \\ (n, N \text{disc}(K / \mathbb{Q}) = 1}} g_K(n)\mu^2(n) = C_{N} \Big ( \sum_{0 \leq i \leq M} c_i \Big ) \cdot X (\log X)^{E(K) - 1} + o ( X (\log X)^{E(K) - 1} )
  \end{equation}
  for some constant $C_N > 0$, and 
  where
  $
  E(K)
  $
  is the Dirichlet density of those primes $p$ for which $\Delta_K(p) = 1$. It remains therefore to show that the left-hand side of \eqref{eq:Odo} is $\gg X (\log X)^{E(K) - 1}$. 
  We adapt the argument of Odoni. By the union bound,
  \begin{equation} \label{eq:ind}
  \Delta_{K}(n) \leq \sum_{i = 1}^{h} \mathbf{1}_{\mathcal{C}_i \in R(n)}
  \end{equation}
  where $\mathcal{C}_h$ denotes the principal class. 
  Pick primes $p_i$ such that $\mathcal{C}_i^{-1} \in R(p_i)$ for all $i=1, \dotsc, h$ \footnote{The infinitude of prime ideals $\mathfrak{p}$ of $K$ with $\mathfrak{p} \in \mathcal{C}_i^{-1}$ follows from applying Chebotarev's density theorem to $H / K$ with $H$ the narrow Hilbert class field of $K$. Note then that $\mathfrak{p} \cap \mathbb{Q}$ is a prime $p$ such that $\mathcal{C}_{i}^{-1} \in R(p)$}.
  Then if $p_i \nmid n$
  and $\mathcal{C}_{i} \in R(n)$ then $\mathcal{C}_h \in R(n p_i)$. As a result summing
  \eqref{eq:ind} over square-free integers co-prime to $\text{disc}(K / \mathbb{Q}) p_1 \dotsm p_h$,
  we find that
  $$
  X (\log X)^{E(K) - 1} \ll \sum_{i = 1}^{h} \Big ( \sum_{\substack{n \leq X \\ (n, p_i) = 1 \\ (n, \text{disc}(K / \mathbb{Q})) = 1}} \mathbf{1}_{\mathcal{C}_i \in R(n)} \Big ) \leq \sum_{i = 1}^{h} 
  \Big ( \sum_{\substack{n \leq X p_i \\ (n, \text{disc}(K / \mathbb{Q})) = 1}} \mathbf{1}_{\mathcal{C}_h \in R(n)} \Big )
  $$
  If $\mathcal{C}_h \in R(n)$ then $n$ is a norm-form. Therefore using \eqref{eq:Odo} we conclude that,
  $$
  X (\log X)^{E(K) - 1} \ll \Big ( \sum_{i = 0}^{h} p_i \Big ) \Big ( \sum_{i = 0}^{h} c_i \Big ) \cdot X (\log X)^{E(K) - 1}
  $$
  and the claim follows.

  In order to prove (vi) we notice that, for each $0 \leq \ell \leq R$, 
  $$
  \sum_{\substack{(n, \text{disc}(K / \mathbb{Q})) = 1}} \frac{f_{\ell}(n) \mu^2(n)}{n^s} = H(s) \prod_{B \in \mathcal{B}_{\ell}} \prod_{\substack{p \leftrightarrow B}} \Big ( 1 + \frac{\zeta_{B}}{p^s} \Big )
  $$
  where 
  $\mathcal{B}_{\ell}$ is some set of admissible $B$'s (depending on $\ell$), $\zeta_B$ is a root of unity depending on $B$ and for any given $\varepsilon > 0$ the function $H(s)$ is analytic and uniformly bounded in $\Re s > \tfrac 12 + \varepsilon$. By Lemma \ref{lem:viko}(i) the above can be factorized as
  \begin{equation} \label{eq:dirsee}
  H(s) \zeta(s)^{\alpha} \prod_{\chi} L(s, \chi)^{c_{\chi}}
  \end{equation}
  for some exponent $\alpha > 0$ and $\chi$ a product over Gr\"o{\ss}encharacters. In particular if we assume the Riemann Hypothesis for Hecke L-functions \eqref{eq:dirsee} is analytic in the region
  $$
  \{ s : \tfrac 12 \leq \Re s \leq 1 \ , \ \Im s \neq 0 \} \cup \{ s : \Re s > 1 \}
  $$
  and for any given $\varepsilon > 0$ and $|t| > 10$ and $\sigma > \tfrac 12 + \varepsilon$ the Dirichlet series \eqref{eq:dirsee} with $s = \sigma + it$ is bounded by $\ll_{\varepsilon} (1 + |t|)^{\varepsilon}$, see \cite[Theorem 5.19]{IwKo04}. Thus (vi) follows by a standard contour integration argument. 

\end{proof}

Now we are ready to show how Theorem~\ref{thm:main2} follows from Theorem~\ref{th:MT}.
\begin{proof}[Proof of Theorem \ref{thm:main2}]
Write $g_K = c_0 f_0 + \dotsc + c_R f_R$ as in Lemma \ref{lem:normforms}. Each $f_\ell$ is an $(\alpha, X)$-non-vanishing multiplicative function for some $\alpha > 0$ depending only on $K$. Let $c' = \sum_{\ell = 1}^R |c_\ell|$.

We can assume that 
\[
\delta > 2c'(R+1) C' \left(\frac{\log \log h_0}{\log h_0}\right)^\alpha + \frac{2c'(R+1)}{(\log X)^{\alpha \rho_\alpha/40}}
\]
for any large constant $C'$ since otherwise the claim is trivial by taking e.g. $\kappa = 80/(\alpha \rho_\alpha)$.

Now, for every $\ell = 0, \ldots, M$, $f_\ell$ is almost real-valued and hence Theorem \ref{th:MT} implies
    \begin{equation} \label{eq:event1}
    \Big | \frac{1}{h_0} \sum_{x < n \leq x + h_0 \delta_{K}(X)^{-1}} f_\ell(n) - \frac{\delta_{K}(X)^{-1}}{X} \sum_{X < n \leq 2X} f_\ell(n) \Big | < \frac{\delta}{2c'(R+1)} 
    \end{equation}
   for all $x \in [X, 2X]$ outside of a set of cardinality 
    \begin{equation}
    \label{eq:NFExcset}
    \ll \frac{X}{h^{c \delta^{\kappa}}}
    \end{equation} 
    for some $c, \kappa$ depending only on $K$. 
    
    Moreover, by Lemmas~\ref{lem:normforms}(iv) and~\ref{le:SparseHalaszComplex}(i) we have, for every $\ell = M + 1, \ldots, R$,
\[
\Big| \frac{1}{X} \sum_{X < n \leq 2X} f_\ell(n) \Big| + \Big| \frac{1}{X} \sum_{X < n \leq 2X} f_\ell(n) n^{-i\widehat{t}_{f_\ell, X}} \Big| \ll \frac{\delta_K(X)}{(\log X)^{\rho/3}}
\]
for some $\rho = \rho(K)$. Hence Theorem \ref{th:MT} implies that~\eqref{eq:event1} holds also for each $\ell = M+1, \dotsc, R$ for every $x \in [X, 2X]$ outside of a set of cardinality~\eqref{eq:NFExcset}. 

It follows that for every $x \in [X, 2X]$ outside of an exceptional set of cardinality~\eqref{eq:NFExcset}, the claim \eqref{eq:event1} holds for every $\ell = 1, \dotsc, R$ simultaneously, and thus, summing over $\ell$ we obtain,
    $$
    \Big | \frac{1}{h} \sum_{x < n \leq x + h \delta_{K}(X)^{-1}} g_K(n) - \frac{\delta_{K}(X)^{-1}}{X} \sum_{X < n \leq 2X} g_K(n) \Big | \leq \delta
    $$
    outside of an exceptional set of cardinality~\eqref{eq:NFExcset} as claimed. 
\end{proof}

  The proof of Theorem~\ref{thm:NormFormLowBound}(i) is rather similar to the proof of Theorem \ref{th:LowerBound} with only minor differences, which we indicate in detail in the proof below.

  \begin{proof}[Proof of (i) of Theorem \ref{thm:NormFormLowBound}]

Write $h := h_0\delta_K(X)^{-1}$. Let us first consider the case $h \leq X^{\varepsilon^3/20000}$. In this case we shall apply Theorem~\ref{th:ThminS}. We let $\mathcal{S}$ be as in proof of Theorem \ref{th:LowerBound}. 

Let us first show that
\begin{equation}
\label{eq:gKclaim}
\frac{1}{X}\sum_{\substack{X < n \leq 2X \\ n \in \mathcal{S} \\ (n, \text{disc}(K / \mathbb{Q})) = 1}} g_K(n) \mu^2(n) \geq \delta_{K, \varepsilon} \cdot \delta_K(X)
\end{equation}
for some $\delta_{K, \varepsilon} > 0$ depending only on $K$ and $\varepsilon$. Since $g_K \geq 0$ and $g_K(mn) \geq g_K(m) g_K(n)$ for all $m, n$, we have
\begin{align*}
& \sum_{\substack{X < n \leq 2X \\ n \in \mathcal{S} \\ (n, \text{disc}(K / \mathbb{Q})) = 1}} g_K(n)  \mu^2(n) \\ & \geq \sum_{P_1 < p_1 \leq Q_1} g_K(p_1) \sum_{\substack{P_{J+1} < p_{J+1} \leq Q_{J+1} \\ P_{J+2} < p_{J+2} \leq Q_{J+2}}}  g_K(p_{J + 1}) g_K(p_{J + 2}) \sum_{\substack{\frac{X}{p_1 p_{J+1} p_{J+2}} < n \leq \frac{2 X}{p_1 p_{J+1} p_{J+2}} \\ n \in \mathcal{S}_{1, J+1, J+2} \\ p \mid n \implies p \not \in (P_1, Q_1] \cup (X^{\nu_1}, X^{\nu_2}] \\ (n, \text{disc}(K / \mathbb{Q})) = 1}}  g_K(n) \mu^2(n),
    \end{align*}
where $\mathcal{S}_{1, J+1, J+2}$ is the set of those $n$ that have at least one prime factor in each of intervals $(P_j, Q_j]$ with $2 \leq j \leq J$. We can lower bound the innermost sum over $n$ by
\[
\sum_{\substack{\frac{X}{p_1 p_{J+1} p_{J+2}} < n \leq \frac{2 X}{p_1 p_{J+1} p_{J+2}} \\ p \mid n \implies p \not \in (P_1, Q_1] \cup (X^{\nu_1}, X^{\nu_2}] \\ (n, \text{disc}(K / \mathbb{Q})) = 1}} g_K(n) \mu^2(n) - \sum_{j=2}^J \sum_{\substack{\frac{X}{p_1 p_{J+1} p_{J+2}} < n \leq \frac{2 X}{p_1 p_{J+1} p_{J+2}} \\ p \mid n \implies p \not \in (P_j, Q_j]}} g_K(n).
\]
Using the upper bound $g_K(n) \leq \Delta_K(n)$ and recalling that $\Delta_K$ is $(\alpha, X)$-non-vanishing we can apply Lemma~\ref{le:GrKoMa}, and see that this is at least
\[
\begin{split}
  & \sum_{\substack{\frac{X}{p_1 p_{J+1} p_{J+2}} < n \leq \frac{2 X}{p_1 p_{J+1} p_{J+2}} \\ p \mid n \implies p \not \in (P_1, Q_1] \cup (X^{\nu_1}, X^{\nu_2}] \\ (n, \text{disc}(K / \mathbb{Q})) = 1}} g_K(n) \mu^2(n)
  - 20 \frac{X}{p_1 p_{J+1} p_{J+2}} \prod_{p \leq X} \Big(1+\frac{\Delta_K(p)-1}{p}\Big) \sum_{j=2}^J \Big(\frac{\delta'}{j}\Big)^6
\end{split}
\]

Recall that by Lemma \ref{lem:normforms}, $g_K = c_0 f_0 + c_1 f_1 + \ldots + c_{R} f_{R}$ and on square-free $n$ co-prime to the discriminant of $K / \mathbb{Q}$ we have $f_\ell(n) = \Delta_K(n)$ for all $0 \leq \ell \leq M$. Using again that $\Delta_K$ is $(\alpha, X)$ non-vanishing we have that 
\begin{align*}
 & \Big ( \sum_{0 \leq \ell \leq M} c_\ell \Big ) \sum_{\substack{\frac{X}{p_1 p_{J+1} p_{J+2}} < n \leq \frac{2 X}{p_1 p_{J+1} p_{J+2}} \\ p \mid n \implies p \not \in (P_1, Q_1] \cup (X^{\nu_1}, X^{\nu_2}] \\ (n, \text{disc}(K / \mathbb{Q}) = 1}} f_\ell(n) \mu^2(n)
  -  \frac{20 \cdot X}{p_1 p_{J+1} p_{J+2}} \prod_{p \leq X} \Big(1+\frac{\Delta_K(p)-1}{p}\Big) \sum_{j=2}^J \Big(\frac{\delta'}{j}\Big)^6
 \\ &\gg_{K} \frac{X}{p_1 p_{J+1} q_{J+1}} \cdot  \prod_{p \leq X} \Big(1+\frac{\Delta_K(p)-1}{p}\Big)
\end{align*}
by Lemma~\ref{le:GrKoMa}(iii) once $\delta'$ is small enough in terms of $\sum_{0 \leq \ell \leq M} c_\ell > 0$ which depends only on $K$. On the other hand Lemmas~\ref{lem:normforms}(iv) and \ref{le:SparseHalaszComplex}(i) imply that, for each $\ell = M + 1 ,M + 2, \ldots, R$,
$$
c_\ell \sum_{\substack{\frac{X}{p_1 p_{J+1} p_{J+2}} < n \leq \frac{2 X}{p_1 p_{J+1} p_{J+2}} \\ p \mid n \implies p \not \in (P_1, Q_1] \cup (X^{\nu_1}, X^{\nu_2}] \\ (n, \text{disc}(K / \mathbb{Q})) = 1}} f_\ell(n) \mu^2(n) \ll \frac{X}{p_1 p_{J + 1} p_{J + 2}} \cdot \frac{1}{(\log X)^{\rho/3}} \prod_{p \leq X} \Big ( 1 + \frac{\Delta_K(p) - 1}{p} \Big )
$$
with $\rho = \rho(K) > 0$. Collecting everything, we obtain
\[
\begin{split}
&\frac{1}{X}\sum_{\substack{X < n \leq 2X \\ n \in \mathcal{S} \\ (n, \text{disc}(K / \mathbb{Q})) = 1}} g_K(n) \mu^2(n) \\
&\gg_K \sum_{P_1 < p_1 \leq Q_1} g_K(p_1) \sum_{\substack{P_{J+1} < p_{J+1} \leq Q_{J+1} \\ P_{J+2} < p_{J+2} \leq Q_{J+2}}}  g_K(p_{J + 1}) g_K(p_{J + 2}) \frac{X}{p_1 p_{J+1} q_{J+1}} \cdot  \prod_{p \leq X} \Big(1+\frac{\Delta_K(p)-1}{p}\Big).
\end{split}
\]
Summing over $p_1, p_{J+1}$ and $p_{J+2}$ using Lemma~\ref{lem:viko}(iii) we conclude that~\eqref{eq:gKclaim} indeed holds for some $\delta_{K, \varepsilon} > 0 $ depending only on $K$ and $\varepsilon$. 

Write $c = \sum_{0 \leq \ell \leq R} |c_\ell|$. By Theorem~\ref{th:ThminS}(iii) we have, for $\ell = 0, \dotsc, M$ and for all $x \in [X, 2X]$, apart from an exceptional set of size $\ll X h_0^{-1/2+\varepsilon}$,
\begin{equation}
\label{eq:ficomp}
\Biggl|\frac{1}{h}\sum_{\substack{x < n \leq x+h \\ n \in \mathcal{S} \\ (n, \text{disc}(K / \mathbb{Q})) = 1}} f_\ell(n)  \mu^2(n) - \frac{1}{X}\sum_{\substack{X < n \leq 2X \\ n \in \mathcal{S} \\ (n, \text{disc}(K / \mathbb{Q})) = 1}} f_\ell(n)  \mu^2(n)\Biggr| < \frac{\delta_{K, \varepsilon}}{2c(R+1)} \cdot \delta_K(X).
\end{equation}

By Lemmas~\ref{lem:normforms}(iv) and~\ref{le:SparseHalaszComplex}(i) together with Lemma~\ref{le:Sinclexcl} we have, for $i = M+1, \dotsc, R$ and for all $x \in [X, 2X]$, apart from an exceptional set of size $\ll X h_0^{-1/2+\varepsilon}$,
\[
\Biggl|\frac{1}{X}\sum_{\substack{X < n \leq 2X \\ n \in \mathcal{S} \\ (n, \text{disc}(K / \mathbb{Q})) = 1}} f_\ell(n)  \mu^2(n)\Biggr| + \Biggl| \frac{1}{X}\sum_{\substack{X < n \leq 2X \\ n \in \mathcal{S} \\ (n, \text{disc}(K / \mathbb{Q})) = 1}} f_\ell(n)\mu^2(n) n^{-i\widehat{t}_{f_\ell, X}} \Biggr| < \frac{\delta_{K, \varepsilon}}{4c(R+1)} \cdot \delta_K(X)
\]
Hence, by Theorem~\ref{th:ThminS}(i), ~\eqref{eq:ficomp} holds also for $i = M+1, \dotsc, R$ for all $x \in [X, 2X]$, apart from an exceptional size of $\ll X h_0^{-1/2+\varepsilon}$. Hence, summing over $\ell$, we obtain that, for all $x \in [X, 2X]$ apart from this acceptable exceptional set, 
\[
\Biggl|\frac{1}{h}\sum_{\substack{x < n \leq x+h \\ n \in \mathcal{S} \\ (n, \text{disc}(K / \mathbb{Q})) = 1}} g_K(n)  \mu^2(n) - \frac{1}{X}\sum_{\substack{X < n \leq 2X \\ n \in \mathcal{S} \\ (n, \text{disc}(K / \mathbb{Q})) = 1}} g_K(n) \mu^2(n) \Biggr| < \frac{\delta_{K, \varepsilon}}{2} \cdot \delta_K(X)
\]
and the claim follows from~\eqref{eq:gKclaim}.

Thanks to Lemma~\ref{lem:normforms}(vi), the conditional claim in case $h \leq X^{\varepsilon^3 / 20000}$ follows similarly from Theorem~\ref{th:ThminS}(ii).

It remains to deal with $h > X^{\varepsilon^3 / 20000}$. To prove the unconditional part of the claim 
we choose $\mathcal{P}$ as in Proposition \ref{pr:largevalues}. Let us first show that
\begin{equation}
\label{eq:gkpjmclaim}
\frac{1}{X} \sum_{\substack{p_j \in \mathcal{P}, m \in \mathbb{N} \\ X < p_1 \ldots p_k m \leq 2X}} g_K(p_1 \ldots p_k m) \geq \delta_{K, \varepsilon} \prod_{p \leq X} \Big ( 1 + \frac{\Delta_K(p) - 1}{p} \Big )
\end{equation}
for some $\delta_{K, \varepsilon} > 0$ depending only on $K$ and $\varepsilon$. 

Note that $g_K(p_1 \ldots p_k m) \geq g_K(p_1) \ldots g_K(p_k) g_K(m)$. Writing $g_K = c_0 f_0 + \dotsb c_R f_R$ as in Lemma \ref{lem:normforms} and applying Lemma \ref{le:GrKoMa}(iii) for $\ell = 0, \dotsc, M$ and Lemma~\ref{le:SparseHalaszComplex} for $\ell = M+1, \dotsc, R$ we find that
\[
\begin{split}
\frac{1}{X} \sum_{\substack{p_j \in \mathcal{P}_j, m \in \mathcal{M} \\ X < p_1 \ldots p_k m \leq 2X}} g_K(p_1) \ldots g_K(p_k) g_K(m)   &\gg \prod_{p \leq X} \Big ( 1 + \frac{\Delta_K(p) - 1}{p} \Big ) \sum_{p_j \in \mathcal{P}_j} \frac{g_K(p_1) \ldots g_K(p_k)}{p_1 \ldots p_k} \\
&\gg_{K, \varepsilon} \prod_{p \leq X} \Big (1 + \frac{\Delta_K(p) - 1}{p} \Big ).
\end{split}
\]
since by Lemma~\ref{lem:viko}(iii) we can control the sums over primes. In particular~\eqref{eq:gkpjmclaim} holds for some $\delta_{K, \varepsilon} > 0$ depending only on $K$ and $\varepsilon$.
 
Writing once again $g_K = c_0 f_0 + \ldots + c_{R} f_{R}$ and applying Proposition \ref{pr:largevalues}(i) for each $f_1, \dotsc, f_{M}$ and Proposition \ref{pr:largevalues}(ii) together with Lemmas~\ref{lem:normforms}(iv) and~\ref{le:SparseHalaszComplex}(i) for each $f_{M+1}, \dotsc, f_R$, we get that
\begin{align*}
\Big | & \frac{1}{h} \sum_{\substack{p_j \in \mathcal{P}_j , m \in \mathcal{M} \\ x < p_1 \ldots p_k m \leq x + h}} g_K(p_1 \ldots p_k m) - \frac{1}{X} \sum_{\substack{p_j \in \mathcal{P}_j, m \in \mathcal{M} \\ X < p_1 \ldots p_k m \leq 2X }} g_K(p_1 \ldots p_k m) \Big | \\ & \leq \frac{\delta_{K, \varepsilon}}{2} \cdot \delta_K(X)
\end{align*}
with at most $\ll_{K, \varepsilon} X h^{-1/2 + \varepsilon}$ exceptions $x \in [X, 2X]$. Since every integer of the form $p_1 \ldots p_k m$ with $p_j \in \mathcal{P}$ has $O_\varepsilon(1)$ representations in such form, we get that, for some $\delta > 0$ depending only on $K, \varepsilon$,
$$
\sum_{x < n \leq x + h} g_K(n) \geq \delta \cdot h_0
$$
with at most $\ll_{K, \varepsilon} X h^{-1/2 + \varepsilon}$ exceptions $x \in [X, 2X]$. 

Let us now concentrate on the conditional part of the claim in the case $h_0 \in (X^{\varepsilon^3 / 20000}, X^{1-\varepsilon^2}]$ --- we shall show that in this range the claim holds with $\ll X h_0^{-1+\varepsilon/2}$ exceptions, and so the claim follows also for $h_0 \geq X^{1-\varepsilon^2}$ . 
We define
$$
S_0(x) := \frac{1}{h_0} \sum_{x < n \leq x + h_0 \delta_K(X)^{-1}} g_K(n) W \Big ( \frac{n}{X} \Big )
$$
where $W$ is a smooth function such that $W(x) = 1$ for $x \in [1,2]$ and $W$ is compactly supported in $[1/2, 3]$. Let also
\begin{align*}
M(x) := \frac{1}{2\pi i h_0} & \int_{-T_1}^{T_1} \Big ( \sum_{m} \frac{g_K(m)}{m^{1 + it}} W \Big ( \frac{m}{X} \Big ) \Big ) \cdot \frac{(x + h_0 \delta_{K}^{-1}(X))^{1 + it} - x^{1 + it}}{1 + i t} dt
\end{align*}
with $T_1 := X^{\varepsilon'}$, and $\varepsilon'$ a small constant depending only on $K$. 

Using Taylor expansion, 
$$
\frac{(x + h_0 \delta_K(X)^{-1})^{1 + it} - x^{1 + it}}{1 + it} = h_0 \delta_K(X)^{-1} \cdot x^{it} + O\left((1+|t|) \frac{(h_0 \delta_K(X)^{-1})^2}{X}\right). 
$$
By Lemma \ref{lem:normforms}(vi) we see that the error term contributes $o(X^{-\varepsilon^2/2})$ to $M(x)$ since $h_0 \leq X^{1-\varepsilon^2}$ and furthermore that the contribution of $|t| \geq T_0 := (\log x)^{\varepsilon'}$ to $M(x)$ is $O_{\varepsilon', A}((\log X)^{-A})$ for any $A \geq 0$. Hence
\[
M(x) = \frac{\delta_K(X)^{-1}}{2\pi i} \int_{-T_0}^{T_0} \Big ( \sum_{m} \frac{g_K(m)}{m^{1 + it}} W \Big ( \frac{m}{X} \Big ) \Big ) x^{it} dt + O_K((\log X)^{-10}),
\]
say.

 Let $F(x) = (1 - |x|)_{+}$ be the triangular function so that $F$ is compactly supported in $[-1, 1]$. Using the decay of the Dirichlet polynomial 
we can replace this 
$$
\frac{\delta_K(X)^{-1}}{2\pi i} \int_{\mathbb{R}} F \Big ( \frac{t}{T_0} \Big ) \Big ( \sum_{m} \frac{g_K(m)}{m^{1 + it}} W \Big ( \frac{m}{X} \Big ) \Big ) \cdot x^{it} dt 
$$
at the price of an error term that is
$$
\ll_{A} \frac{\delta_K(X)^{-1}}{2\pi} \int_{|t| \leq T_0^{1/3}} \frac{|t|}{T_0} \cdot \delta_K(X) dt + T_0^{-A} \ll (\log x)^{-\varepsilon' / 3} .
$$

As a result we see that
$$
M(x) = \delta_K(X)^{-1} \frac{T_0}{2\pi} \sum_{m} \frac{g_K(m)}{m} \widehat{F} \Big ( -\frac{T_0}{2\pi} \log \frac{x}{m} \Big ) W \Big ( \frac{m}{X} \Big ) + O((\log x)^{-\varepsilon' / 3}),
$$
where $\widehat{F}(\xi) = \int_{\mathbb{R}} F(x) e^{-2\pi i x \xi} dx$ is the Fourier transform.
Since $\widehat{F}(x) \geq 0$ for all $x \in \mathbb{R}$ and $\widehat{F}(x) \geq 1/2$ for $|x| \leq 1/100$ we can select as a lower bound an interval $|x - m | \leq x / (1000 T_0)$. This gives
$$
M(x) \gg T_0 \delta_K(X)^{-1} \sum_{|x - m| \leq x / (1000 T_0)} \frac{g_K(m)}{m} \cdot W \Big ( \frac{m}{X} \Big ). 
$$
Once again we use Lemma~\ref{lem:normforms} to write $g_K(n) = c_0 f_0 + \dotsc c_R f_R$. Then we use Lemma~\ref{le:Lipschitz}(i) for each $f_\ell$ (once $\varepsilon'$ is small enough in terms of $K$ it is applicable) and then Lemma~\ref{le:GrKoMa}(iii) for $f_0, \dotsc, f_M$ and Lemma~\ref{le:SparseHalaszComplex}(i) for $f_{M+1}, \dotsc, f_R$. This way we obtain $M(x) \gg 1$ and it remains to show that
$$
\frac{1}{X} \int_{X}^{2X} |S_0(x) - M(x)|^2 dx \ll_{\varepsilon} h_0^{-1 + \varepsilon/3}
$$
for every $\varepsilon > 0$. By Lemma \ref{le:Parseval} this ensues provided that we can show that,
$$
\max_{T \geq X/(h_0 \delta_K(X)^{-1})} \frac{X / (h_0 \delta_K(X)^{-1})}{T} \int_{T_1}^{T} \Big | \sum_{m} \frac{g_K(m)}{m^{1 + it}} W \Big ( \frac{m}{X} \Big ) \Big |^2 dt \ll_{\varepsilon} h_0^{-1 + \varepsilon/3}
$$
for all $\varepsilon > 0$. For $T \geq X$ this follows trivially from the mean value theorem (see \eqref{eq:contMVT}). Since $h_0 > X^{\varepsilon^3 / 20000}$ we notice that for the remaining $T$ this follows from applying the point-wise bound of Lemma \ref{lem:normforms}(vi).

\end{proof}

  \begin{proof}[Proof of (ii) of Theorem \ref{thm:NormFormLowBound}]
    This follows from Theorem \ref{thm:NormFormLowBound}(i) in the same way as Corollary \ref{cor:HooleyGen}(ii) followed from Corollary \ref{cor:HooleyGen}(i). 
  \end{proof}

\appendix
\section{A ``trivial'' inequality}\label{se:appendix}

In this appendix we prove~\eqref{eq:f(p)coslb}, i.e. the following lemma.
\begin{lemma}
\label{le:appendix}
Let $\varepsilon > 0$, let $f \colon \mathbb{N} \to \mathbb{U}$ be an $(\alpha, X^\theta)$-non-vanishing multiplicative function, let $|t| \in [\frac{2}{\theta \log X}, 2X]$, and let
\[
Y := \max\{\exp((\log X)^{2/3+\varepsilon}), \exp(1/|t|)\}.
\]
Then 
\[
\begin{split}
&\sum_{Y < p \leq X^\theta} \frac{|f(p)|}{p}\Big(1-\Big|\cos\Big(\pi \Big\Vert \frac{t \log p}{2\pi} \Big\Vert\Big)\Big|\Big) \\
&\geq \Big ( 2 \int_{0}^{\alpha/2} (1-\cos(\pi x)) dx + O\Big(\frac{1}{\log \log X}\Big) \Big ) \log \frac{\log X^\theta}{\log Y},
\end{split}
\]
where the implied constant depends only on $\theta$ and $\varepsilon$.
\end{lemma}

In the proof of this we will use the following simple auxiliary lemma based on the rearrangement inequality.
\begin{lemma}
\label{le:rearrangement}
Let $N, N_0 \in \mathbb{N}$ with $N \geq N_0$. For $i=1, \dotsc, N$, let $\alpha_i \in [0, 1]$ and $b_i \in \mathbb{R}_{\geq 0}$. Assume that
\begin{equation}
\label{eq:alisumcond}
\sum_{i = 1}^N \alpha_i \geq N_0
\end{equation}
and write $b_i^\ast$ for the sequence $b_i$ rearranged in the increasing order. Then
\[
\sum_{i = 1}^N \alpha_i b_i \geq \sum_{i = 1}^{N_0} b_i^\ast.
\]
\end{lemma}

\begin{proof}
Decreasing some of $\alpha_i$ if necessary, we can 
assume that~\eqref{eq:alisumcond} holds with equality. 

We fix $N_0$ and prove the claim by induction on $N$. In case $N = N_0$ one has $\alpha_i = 1$ for every $i$, and the claim is trivial. Let us now assume that the claim holds for some $N \geq N_0$ and prove it with $N+1$ in place of $N$. By the rearrangement inequality
\[
S:= \sum_{i=1}^{N+1} \alpha_i b_i \geq \sum_{i=1}^{N+1} \alpha_i^\sharp b_i^\ast,
\]
where $\alpha_i^\sharp$ is the sequence $\alpha_i$ in  the decreasing order. 

Since $\sum_{i=1}^{N+1} \alpha_i^\sharp = N_0 \leq N$, one can write $\alpha_{N+1}^\sharp = \sum_{i=1}^N \alpha_i'$ for some $\alpha_i' \in [0,1]$ such that $\alpha_i^\sharp+\alpha_i' \leq 1$ for every $i = 1, \dotsc, N$. Thus
\[
S \geq \alpha_{N+1}^\sharp b_{N+1}^\ast + \sum_{i=1}^{N} \alpha_i^\sharp b_i^\ast =\sum_{i = 1}^N (\alpha_i^\sharp b_i^\ast + \alpha_i' b_{N+1}^\ast) \geq \sum_{i=1}^N (\alpha_i^\sharp + \alpha_i') b_i^\ast \geq \sum_{i=1}^{N_0} b_i^\ast,
\]
where the last step followed from the induction hypothesis.
\end{proof}

\begin{proof}[Proof of Lemma~\ref{le:appendix}]
First we note that by splitting $[Y, X^\theta]$ into shorter intervals, it suffices to prove that for any $x \in [Y, X^\theta]$, one has 
\[
\begin{split}
S_x &:= \sum_{x < p \leq x (\log x)} \frac{|f(p)|}{p}\Big(1-\Big|\cos\Big(\pi \Big\Vert \frac{t \log p}{2\pi} \Big\Vert\Big)\Big|\Big) \\
&\geq \Big ( 2 \int_{0}^{\alpha/2} (1-\cos(\pi x)) dx + O\Big(\frac{1}{\log \log x}\Big) \Big ) \log \frac{\log (x\log x)}{\log x}.
\end{split}
\]
We further split the summation range $(x, x\log x]$ into shorter intervals in order to stabilize the weights and to control the behaviour of $\cos(\pi \Vert \frac{t\log p}{2\pi} \Vert)$. More precisely, we consider intervals $(w_r, w_{r+1}]$ with
\[
w_r := x\left(1+\frac{1}{(\log x)^\nu}\right)^r \quad\text{for $r = 0, \dotsc, R$}
\]
with
\[
R := \left\lfloor\frac{\log \log x}{\log(1+1/(\log x)^\nu)}\right\rfloor - 1 = (\log x)^{\nu}\log \log x + O(\log \log x).
\]
Here $\nu \geq 2$ will be chosen later depending on the size of $|t|$.

We also let $J_r := \sum_{w_r < p \leq w_{r+1}} 1$.
By the prime number theorem in short intervals, recalling that $w_r \in [x, x \log x]$, we have, for each $r = 0, \dotsc, R$,
\[
\frac{J_r}{w_r} = \left(1+O\left(\frac{1}{\log x}\right)\right) \frac{w_{r+1}-w_r}{w_r \log w_r} = \frac{1}{(\log x)^{\nu+1}} \left(1+O\left(\frac{\log \log x}{\log x}\right)\right)
\]
In particular, for every $r = 0, \dotsc, R$
\begin{equation}
\label{eq:JrwrJ0w0}
\frac{J_r}{w_r} = \left(1+O\left(\frac{\log \log x}{\log x}\right)\right)\frac{J_0}{w_0} 
\end{equation}
and
\begin{equation}
\label{eq:JrwrSumRJ0w0}
\begin{split}
(R+1)\frac{J_0}{w_0} &= \left(1+O\left(\frac{\log \log x}{\log x}\right)\right)\sum_{r=0}^R \frac{J_r}{w_r} \\&= \left(1+O\left(\frac{\log \log x}{\log x}\right)\right) \sum_{x < p \leq x \log x} \frac{1}{p} = \frac{\log \log x}{\log x} + O\left(\frac{1}{\log x}\right).
\end{split}
\end{equation}
Furthermore we define $\alpha_r$ by the equality
\begin{equation}
\label{eq:alphardef}
\sum_{w_r < p \leq w_{r+1}} |f(p)| = \alpha_r J_r.
\end{equation}
Note that $\alpha_r \in [0, 1]$ for every $r$. On one hand
\[
\sum_{x < p \leq x \log x} \frac{|f(p)|}{p} \geq 
\alpha \sum_{x < p \leq x \log x} \frac{1}{p} + O\left(\frac{1}{\log x}\right) \geq \alpha(R+1) \frac{J_0}{w_0} + O\left(\frac{1}{\log x}\right)
\]
and on the other hand
\[
\begin{split}
 \sum_{x < p \leq x \log x} \frac{|f(p)|}{p} &\leq \sum_{r=0}^{R} \sum_{w_r < p \leq w_{r+1}} \frac{|f(p)|}{p} + O\left(\frac{J_0}{w_0}\right) \\
 &\leq \sum_{r=0}^R \alpha_r \frac{J_r}{w_r} + O\left(\frac{J_0}{w_0}\right) = \frac{J_0}{w_0} \sum_{r=0}^R \alpha_r + O\left(\left(\frac{\log \log x}{\log x}\right)^2\right).
\end{split}
\]
Hence
\begin{equation}
\label{eq:sumalpr}
\sum_{r=0}^R \alpha_r \geq \left(\alpha+O\left(\frac{1}{\log \log x}\right)\right) (R+1).
\end{equation}
Now
\begin{equation}
\label{eq:appInBound}
\begin{split}
S_x = \sum_{r=0}^R \frac{1}{w_r} \sum_{w_r < p \leq w_{r+1}} |f(p)| \Big(1-\Big|\cos\Big(\pi \Big\Vert \frac{t \log p}{2\pi} \Big\Vert\Big)\Big|\Big) + O\left(\frac{\log \log x}{(\log x)^{\nu+1}}\right).
\end{split}
\end{equation}

We split into two cases according to the size of $|t|$. First consider the case $2/(\theta \log X) \leq |t| \leq (\log X)^{10}$. In this case we choose $\nu = 20$, so that, for any $p \in (w_r, w_{r+1}]$, one has
\[
\frac{t \log p}{2\pi} = \frac{t \log w_r}{2\pi} + O\left(\frac{1}{(\log x)^{4}}\right)
\]
and hence
\[
\begin{split}
S_x &\geq \sum_{r=0}^R \frac{1}{w_r} \sum_{w_r < p \leq w_{r+1}} |f(p)| \Big(1-\Big|\cos\Big(\pi \Big\Vert \frac{t \log w_r}{2\pi} \Big\Vert\Big)\Big|\Big) + O\left(\frac{\log \log x}{(\log x)^{5}}\right).
\end{split}
\]
Recalling also~\eqref{eq:alphardef} and~\eqref{eq:JrwrJ0w0} we get that
\[
\begin{split}
S_x &\geq \sum_{r=0}^R \frac{\alpha_r J_r}{w_r} \Big(1-\Big|\cos\Big(\pi \Big\Vert \frac{t \log w_r}{2\pi} \Big\Vert\Big)\Big|\Big) + O\left(\frac{\log \log x}{(\log x)^{5}}\right) \\
& \geq \frac{J_0}{w_0} \sum_{r=0}^R \alpha_r \Big(1-\Big|\cos\Big(\pi \Big\Vert \frac{t \log w_r}{2\pi} \Big\Vert\Big)\Big|\Big) + O\left(\frac{1}{(\log x)^2}\right).
\end{split}
\]
For $r=0, \dotsc, R$, write
\[
b_r := 1-\Big|\cos\Big(\pi \Big\Vert \frac{t \log w_r}{2\pi} \Big\Vert\Big)\Big|.
\] 
Note that $\log w_r$ are evenly spaced by $\log (1+1/(\log x)^{\nu}) = (\log x)^{-20} + O((\log x)^{-40})$, so that when $b_r^\ast$ is $b_r$ in increasing order, we have
\[
b_r^\ast = 1-\cos\Big(\frac{\pi r}{2R}\Big)+O\left(\frac{1}{(\log x)^5}\right)
\]
Hence, by~\eqref{eq:sumalpr} and Lemma~\ref{le:rearrangement},
\[
S_x \geq \frac{J_0}{w_0} \sum_{r=0}^{\lfloor \alpha (R+1) \rfloor - 1} \Big(1-\cos\Big(\frac{\pi r}{2R}\Big)\Big) + O\left(\frac{1}{\log x}\right) \geq 2R \frac{J_0}{w_0} \int_0^{\alpha/2} \left(1-\cos(\pi x)\right) dx + O\left(\frac{1}{\log x}\right)
\]
and the claim follows from~\eqref{eq:JrwrSumRJ0w0}.

Now consider the case $|t| \geq (\log X)^{10}$. Choose in this case $\nu = 2$. We show that, for any $[\alpha, \beta] \subseteq [0, 1/2]$, one has
\[
\# \left\{p \in (w_r, w_{r+1}] \colon \left\Vert \frac{t \log p}{2\pi} \right\Vert \in [\alpha, \beta] \right\} = J_r \left(2(\beta-\alpha)+ O\left(\frac{1}{(\log x)^2}\right)\right).
\]
To prove this we note that
\[
e\Big(\frac{kt \log p}{2\pi}\Big) = p^{ik t}
\]
and one has, for $x \in [Y, X]$,
\[
\sum_{w_r < p \leq w_{r+1}} p^{-2ikt} = O\Big(\frac{x}{(\log x)^6}\Big)
\]
for any $k \leq (\log x)^3$ by the zero-free region for the Riemann zeta function. Hence the claimed equidistribution follows from the Erd{\H o}s-Tur{\'a}n inequality.

Now, if we write, for $p \in (w_r, w_{r+1}]$, 
\[
b_p = 1-\Big|\cos\Big(\pi \Big\Vert \frac{t \log p}{2\pi} \Big\Vert\Big)\Big|,
\]
then arranging $b_p$ in increasing order, the $j$th element will be
\[
1-\cos\left(\frac{j}{2J_r} \pi\right)+O\left(\frac{1}{(\log x)^2}\right).
\]
Hence, applying Lemma~\ref{le:rearrangement} to~\eqref{eq:appInBound} and recalling~\eqref{eq:JrwrJ0w0}--\eqref{eq:alphardef}, we obtain
\[
\begin{split}
&S_x \geq \sum_{r=0}^R \frac{1}{w_{r}} \sum_{j \leq \alpha_r J_r-1}  \left(1-\cos\left(\frac{j}{2J_r}\pi\right)+O\left(\frac{1}{(\log x)^2}\right)\right) \\
&= 2\sum_{r=0}^R \frac{J_r}{w_r} \int_0^{\alpha_r/2} \left(1-\cos(\pi x)\right) dx + O\left(\frac{\log \log x}{(\log x)^3}\right) \\
&= 2\frac{J_0}{w_0}\sum_{r=0}^R \int_0^{\alpha_r/2} \left(1-\cos(\pi x)\right) dx + O\left(\frac{(\log \log x)^2}{(\log x)^2}\right).
\end{split}
\]
Applying Jensen's inequality to $F(x) = \int_0^x (1-\cos(\pi x)) dx$, we obtain that the previous expression is 
\[
\geq 2\frac{J_0}{w_0}(R+1) \int_0^{\frac{1}{R+1}\sum_{r=0}^R \alpha_r/2} \left(1-\cos(\pi x)\right) dx + O\left(\frac{(\log \log x)^2}{(\log x)^2}\right)
\]
and the claim follows from~\eqref{eq:JrwrSumRJ0w0} and \eqref{eq:sumalpr}.

\end{proof}

\addtocontents{toc}{\setcounter{tocdepth}{-10}}
\bibliographystyle{plain}
\bibliography{../biblio}

\end{document}